\documentclass[a4paper,10pt,reqno]{amsart}
\usepackage{amssymb,amsthm,amsmath,mathrsfs}
\usepackage{mathtools}
\usepackage{bbm}
\usepackage{stackengine}
\usepackage{enumitem}
\usepackage{soul}
\usepackage{placeins}

\usepackage{xurl}
\usepackage[foot]{amsaddr}

\usepackage{caption,color,graphicx}
\usepackage[dvipsnames]{xcolor}
\usepackage{soul}

\usepackage{hyperref}
\usepackage{doi}
\hypersetup{colorlinks  = true,
            urlcolor    = blue,
            linkcolor   = red,
            citecolor   = blue}

\usepackage{verbatim}

\usepackage{environ}
\mathtoolsset{showonlyrefs,showmanualtags}

\usepackage{graphicx}

\usepackage[numbers]{natbib}

\usepackage[centering, top=3cm, bottom=2.5cm]{geometry}

\usepackage{tikz}
\usetikzlibrary{arrows.meta,calc,decorations.pathreplacing,positioning}
\usepackage{float}
\tikzset{
bl base/.style={
    font=\small,
    line width=.6pt,
    >=Stealth,
    color=black
},
bl block/.style={
    draw=black,
    fill=white,
    rounded corners=1.5pt
},
bl muted/.style={
    draw=black,
    fill=white,
    dashed
},
bl tree/.style={
    draw=black,
    line width=.65pt
},
bl node/.style={
    circle,
    draw=black,
    fill=white,
    inner sep=1.6pt
},
bl dot/.style={
    circle,
    fill=black,
    inner sep=1.5pt
},
bl note/.style={
    font=\footnotesize,
    align=center
},
bl formula/.style={
    draw=black,
    fill=white,
    rounded corners=2pt,
    inner sep=8pt,
    align=center
},
bl arrow/.style={
    ->,
    draw=black,
    line width=.8pt
},
bl brace/.style={
    decorate,
    decoration={brace,amplitude=4pt},
    draw=black
}
}

\usepackage{needspace}
\usepackage{centernot}

\newcommand{\Z}{\mathbb{Z}}

\newcommand{\R}{\mathbb{R}}

\newcommand{\E}{\mathbb{E}}
\renewcommand{\P}{\mathbb{P}}

\newcommand{\supp}{\operatorname{supp}\,}

\renewcommand{\d}{\operatorname{d}\!}

\newcommand{\dist}{\operatorname{dist}}

\newcommand{\vertiii}[1]{{\left\vert\kern-0.25ex\left\vert\kern-0.25ex\left\vert #1 \right\vert\kern-0.25ex\right\vert\kern-0.25ex\right\vert}}

\DeclareMathOperator{\Law}{Law}
\DeclareMathOperator{\Lip}{Lip}
\let\oldtocsection=\tocsection
\let\oldtocsubsection=\tocsubsection
\renewcommand{\tocsection}[2]{\hspace{0em}\oldtocsection{#1}{#2}}
\renewcommand{\tocsubsection}[2]{\hspace{1em}\oldtocsubsection{#1}{#2}}

\theoremstyle{plain}
\newtheorem*{theorem*}{Theorem}
\newtheorem{theorem}{Theorem}[section]
\newtheorem{lemma}[theorem]{Lemma}
\newtheorem{proposition}[theorem]{Proposition}
\newtheorem{corollary}[theorem]{Corollary}

\newtheorem{maintheorem}{Theorem}

\theoremstyle{definition}
\newtheorem{definition}[theorem]{Definition}
\newtheorem{remark}[theorem]{Remark}

\newtheorem{notation}[theorem]{Notation}

\newenvironment{stepproof}[1][]{%
  \begin{proof}[#1]%
}{%
  \end{proof}%
}

\newtheorem{manualstep}{\bf Step}
\newenvironment{step}[1][]{%
    \renewcommand\themanualstep{#1}%
    \begin{manualstep}\itshape
}{\end{manualstep}}

\newtheorem{manualhyp}{Hypothesis}

\IfPackageLoadedTF{MnSymbol}{}{}

\providecommand{\T}{\mathbb{T}}
\providecommand{\one}{\mathbf{1}}

\makeatletter
\renewcommand{\part}{%
  \@startsection{part}{0}%
  \z@{\linespacing\@plus\linespacing}{.5\linespacing}%
  {\normalfont\raggedright}}
\makeatother

\makeatletter
\renewcommand{\l@part}{\@tocline{-1}{9pt plus2pt}{0pt}{}{}}
\makeatother

\theoremstyle{definition}
\newtheorem{question}{Question}[section]
\theoremstyle{plain}

\DeclareRobustCommand{\itemref}[1]{\hyperref[#1]{\textup{(\ref*{#1})}}}
\author[Bernat Bassols Cornudella]{Bernat Bassols Cornudella$^{1}$}
\email{\href{mailto:bernat.bassols-cornudella20@imperial.ac.uk}{bernat.bassols-cornudella20@imperial.ac.uk}}

\author[Matheus M.~Castro]{Matheus M. Castro$^{2}$}
\email{\href{mailto:mmcastro@ime.unicamp.br}{mmcastro@ime.unicamp.br}}

\address{$^{1}$Department of Mathematics, Imperial College London, London SW7 2AZ, UK.}
\address{$^{2}$Departamento de Matemática – Universidade Estadual de Campinas (UNICAMP), 13083-859, Campinas, SP, Brazil.}

\title[Finite-order exponential mixing does not imply the CLT]{Finite-order exponential mixing does not imply the central limit theorem}

\date{\today}
\subjclass[2020]{37A25, 37D20, 60F05, 60G10}
\keywords{Exponential decay of correlations, CLT, stationary process, finite-order independence, singular invariant measure, physical measure, hyperbolic coding}

\begin{document}
\begin{abstract}
   We show that finite-order exponential mixing does not imply the central limit theorem for smooth diffeomorphisms. 
   For every integer $R\geq2$, we construct a $\mathcal{C}^{\infty}$ diffeomorphism $T:M\to M$ of a compact smooth manifold $M$ and a $T$-invariant probability measure $\mu$ such that $(T,\mu)$ is exponentially mixing of order $R$ for $\mathcal{C}^1(M)$ observables, but not of order $2R+2$. For some centred observable $f_0\in\mathcal{C}^{\infty}(M)$, no subsequence of the normalised sums $\frac{1}{\sqrt n}\sum_{i=0}^{n-1}f_0\circ T^i$ converges in distribution to a Gaussian law. In particular, exponential decay of two-point correlations implies neither exponential mixing of all orders nor the central limit theorem. 
   We realise $(T,\mu)$ in two ways. In the first, $T$ is a fixed linear Anosov diffeomorphism of $\mathbb T^2$ and $\mu$ is singular with respect to Lebesgue measure. In the second, $M$ can be any compact smooth surface without boundary and $\mu$ is a physical measure whose statistical basin contains a nonempty open set. 
   The main theorems have been formalised in Lean~4.

\end{abstract}
\maketitle

\begingroup
\small
\setcounter{tocdepth}{1}
\tableofcontents
\endgroup

\newpage

\section{Introduction}\label{sec:introduction}

The statistics of a given transformation $T:M \to M$ are encoded by the behaviour of the partial sums
\begin{align*}
S_nf:=\sum_{i=0}^{n-1}f\circ T^i
\end{align*}
for suitable observables $f: M \to \mathbb R$. When $T$ preserves a probability measure $\mu$, the sequence of functions $(f\circ T^i:(M,\mu) \to \mathbb R)_{i\in \mathbb N}$ may be seen as a stationary process and, in particular, the random variables $f\circ T^i$ are identically distributed. 

If $T$ is $\mu$-ergodic and $f\in L^1(M,\mu)$, Birkhoff's ergodic theorem~\cite{Birkhoff1931} yields
\begin{equation}\label{eq:BET}
\frac{S_nf(x)}{n}\xrightarrow[]{n\to\infty}\mu(f):=\int f \d \mu \quad\text{for $\mu$-almost every $x$.}
\end{equation}
This is the dynamical analogue of the strong law of large numbers for independent and identically distributed (i.i.d.) random variables. While \eqref{eq:BET} determines the long-term average behaviour of the systems, this convergence is not effective. Hence, Birkhoff's ergodic theorem provides no information about the fluctuations of $S_n f-n\mu(f)$, their appropriate scaling, or their limiting distribution.

For i.i.d.~random variables with finite variance, the classical Lindeberg--L\'evy central limit theorem (CLT) states that their centred sums, normalised by $\sqrt n$, converge in distribution to a centred normal distribution~\cite{Feller1971}.  The corresponding question in dynamics concerns whether $(S_nf-n\mu(f))/\sqrt n$ also converges to a normal distribution. Replacing $f$ by $f-\mu(f)$, we may assume that $\mu(f)=0$. We say that $f$ satisfies a CLT if
\begin{align}
\frac{S_nf}{\sqrt n} \mathrel{\xRightarrow{n\to\infty}} \mathcal N(0,\sigma_f^2), \label{eq:clt-introduction}
\end{align}
for some $\sigma_f^2\geq0$, where $\Rightarrow$ denotes convergence in distribution with respect to $\mu$. 

The random variables $(f\circ T^i)_{i\in\mathbb N}$ are, in general, far from being independent. Therefore, it is natural to ask which conditions on the dependence of observations $(f\circ T^i)_{i\in\mathbb N}$ ensure the CLT for regular observables. A basic measure of this dependence is the \emph{correlation} $\mu(g\,h\circ T^i)-\mu(g)\mu(h)$ between observations separated by $i$ iterates. On a compact Riemannian manifold $M$, we say that correlations decay at rate $\rho(i)\to0$ for $\mathcal{C}^r$ observables if
$$
\left|\mu(g\,h\circ T^i)-\mu(g)\mu(h)\right| \leq C\rho(i) \|g\|_{\mathcal{C}^r(M)} \|h\|_{\mathcal{C}^r(M)}
$$
for all $g,h\in\mathcal{C}^r(M)$ and $i\geq1$, with $C$ independent of $g$, $h$ and $i$. \emph{Exponential decay of correlations} corresponds to $\rho(i)=e^{-ci}$ for some $c>0$.

A remarkable theorem of Bj\"orklund and Gorodnik~\cite{BjorklundGorodnik} establishes exponential mixing of all orders as a sufficient condition for the CLT. This condition controls joint correlations of arbitrarily many observables. We say that $(T,\mu)$ is \emph{exponentially mixing of all orders} for $\mathcal{C}^1$ observables\footnote{By \cite[Theorem B.1]{DKRH2024}, given $T:M\to M\in \mathcal C^r$, with $r\geq 1$, we have that $T$ is $\mu$-exponentially mixing of order $R$ for $\mathcal C^\beta$ observables for some $\beta\in (0,r)$ if and only if $T$ is $\mu$-exponentially mixing of order $R$ for $\mathcal C^\alpha$ observables for all $\alpha \in (0,r]$. The choice of $\beta=1$ in this paper is just for simplicity.} if, for every integer $s\geq2$, there exist constants $C_s,c_s>0$ such that
\begin{align}
\left| \mu\left(\prod_{i=1}^s f_i\circ T^{n_i}\right) -\prod_{i=1}^s\mu(f_i) \right| \leq C_s e^{-c_s\min_{1\leq i<s}(n_{i+1}-n_i)} \prod_{i=1}^s\|f_i\|_{\mathcal{C}^1(M)} \label{eq:exponential-mixing-all-orders}
\end{align}
for all $f_1,\ldots,f_s\in\mathcal{C}^1(M)$ and integers $0\leq n_1<\cdots<n_s$.  In this setting, if $T$ is a $\mathcal{C}^1$ diffeomorphism, the result of Bj\"orklund and Gorodnik~\cite[Theorem 1.1]{BjorklundGorodnik} implies that every centred real observable $f\in\mathcal{C}^1(M)$ satisfies the CLT \eqref{eq:clt-introduction}.

In this paper, we show that the hypothesis of exponential mixing of all orders in the theorem of Bj\"orklund and Gorodnik is sharp with respect to the number of mixing orders. This result holds even for systems with physical invariant measures. Our main results have been formalised in Lean~4 (see Section~\ref{sec:lean-formalisation}).

For every nonempty compact smooth surface $M$ without boundary and every integer $R\geq2$, we construct a $\mathcal{C}^{\infty}$ diffeomorphism $T:M\to M$ preserving a physical probability measure $\mu$ such that $(T,\mu)$ is exponentially mixing of order $R$ for $\mathcal{C}^1$ observables, while the normalised sums $S_nf/\sqrt n$
have no Gaussian weak subsequential limit for some centred
$f\in\mathcal{C}^{\infty}(M)$. Moreover, the statistical basin of $\mu$ contains a nonempty open set.

The same mixing and limit-law properties can also be realised by a fixed real-analytic linear Anosov diffeomorphism of $\mathbb T^2$. Both constructions fail exponential mixing at order $2R+2$. As a corollary, we obtain, to the best of our knowledge, the first examples in smooth dynamics of exponential mixing without exponential mixing of all orders, including examples with physical measures for $\mathcal{C}^{\infty}$ diffeomorphisms.

To place these results in context, we briefly recall some classical central limit theorems. 
Beyond the i.i.d.~setting, the theorems of Lyapunov \cite{Lyapunov1901} and Lindeberg~\cite{Lindeberg1922} give sufficient conditions for (non-identical) independent random variables to satisfy a CLT after normalisation by the standard deviation of the sum. For dependent observations, central limit theorems hold under suitable moment and dependence assumptions. Early results for Markov chains were obtained by Nagaev~\cite{Nagaev1957}. For stationary mixing processes, classical results include those of Rosenblatt~\cite{Rosenblatt1956}, Ibragimov~\cite{Ibragimov1962} and Doukhan, Massart and Rio~\cite{DMR}.

In dynamics, early central limit theorems were obtained by Fortet~\cite{Fortet1940} and Kac~\cite{Kac1946} for a special class of expanding endomorphisms of the circle. These results already highlighted the importance of regularity of the observable, only covering suitable classes of H\"older continuous functions. Further early results include Sinai's CLT for geodesic flows on manifolds of constant negative curvature \cite{Sinai1960} and Leonov's theorem for suitable classes of observables of ergodic endomorphisms of compact commutative groups~\cite{Leonov1960}.

Exponential decay of correlations and central limit theorems for sufficiently regular observables have been established in several important classes of dynamical systems. These include mixing Axiom~A diffeomorphisms with Gibbs measures~\cite{Bowen1975, GuivarchHardy1988} and classes of piecewise expanding interval maps~\cite{Liverani1995Expanding, Gouezel2015}. For nonuniformly hyperbolic systems, Young~\cite{Young1998} developed a framework in which suitable assumptions on return times yield both exponential decay of correlations and the CLT. For geodesic flows on compact manifolds of negative curvature, central limit theorems were established by Ratner \cite{Ratner1973}, while exponential decay of correlations was obtained by Dolgopyat~\cite{Dolgopyat1998} and Liverani~\cite{Liverani2004}. For surveys of central limit theorems in dynamics and their proof methods, we refer to~\cite{Denker1989,Gouezel2015} and the references therein.

\subsection{Decay of correlations and the central limit theorem}

For Markov chains, quantitative convergence to equilibrium provides classical criteria for the CLT. For example, uniform ergodicity ensures the CLT for every square-integrable observable, without any continuity assumption~\cite{Nagaev1957,Jones2004}. More generally, for stationary processes, sufficiently fast decay of the $\alpha$-mixing coefficients, together with suitable moment assumptions, implies the CLT~\cite{Rosenblatt1956,Ibragimov1962,DMR}. In particular, exponential $\alpha$-mixing suffices when a moment of order $2+\delta$ is finite for some $\delta>0$.

The corresponding question in smooth dynamics is more delicate, since correlation estimates are usually available only for observables with prescribed regularity, such as $\mathcal{C}^1(M)$, and their bounds depend on the corresponding regularity norms. Such estimates do not directly give the uniform control over measurable events required by $\alpha$-mixing criteria. Nevertheless, suitable correlation estimates yield the CLT in several dynamical settings. For ergodic noninvertible maps, a sufficient condition for the CLT is summable decay of correlations for regular observables against bounded measurable observables~\cite{Gordin1969,Liverani1996}. For example, suppose that there exist $C>0$ and $\rho(n)\geq0$, with $\sum_{n\geq1}\rho(n)<\infty$, such that
\begin{align}
    \left|\mu(f\,g\circ T^n)-\mu(f)\mu(g)\right| \leq C\rho(n) \|f\|_{\mathcal{C}^1(M)} \|g\|_{L^\infty(\mu)} \label{eq:regular-bounded-correlations}
\end{align}
for every $f\in\mathcal{C}^1(M)$, $g\in L^\infty(\mu)$, and $n\geq1$. Then every centred $f\in\mathcal{C}^1(M)$ satisfies the CLT through martingale approximation (see~\cite[Theorem~1.1]{Liverani1996}). The estimate \eqref{eq:regular-bounded-correlations} requires no regularity of the second observable and therefore gives stronger control than an estimate involving $\|f\|_{\mathcal{C}^1(M)} \|g\|_{\mathcal{C}^1(M)}$. The estimate \eqref{eq:regular-bounded-correlations} cannot hold for invertible maps unless $\mu$ is a point mass.

Other criteria yield the CLT by controlling correlations between functions of blocks of observations. Under such decorrelation assumptions, Le Borgne and P\`ene \cite{LeBorgnePene2005} and P\`ene~\cite{Pene2004} also establish Berry--Esseen-type rates of convergence in the CLT.
The theorem of Bj\"orklund and Gorodnik~\cite{BjorklundGorodnik} instead gives a sufficient condition in terms of exponential mixing of all orders for regular observables without effective rates of convergence.

Our results show that exponential decay of correlations for $\mathcal{C}^1$ observables, even when extended to exponential mixing of any prescribed finite order, does not suffice to guarantee the CLT. Indeed, for every $R\geq2$, we construct systems that are exponentially mixing of order $R$ for $\mathcal{C}^1$ observables, yet admit centred smooth observables for which the CLT fails.
Thus, no fixed finite number of exponential mixing estimates can replace the all-orders hypothesis in the theorem of Bj\"orklund and Gorodnik.

\subsection{Rokhlin's problem and its quantitative counterparts}
Rokhlin's classical multiple mixing problem~\cite{Rokhlin1949} asks whether mixing implies mixing of all orders. To state the question precisely, for ordered times $0\leq n_1<\cdots<n_s$, write
$$
\Delta(n_1,\ldots,n_s):=\min_{1\leq i<s}(n_{i+1}-n_i).
$$
Here, the order $s$ counts the number of observables or sets involved.

\begin{question}[Rokhlin's multiple mixing problem]
    \label{question:rokhlin}
    Let $T$ be an invertible probability-preserving transformation of a standard probability space. If $ \mu(A\cap T^{-n}B)\xrightarrow[]{n\to\infty}\mu(A)\mu(B) $ for all measurable $A,B$, must
    $$
    \mu\left(\bigcap_{i=1}^sT^{-n_i}A_i\right) \longrightarrow\prod_{i=1}^s\mu(A_i) \quad\text{as }\Delta(n_1,\ldots,n_s)\to\infty
    $$
    hold for every $s\geq3$ and every collection of measurable sets $A_1,\ldots,A_s$?
\end{question}

Rokhlin's problem and its variants have motivated extensive work on multiple mixing (see ~\cite{Kalikow1984,Host1991,Ryzhikov1993, SchmidtWard1993,FayadKanigowski2016,KKU2019, BergelsonZelada,KanigowskiRavotti2024} and the references therein). For an overview and further references, see~\cite{Ryzhikov2024}. Rokhlin's problem was recently answered affirmatively in~\cite[Theorem~1.1 and Corollary~8.1]{OpenAI2026Rokhlin}\footnote{The main theorem was also formally verified in Lean.}, where it has been shown that every mixing probability-preserving transformation is mixing of all finite orders.

In smooth dynamics, exponential mixing has strong qualitative consequences. Dolgopyat, Kanigowski and Rodriguez--Hertz~\cite{DKRH2024} prove that any $\mathcal{C}^{1+\alpha}$ diffeomorphism of a compact manifold preserving a smooth probability measure is Bernoulli whenever it is exponentially mixing. Maldonado~\cite[Theorem~A]{Maldonado2025} further extends this conclusion to invariant SRB measures. Since the Bernoulli property implies mixing of all orders, these results provide a partial answer to Question~\ref{question:rokhlin}. 

A finer quantitative version of Rokhlin's problem asks whether exponential mixing for $\mathcal C^1$ observables implies exponential mixing of all orders for $\mathcal C^1$ observables~\cite{DKRH2024}. Notice that the Bernoulli conclusion alone does not provide these higher-order exponential mixing estimates.

A related question concerns the CLT. Dolgopyat, Dong, Kanigowski and N\'andori~\cite[Appendix~C]{DDKN} ask whether exponential mixing implies the classical CLT for smooth observables, in the setting of smooth systems preserving a smooth probability measure. By the theorem of Bj\"orklund and Gorodnik~\cite{BjorklundGorodnik}, a counterexample to this implication would provide a negative answer to the quantitative (exponential) version of Rokhlin's problem. 

The following question summarises the discussion of the paragraphs above. Figure~\ref{fig:mixing-implications} highlights the implications discussed thus far and the contributions of this paper. 

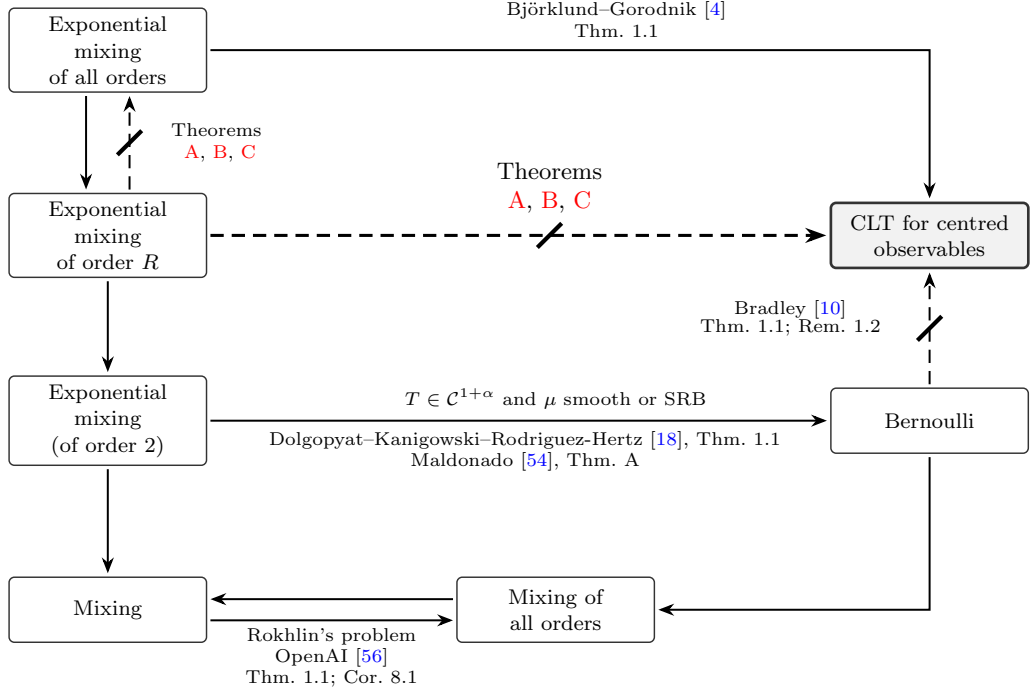
\begin{figure}[!ht]
\centering
\begin{tikzpicture}[
    x=1.15cm,
    y=1cm,
    >=Stealth,
    property/.style={
        draw=black!80,
        line width=.55pt,
        rounded corners=2pt,
        text width=2.40cm,
        minimum height=.86cm,
        align=center,
        inner sep=3pt,
        font=\footnotesize
    },
    endpoint/.style={
        property,
        line width=1pt,
        fill=black!5
    },
    implication/.style={
        -{Stealth[length=2.1mm,width=1.6mm]},
        line width=.75pt,
        shorten >=1.2pt,
        shorten <=1.2pt
    },
    notslash/.style={
        line width=1.6pt,
        black
    },
    nonimplication/.style={
        implication,
        dash pattern=on 5pt off 3pt
    },
    reference/.style={
        font=\scriptsize,
        align=center,
        inner sep=1pt
    }
]

\node[property] (all) at (0,7.55)
    {Exponential mixing\\of all orders};

\node[property] (finite) at (0,5.10)
    {Exponential mixing\\of order $R$};

\node[property] (two) at (0,2.65)
    {Exponential mixing\\(of order $2$)};

\node[property] (mixing) at (0,.15)
    {Mixing};

\node[property] (multiple) at (5.15,.15)
    {Mixing of\\all orders};

\node[property] (bernoulli) at (9.45,2.65)
    {Bernoulli};

\node[endpoint] (clt) at (9.45,5.10)
    {CLT for centred\\observables};

\draw[implication] (finite) -- (two);
\draw[implication] (two) -- (mixing);

\draw[implication]
    ($(all.south)+(-.25,0)$) --
    ($(finite.north)+(-.25,0)$);

\draw[nonimplication]
    ($(finite.north)+(.25,0)$) --
    ($(all.south)+(.25,0)$);

\path
    ($(finite.north)+(.25,0)$) --
    coordinate[pos=.5] (notAll)
    ($(all.south)+(.25,0)$);

\draw[notslash]
    ($(notAll)+(-.14,-.16)$) --
    ($(notAll)+(.14,.16)$);

\node[reference,text width=2.65cm] at (1.25,6.33)
    {Theorems\\~\ref{thm:analytic},
     \ref{thm:physical},~\ref{thm:symbolic}};

\draw[implication]
    ($(multiple.west)+(0,0.14)$) --
    ($(mixing.east)+(0,0.14)$);

\draw[implication]
    ($(mixing.east)+(0,-0.14)$) --
    ($(multiple.west)+(0,-0.14)$);

\node[reference] at (2.575,-0.48)
    {Rokhlin's problem\\OpenAI~\cite{OpenAI2026Rokhlin}\\
     Thm.~1.1; Cor.~8.1};

\draw[implication]
    (all.east) -- (9.45,7.55) -- (clt.north);

\node[reference] at (5.85,7.96)
    {Bj\"orklund--Gorodnik~\cite{BjorklundGorodnik}\\
     Thm.~1.1};

\draw[implication]
    (two.east) -- (bernoulli.west);

\node[reference,text width=4.0cm] at (5.15,2.95)
    {$T\in\mathcal C^{1+\alpha}$ and $\mu$ smooth or SRB};

\node[reference,text width=7.5cm] at (4.8,2.25)
    {Dolgopyat--Kanigowski--Rodriguez-Hertz
     ~\cite{DKRH2024}, Thm.~1.1\\
     Maldonado~\cite{Maldonado2025}, Thm.~A};

\draw[implication]
    (bernoulli.south) --
    (9.45,.15) --
    (multiple.east);

\draw[
    nonimplication,
    line width=1.15pt,
    -{Stealth[length=2.8mm,width=2mm]}
]
    (finite.east) -- (clt.west);

\path (finite.east) --
    coordinate[pos=.55] (notFinite) (clt.west);

\draw[notslash]
    ($(notFinite)+(-.14,-.16)$) --
    ($(notFinite)+(.14,.16)$);

\node[reference,font=\small] at ($(notFinite)+(0,.63)$)
    {Theorems\\
     \ref{thm:analytic},
     \ref{thm:physical},~\ref{thm:symbolic}};

\draw[nonimplication]
    (bernoulli.north) -- (clt.south);

\path (bernoulli.north) --
    coordinate[pos=.5] (notBernoulli) (clt.south);

\draw[notslash]
    ($(notBernoulli)+(-.14,-.16)$) --
    ($(notBernoulli)+(.14,.16)$);

\node[reference] at (7.85,4.12)
    {Bradley~\cite{Bradley2010}};

\node[reference] at (7.85,3.88)
    {Thm.~1.1; Rem.~1.2};

\end{tikzpicture}

\caption{\footnotesize
Implications (solid arrows) and non-implications
(crossed dashed arrows) between mixing properties
and the CLT.}

\label{fig:mixing-implications}
\end{figure}

\begin{question}[{Exponential mixing and the CLT~ \cite{DKRH2024, DDKN}}]
    \label{question:exponential-literature}
    Let $T$ be a smooth diffeomorphism of a compact manifold preserving a smooth probability measure $\mu$. Suppose that $(T, \mu)$ has exponential decay of correlations for smooth observables.
    \begin{enumerate}
        \item\label{it:question-exponential-i} Does $(T, \mu)$ have exponential mixing of all orders for smooth observables?
        \item\label{it:question-exponential-ii} Does every centred smooth observable satisfy the CLT with respect to $\mu$?
    \end{enumerate}
\end{question}

Here, exponential mixing of order $s$ means that \eqref{eq:exponential-mixing-all-orders} holds for this value of $s$, with constants $C_s,c_s>0$ independent of the observables and the observation times.

While Question~\ref{question:exponential-literature} concerns smooth invariant measures, our results show that both implications fail in the broader setting of invariant probability measures, including realisations with physical measures. The case of smooth invariant measures is not resolved by our constructions.

Similarly, we may extend Question~\ref{question:exponential-literature} to the setting of equilibrium states for H\"older continuous potentials:

\begin{question}[Exponential mixing and CLT for equilibrium states]
    \label{question:equilibrium-states}
    Let $T:M\to M$ be a smooth diffeomorphism of a compact smooth manifold, and let $\mu$ be an equilibrium state for a H\"older continuous potential $\varphi:M\to\mathbb{R}$. Suppose that $(T,\mu)$ has exponential decay of correlations for $\mathcal{C}^1(M)$ observables.
    \begin{enumerate}
        \item\label{it:question-equilibrium-i} Is $(T,\mu)$ exponentially mixing of all orders for $\mathcal{C}^1(M)$ observables?
        \item\label{it:question-equilibrium-ii} Does every centred smooth observable satisfy the CLT with respect to $\mu$?
    \end{enumerate}
\end{question}

For $\mathcal{C}^{\infty}$ diffeomorphisms of closed surfaces with positive topological entropy, both conclusions hold for mixing measures of maximal entropy ($\varphi\equiv0$)~\cite{BuzziCrovisierSarig2025}, and more generally for mixing equilibrium states of high-temperature H\"older potentials, i.e.~satisfying $\sup_M\varphi-\inf_M\varphi<h_{\mathrm{top}}(T)$~\cite{BuzziCrovisierSarig2025, LuoYang2026}.

\subsection{Stationary processes without a CLT}
The failure of the CLT for dependent stochastic processes has been studied for many decades. A recurring question is how much independence can be retained while the CLT still fails. Janson constructed stationary pairwise-independent counterexamples to the CLT~\cite{Janson1988}. Bradley obtained a $\beta$-mixing counterexample~\cite{Bradley1989}, which he later showed to be triplewise independent~\cite{Bradley2007}. For each $R\in\mathbb N$, Pruss constructed $R$-tuplewise independent counterexamples in which the CLT fails~\cite{Pruss1998}, and Bradley and Pruss subsequently obtained this property for strictly stationary sequences~\cite{BradleyPruss}.

Failure of the CLT can also coexist with stronger ergodic properties. For each $R\in\mathbb N,$ Nguyen constructed a strictly stationary counterexample that is $R$-tuplewise independent and mixing in the ergodic-theoretic sense~\cite{Nguyen}. Bradley constructed a stationary $\{-1,1\}$-valued process that is five-tuplewise independent and Bernoulli, but whose normalised sums have no Gaussian weak subsequential limit~\cite{Bradley2010}.

These examples show that arbitrary finite orders of independence can coexist with failure of the CLT, even for Bernoulli processes. Our main probabilistic contribution is to add quantitative (exponential) finite-order mixing to this picture. For every $R\geq2$, Theorem~\ref{thm:symbolic} constructs a stationary symmetric $\{-1,1\}$-valued process which is $(2R+1)$-tuplewise independent and exponentially mixing of order $R$, while the CLT fails. In fact, its $\sqrt{n}$-normalised partial sums admit no Gaussian weak limit along any subsequence. Moreover, exponential mixing fails at order $2R+2$. Thus, arbitrarily high finite orders of independence and of exponential mixing are compatible with failure of the CLT.

Theorems~\ref{thm:analytic} and~\ref{thm:physical} provide smooth realisations of the stationary processes constructed in Theorem~\ref{thm:symbolic}. In particular, this implies that exponential mixing of any prescribed order $R$ for $\mathcal{C}^1(M)$ observables does not imply the CLT. Hence the all-orders hypothesis in the theorem of Bj\"orklund and Gorodnik~\cite{BjorklundGorodnik} is sharp with respect to the number of mixing orders.

For every integer $R\geq2$, we construct a smooth diffeomorphism $T:M\to M$ preserving a probability measure $\mu$ such that there are constants $C_R,c_R>0$ satisfying
\begin{align}
    \left| \mu\left(\prod_{i=1}^R f_i\circ T^{n_i}\right) -\prod_{i=1}^R\mu(f_i) \right| \leq C_R e^{-c_R\Delta(n_1,\ldots,n_R)} \prod_{i=1}^R\|f_i\|_{\mathcal{C}^1(M)} \label{eq:intro-multiple}
\end{align}
for all $f_1,\ldots,f_R\in\mathcal{C}^1(M)$ and all integers $0\leq n_1<\cdots<n_R$, while there exists $f_0\in\mathcal{C}^{\infty}(M)$ with $\mu(f_0)=0$ such that
$$
\frac{1}{\sqrt n}\sum_{i=0}^{n-1}f_0\circ T^i \quad\text{has no Gaussian weak subsequential limit},
$$
including degenerate Gaussian laws (Dirac masses). The construction also gives constants ${c,C>0}$ such that, for every sufficiently large integer $N$, there are times $0\leq n_1<\cdots<n_{2R+2}<N$, satisfying
\begin{align}
    \Delta(n_1,\ldots,n_{2R+2})\geq \frac{N}{C} \quad\text{and}\quad \left|\mu\left(\prod_{i=1}^{2R+2} f_0\circ T^{n_i}\right)\right| \geq cN^{-R-1}. \label{eq:intro-higher-obstruction}
\end{align}
This polynomial lower bound rules out exponential mixing of order $2R+2$ for $\mathcal{C}^1(M)$ observables. Thus $(T,\mu)$ is exponentially mixing of order $R$, but not exponentially mixing of all orders.

We realise the dynamical system $(T,\mu)$ described above in two ways.
In Theorem~\ref{thm:analytic}, $T$ is a fixed linear Anosov diffeomorphism of $\mathbb T^2$ and $\mu$ is supported on a fixed two-symbol horseshoe $\Lambda$. The same pair $(T,\Lambda)$ applies to every order of mixing $R$ chosen, whereas the invariant measure $\mu$ depends on $R$.
In Theorem~\ref{thm:physical}, for every order $R$, $T$ is a $\mathcal{C}^{\infty}$ diffeomorphism of any prescribed nonempty compact smooth surface without boundary, and $\mu$ is a physical measure.

Both realisations come from the shift $(\sigma,\nu)$ on $\{-1,1\}^{\mathbb Z}$, with the invariant measure $\nu$ constructed in Theorem~\ref{thm:symbolic}. Under $\nu$, the shift is exponentially mixing of order $R$, but not of order $2R+2$, and the $\sqrt{n}$-normalised Birkhoff sums of the coordinate observable $f_0({\omega}):=\omega_0$ have no Gaussian weak subsequential limit. These properties pass to $(T,\mu)$ through a regular conjugacy $\pi$ between the shift and the restriction of $T$ to the horseshoe, with $\mu=\pi_*\nu$.

In the linear Anosov realisation of Theorem \ref{thm:analytic}, we take $T$ to be a fixed power of the classical cat map and use standard hyperbolic coding \cite{Smale1967,Bowen1970,KatokHasselblatt}. The map, the horseshoe and the smooth observable are independent of the order $R$. For each invariant measure $\mu$, the failure of the CLT persists under small $\mathcal{C}^1$ perturbations of the observable, so the observable can also be chosen to be a real trigonometric polynomial. 

The physical realisation in Theorem~\ref{thm:physical} is based on methods developed by Colli and Vargas \cite{ColliVargas2001} and Kiriki, Li, Nakano, Soma and Vargas \cite{KLNSV} (see also \cite{KirikiSoma2017,KirikiNakanoSoma2024}). In Appendix~\ref{app:smooth-realisation-short}, we adapt these methods to a symmetric affine horseshoe and prove Proposition~\ref{prop:cv-all-measures}, which realises every ergodic invariant measure of the full
two-symbol shift as a physical measure of a $\mathcal{C}^{\infty}$ diffeomorphism of any prescribed compact smooth surface $M$ without boundary.

\subsection{Organisation of the paper}

The paper is divided into two parts. Part~\ref{part:dynamical} concerns the dynamical realisations, i.e.~Theorems~\ref{thm:analytic} and~\ref{thm:physical} and one appendix. Section~\ref{sec:statements} states the main results. Section~\ref{sec:stochastic-consequences} proves both theorems assuming Theorem~\ref{thm:symbolic}. 
The analytic (Theorem~\ref{thm:analytic}) and physical (Theorem~\ref{thm:physical}) realisations then follow from
hyperbolic coding and the smooth physical realisation of
symbolic measures, respectively (Subsections~\ref{sec:realisation} and~\ref{sec:physical}). The latter construction is proved in
Appendix~\ref{app:smooth-realisation-short}.

Part~\ref{part:probabilistic} concerns the proof of Theorem~\ref{thm:symbolic}. In Section~\ref{sec:geometry}, Proposition~\ref{prop:hierarchy} constructs random stationary nested interval partitions of $\mathbb Z$.
Using these partitions, we construct a particular probability measure $\nu$ on $\{-1,1\}^{\mathbb Z}$ in Sections~\ref{sec:block-laws} and \ref{sec:limit}, under which the coordinate process $Z=(Z_i)_{i\in\mathbb Z}$, defined by $Z_i(\omega):=\omega_i$, is stationary (Proposition~\ref{prop:limit}).
In Section~\ref{sec:correlations}, Propositions~\ref{prop:windows} and~\ref{prop:holder} establish independence of sufficiently separated coordinates of $Z$ and exponential mixing of order $R$ for $(\sigma,\nu)$. Finally, Section~\ref{sec:moment} proves failure of the CLT for $Z$ by showing that its normalised partial sums $n^{-1/2}\sum_{i=0}^{n-1}Z_i$ have no Gaussian weak subsequential limit (Corollary~\ref{cor:no-gaussian}). Section~\ref{sec:moment} also proves failure of exponential mixing at order $2R+2$ (Corollary~\ref{cor:higher-order}). Both conclusions follow from the estimate in Proposition~\ref{prop:key-estimate}. Figure~\ref{fig:dependencies} summarises the main proof dependencies. 

Appendix~\ref{app:smooth-realisation-short} proves Proposition~\ref{prop:cv-all-measures}, which realises every ergodic invariant probability measure of the full two-symbol shift as a physical measure of a $\mathcal{C}^{\infty}$ diffeomorphism of any prescribed nonempty compact smooth surface without boundary.
This proposition is used in the proof of Theorem~\ref{thm:physical}.

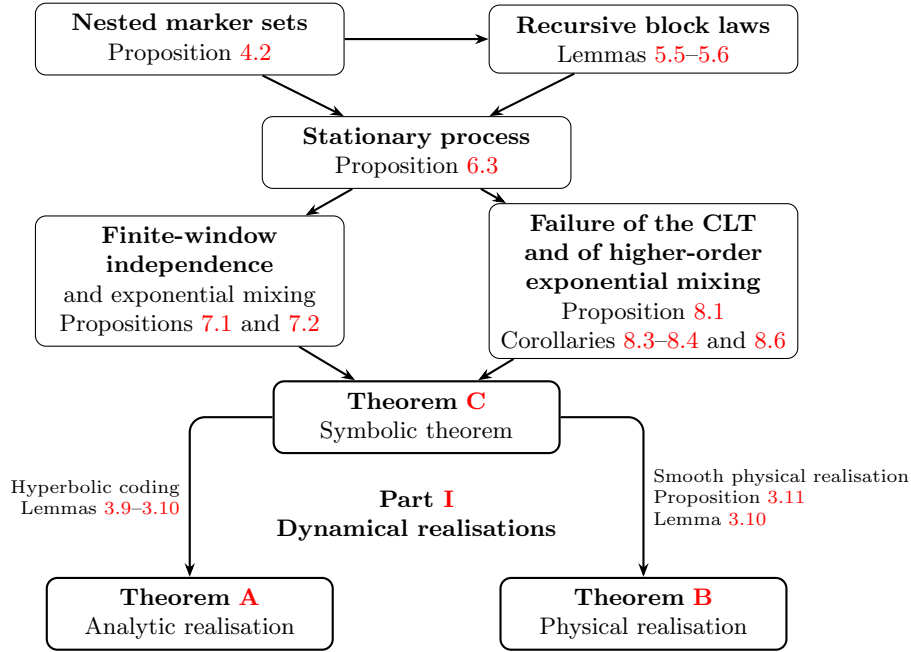
\begin{figure}[htbp]
    \centering
    \begin{tikzpicture}[
        box/.style={
            draw, rounded corners, align=center,
            text width=3.8cm, minimum height=9mm,
            inner sep=4pt, font=\small
        },
        main/.style={box, text width=3.5cm, thick},
        heading/.style={align=center, font=\small\bfseries},
        arrow/.style={
            -{Stealth[length=1.8mm]},
            thick, rounded corners=2mm
        }
        ]
        \node[heading] at (0,1)
        {Part~\ref{part:probabilistic}. Probabilistic construction};
        
        \node[box] (geometry) at (-3,0)
        {\textbf{Nested marker sets}\\
            Proposition~\ref{prop:hierarchy}};
        
        \node[box] (blocks) at (3,0)
        {\textbf{Recursive block laws}\\
            Lemmas~\ref{lem:block-properties}--\ref{lem:block-correlations}};
        
        \node[box] (process) at (0,-1.5)
        {\textbf{Stationary process}\\
            Proposition~\ref{prop:limit}};
        
        \node[box] (mixing) at (-3,-3.2)
        {\textbf{Finite-window}\\
            \textbf{independence}\\
            and exponential mixing\\
            Propositions~\ref{prop:windows} and~\ref{prop:holder}};
        
        \node[box] (obstruction) at (3,-3.2)
        {\textbf{Failure of the CLT}\\
            \textbf{and of higher-order}\\
            \textbf{exponential mixing}\\
            Proposition~\ref{prop:key-estimate}\\
            Corollaries~\ref{cor:moment}--\ref{cor:no-gaussian}
            and~\ref{cor:higher-order}};
        
        \node[main] (symbolic) at (0,-5)
        {\textbf{Theorem~\ref{thm:symbolic}}\\
            Symbolic theorem};
        
        \draw[arrow] (geometry.east) -- (blocks.west);
        \draw[arrow] (geometry) -- (process);
        \draw[arrow] (blocks) -- (process);
        \draw[arrow] (process) -- (mixing);
        \draw[arrow] (process) -- (obstruction);
        \draw[arrow] (mixing) -- (symbolic);
        \draw[arrow] (obstruction) -- (symbolic);
        
        \node[heading, text width=4cm] at (0,-6.3)
        {Part~\ref{part:dynamical}\\
            Dynamical realisations};
        
        \node[main] (analytic) at (-3,-7.6)
        {\textbf{Theorem~\ref{thm:analytic}}\\
            Analytic realisation};
        
        \node[main] (physical) at (3,-7.6)
        {\textbf{Theorem~\ref{thm:physical}}\\
            Physical realisation};
        
        \draw[arrow]
        (symbolic.west) -- (-3,-5) --
        node[midway, left, align=right, font=\scriptsize]
        {Hyperbolic coding\\
            Lemmas~\ref{lem:fixed-horseshoe}--\ref{lem:transfer}}
        (analytic.north);
        
        \draw[arrow]
        (symbolic.east) -- (3,-5) --
        node[midway, right, align=left, font=\scriptsize]
        {Smooth physical realisation\\
    Proposition~\ref{prop:cv-all-measures}\\
    Lemma~\ref{lem:transfer}}
        (physical.north);
    \end{tikzpicture}
    
    \caption{Main proof dependencies.
        Part~\ref{part:probabilistic} proves Theorem~\ref{thm:symbolic}, which Part~\ref{part:dynamical} assumes to obtain Theorems~\ref{thm:analytic} and~\ref{thm:physical}. The dynamical realisations are presented before the probabilistic construction.}
    \label{fig:dependencies}
\end{figure}

\part{Dynamical realisations}
\label{part:dynamical}
\addtocontents{toc}{\protect\addvspace{4pt}}

In this part we prove Theorems~\ref{thm:analytic} and~\ref{thm:physical} using Theorem~\ref{thm:symbolic}, whose proof is given in Part~\ref{part:probabilistic}. We first establish stability of CLT failure under perturbations of the observable, then construct the fixed linear Anosov realisation and the physical realisation, the latter using Proposition~\ref{prop:cv-all-measures}. The proofs of Theorems~\ref{thm:analytic}, \ref{thm:physical} and~\ref{thm:symbolic} have been formalised in Lean~4 (see Section~\ref{sec:lean-formalisation}).

\section{Main results}
\label{sec:statements}
We use the notation $S_nf$ and $\Delta(n_1,\ldots,n_s)$ as in Section \ref{sec:introduction}. On a compact Riemannian manifold $M$, set $ \|f\|_{\mathcal{C}^1(M)} :=\|f\|_\infty+\|Df\|_\infty.$ We recall two notions of convergence  which are used in the statements and proofs below.
\begin{definition}
    For probability measures $\lambda_n$ and $\lambda$ on a metric space $E$, we write $\lambda_n\mathrel{\xRightarrow{n\to\infty}}\lambda$ and say that $\lambda_n$ \emph{converges weakly} to $\lambda$ if
    \begin{align*}
        \int_E u\,\d\lambda_n\longrightarrow\int_E u\,\d\lambda \quad\text{for every }u\in\mathcal{C}_b(E),
    \end{align*}
    where $\mathcal{C}_b(E)$ is the space of bounded continuous real-valued functions on $E$.
    
    For real-valued random variables $Y_n$ on $(M,\mu)$, we write $Y_n\mathrel{\xRightarrow{n\to\infty}}\lambda$ and say that $Y_n$ \emph{converges in distribution} to a probability measure $\lambda$ on $\mathbb R$ if $\Law_\mu(Y_n)\mathrel{\xRightarrow{n\to\infty}}\lambda$. Here $\Law_\mu(Y_n)(A):=\mu\{x\in M:Y_n(x)\in A\}$ for every Borel set $A\subset\mathbb R$. Thus convergence in distribution of random variables means weak convergence of their laws~\cite[Section~3, p.~25]{Billingsley}.
\end{definition}

\begin{definition} 
    Let $X=(X_i)_{i\in\mathbb Z}$ be a real-valued process on a probability space $(\Omega,\mathcal F,\mathbb P)$. For an integer $p\geq1$, the process is \emph{$p$-tuplewise independent} if every collection of at most $p$ coordinates with distinct indices is jointly independent.
\end{definition}

All assertions that an observable is $\{-1,1\}$-valued are understood almost surely with respect to the stated measure. Gaussian laws include point masses. We use $\centernot{\Rightarrow}$ to denote failure of convergence in distribution.

\subsection{Linear Anosov realisation of finite-order exponential mixing without the CLT}

The first result realises finite-order exponential mixing without the CLT on a single real-analytic Anosov diffeomorphism. Both the map and the smooth observable are fixed. Varying the invariant measure produces examples with arbitrarily high prescribed orders of exponential mixing.

\begin{maintheorem}[Linear Anosov realisation]
    \label{thm:analytic}
    Let $A:\mathbb T^2\to\mathbb T^2$ be the cat map
    $$
    A(x,y)=(2x+y,x+y)\pmod1.
    $$
    There exist an integer $k_0\geq2$, a two-symbol horseshoe $\Lambda$ for $T=A^{k_0}$, and a real-valued observable $f_0\in\mathcal{C}^{\infty}(\mathbb T^2)$ such that, for every integer $R\geq2$, there is a $T$-invariant probability measure $\mu$ with the following properties.
    \begin{enumerate}[label=\textup{(\roman*)},ref=\roman*]
        \item\label{it:analytic-i} $\operatorname{supp}\mu=\Lambda$.
        
        \item\label{it:analytic-ii} There are constants $C_R,c_R>0$ such that
        \begin{align}
            \left| \mu\left(\prod_{i=1}^R f_i\circ T^{n_i}\right) -\prod_{i=1}^R\mu(f_i) \right| &\leq C_R e^{-c_R\Delta(n_1,\ldots,n_R)} \prod_{i=1}^R\|f_i\|_{\mathcal{C}^1(\mathbb T^2)} \label{eq:main-multiple}
        \end{align}
        for all $f_1,\ldots,f_R\in\mathcal{C}^1(\mathbb T^2)$ and all integers $0\leq n_1<\cdots<n_R$.
        
        \item\label{it:analytic-iii} Under $\mu$, the stationary process $(f_0\circ T^i)_{i\in\mathbb Z}$ is $\{-1,1\}$-valued, satisfies $\mu(f_0=-1)=\mu(f_0=1)=1/2$, and is $(2R+1)$-tuplewise independent. In particular $\mu(f_0)=0.$

        \item\label{it:analytic-iv} The laws of $S_nf_0/\sqrt n$ under $\mu$ have no Gaussian weak subsequential limit; i.e. for every strictly increasing sequence of positive integers $n_k\to\infty$ and every $\sigma^2\geq0$,
        \begin{align*}
            \frac{1}{\sqrt{n_k}} \sum_{i=0}^{n_k-1}f_0\circ T^i \centernot{\Longrightarrow}\mathcal N(0,\sigma^2) \quad\text{as }k\to\infty,
        \end{align*}
        where convergence in distribution is taken with respect to $\mu$.

        \item\label{it:analytic-v} There is an open neighbourhood $\mathcal O$ of $f_0$ in $\mathcal{C}^1(\mathbb T^2)$ such that, for every $g\in\mathcal O$, the laws of $S_n(g-\mu(g))/\sqrt n$ under $\mu$ have no Gaussian weak subsequential limit.
        In particular, such a $g$ can be chosen to be a real trigonometric polynomial.
        
        \item\label{it:analytic-vi}
There are constants $c,C>0$ and $N_0\geq1$ such that, for every integer $N\geq N_0$, there are integers $0\leq n_1<\cdots<n_{2R+2}<N$ satisfying
\begin{align}
\Delta(n_1,\ldots,n_{2R+2})\geq\frac{N}{C} \quad\text{and}\quad \left|\mu\left(\prod_{i=1}^{2R+2} f_0\circ T^{n_i}\right)\right| \geq cN^{-R-1}. \label{eq:main-higher-order}
\end{align}
Hence, $(T,\mu)$ is not exponentially mixing of order $2R+2$ in $\mathcal{C}^1(\mathbb T^2)$.
    \end{enumerate}
\end{maintheorem}

The proof of Theorem \ref{thm:analytic} is given in Subsection~\ref{sec:realisation}, which transports the process of Theorem~\ref{thm:symbolic} to a horseshoe using the Birkhoff--Smale theorem and H\"older symbolic coding~\cite[Sections~6.5 and~18.7]{KatokHasselblatt}. The coding transfers the correlation estimates, while a smooth function separating the two symbol cylinders realises the coordinate process exactly. Proposition~\ref{prop:robustness} then gives stability under perturbations of the observable and allows approximation by a real trigonometric polynomial.

\subsection{Physical realisation of finite-order exponential mixing without the CLT}

Fix a smooth volume $m$ on $M$. The \emph{statistical basin} of an invariant probability measure $\mu$ is
\begin{align}
    B_T(\mu) :=\left\{x\in M: \frac1n\sum_{i=0}^{n-1}f\circ T^i(x) \xrightarrow[]{n\to\infty} \int f\d\mu,\ \text{for every }f\in\mathcal C^0(M)\right\}. \label{eq:physical-basin-definition}
\end{align}
The measure is \emph{physical} if $m(B_T(\mu))>0$. This definition is independent of the choice of smooth volume, since any two smooth positive densities are mutually absolutely continuous.

The next theorem shows that finite-order exponential mixing without the CLT can also occur for a physical measure. The construction applies to every nonempty compact smooth surface without boundary and yields a $\mathcal{C}^{\infty}$ diffeomorphism. Moreover, the statistical basin contains a nonempty open set on which the time averages of every continuous observable converge uniformly to its integral against $\mu$. Thus $\mu$ describes the long-term averages of an open set of initial conditions, while the normalised sums of a smooth observable have no Gaussian weak subsequential limit with respect to $\mu$.

\begin{maintheorem}[Physical realisation]
    \label{thm:physical}
    Let $M$ be a nonempty compact smooth surface without boundary. For every integer $R\geq2$, there exist a $\mathcal{C}^{\infty}$ diffeomorphism $T:M\to M$, a two-symbol horseshoe $\Lambda$, a $T$-invariant probability measure $\mu$, and a real-valued observable $f_0\in\mathcal{C}^{\infty}(M)$ with the following properties.
    \begin{enumerate}[label=\textup{(\roman*)},ref=\roman*]
        \item\label{it:physical-i} $\operatorname{supp}\mu=\Lambda$.
        
        \item\label{it:physical-ii} There are constants $C_R,c_R>0$ such that
        \begin{align}
            \left| \mu\left(\prod_{i=1}^R f_i\circ T^{n_i}\right) -\prod_{i=1}^R\mu(f_i) \right| &\leq C_R e^{-c_R\Delta(n_1,\ldots,n_R)} \prod_{i=1}^R\|f_i\|_{\mathcal{C}^1(M)} \label{eq:physical-multiple}
        \end{align}
        for all $f_1,\ldots,f_R\in\mathcal{C}^1(M)$ and all integers $0\leq n_1<\cdots<n_R$.
        
        \item\label{it:physical-iii} Under $\mu$, the stationary process $(f_0\circ T^i)_{i\in\mathbb Z}$ is $\{-1,1\}$-valued, satisfies $\mu(f_0=-1)=\mu(f_0=1)=1/2$, and is $(2R+1)$-tuplewise independent.

        Moreover, the laws of $S_nf_0/\sqrt n$ under $\mu$ have no Gaussian weak subsequential limit; i.e. for every strictly increasing sequence of positive integers $n_k\to\infty$ and every $\sigma^2\geq0$,
        \begin{align*}
            \frac{1}{\sqrt{n_k}} \sum_{i=0}^{n_k-1}f_0\circ T^i \centernot{\Longrightarrow}\mathcal N(0,\sigma^2) \quad\text{as }k\to\infty,
        \end{align*}
        where convergence in distribution is taken with respect to $\mu$.

        \item\label{it:physical-iv} There is an open neighbourhood $\mathcal O$ of $f_0$ in $\mathcal{C}^1(M)$ such that, for every $g\in\mathcal O$, the laws of $S_n(g-\mu(g))/\sqrt n$ under $\mu$ have no Gaussian weak subsequential limit.
        
        \item\label{it:physical-v} There is a nonempty open set $W\subset M$ such that, for every $g\in\mathcal{C}^0(M)$,
        \begin{align}
            \lim_{n\to\infty}\sup_{y\in W} \left|\frac1nS_ng(y)-\mu(g)\right|=0. \label{eq:physical-uniform-basin}
        \end{align}
        Thus $W\subset B_T(\mu)$, and hence $\mu$ is physical.
        
        \item\label{it:physical-vi} The observable $f_0$ satisfies \eqref{eq:main-higher-order}, with constants depending on this realisation.
        Hence, $(T,\mu)$ is not exponentially mixing of order $2R+2$ in $\mathcal{C}^1(M)$.
    \end{enumerate}
\end{maintheorem}

The proof in Subsection~\ref{sec:physical} combines Theorem~\ref{thm:symbolic} with Proposition~\ref{prop:cv-all-measures}.
The proposition gives the uniform convergence of time averages in \eqref{eq:physical-uniform-basin} and shows that the invariant measure is physical. Lemma~\ref{lem:transfer} transfers the statistical properties through the horseshoe conjugacy.

The preceding theorems show that exponential mixing of any prescribed order $R$ can coexist with failure of exponential mixing at a higher order. In particular, exponential decay of correlations does not, in general, imply exponential mixing of all orders.

\begin{corollary}
    \label{cor:quantitative-rokhlin}
    There exists a real-analytic Anosov diffeomorphism $T:\mathbb T^2\to\mathbb T^2$ such that, for every integer $R\geq2$, there is a $T$-invariant probability measure $\mu_R$ for which $(T,\mu_R)$ is exponentially mixing of order $R$, but not of order $2R+2$, for $\mathcal{C}^1(\mathbb T^2)$ observables.
    
    For every nonempty compact smooth surface $M$ without boundary and every integer $R\geq2$, there also exists a $\mathcal{C}^{\infty}$ diffeomorphism of $M$ preserving a physical measure with the same mixing properties for $\mathcal{C}^1(M)$ observables.
\end{corollary}
\begin{proof}
    Apply items~\itemref{it:analytic-ii} and~\itemref{it:analytic-vi} of Theorem~\ref{thm:analytic}, and items~\itemref{it:physical-ii} and~\itemref{it:physical-vi} of Theorem~\ref{thm:physical}. The lower bound in \eqref{eq:main-higher-order} is incompatible with an exponential upper bound at times whose minimum gap is at least $N/C$.
\end{proof}

\section{Dynamical realisations under Theorem~{\ref{thm:symbolic}}}
\label{sec:stochastic-consequences}
The proofs in this section use a stationary process with the properties in Theorem~\ref{thm:symbolic}. We state that theorem first, then establish stability under changes of the observable, the fixed analytic realisation and the physical realisation. The existence of the process is proved afterwards.

\subsection{Symbolic realisation of finite-order exponential mixing without the CLT}
We state the probabilistic result first and postpone its proof to Sections~\ref{sec:construction}--\ref{sec:moment}. Throughout the present section, we assume this theorem and use only its stated properties.

\begin{notation}
    We use the following notation for the symbolic system and its coordinate process.
    
    \begin{itemize}
        \item Let $\Sigma:=\{-1,1\}^{\mathbb Z}$ and let $\sigma:\Sigma\to\Sigma$ be the left shift, $(\sigma\omega)_i:=\omega_{i+1}$.
        For $i\in\mathbb Z$, let $Z_i(\omega):=\omega_i$, and write $\varphi:=Z_0$.
        
        \item For a finite integer interval $I=[u,v]\cap\mathbb Z$, write $Z[I]:=(Z_u,\ldots,Z_v)$.
        For finite $I,J\subset\mathbb Z$ we use $\dist(I,J):=\min\{|i-j|:i\in I,\ j\in J\}$.
        
        \item For $0<\theta<1$, equip $\Sigma$ with the metric $d_\theta(\omega,\omega'):=\theta^{\min\{|i|:\omega_i\neq\omega_i'\}}$, and let $\mathcal F_\theta(\Sigma)$ denote the space of functions that are Lipschitz for $d_\theta$, with $\|F\|_\theta:=\|F\|_\infty+\Lip_\theta(F)$.
        Since $d_\theta^{\,\eta}=d_{\theta^\eta}$, varying $\theta$ already covers every H\"older exponent, and we use no separate exponent.
        
        \item For $0\leq n_1<\cdots<n_s$ we abbreviate $\mathbf n:=(n_1,\ldots,n_s)$ and set $$\Delta(\mathbf n):=\min_{1\leq i<s}(n_{i+1}-n_i).$$
        
        \item Write $S_n:=\sum_{i=0}^{n-1}Z_i$.
    \end{itemize}
\end{notation}

Theorem~\ref{thm:symbolic} constructs a stationary $\{-1,1\}$-valued process with exponential mixing of any prescribed order $R$, but whose normalised sums have no Gaussian weak subsequential limit. It is the main result from which Theorems~\ref{thm:analytic} and~\ref{thm:physical} are derived.

\begin{maintheorem}
    \label{thm:symbolic}
    For every integer $R\geq2$, there exists a full-support $\sigma$-invariant Borel probability measure $\nu$ on $\Sigma$ with the following properties.
    \begin{enumerate}[label=\textup{(\roman*)},ref=\roman*]
        \item
        \label{it:S-independence}
        Under $\nu$, the coordinate process $(Z_i)_{i\in\mathbb Z}$ is stationary, is $\{-1,1\}$-valued, satisfies $\nu(Z_0=-1)=\nu(Z_0=1)=1/2$, and is $(2R+1)$-tuplewise independent. Thus any collection of at most $2R+1$ coordinates with distinct indices is jointly independent.
        
        \item
        \label{it:S-windows}
        There exists $K_R<\infty$ such that, for all disjoint nonempty finite intervals $I_1,\ldots,I_R$ ordered from left to right,
            \begin{align}
                \min_{i<R}\dist(I_i,I_{i+1})>K_R\max_i|I_i| \quad\text{then}\quad Z[I_1],\ldots,Z[I_R]\text{ are jointly independent}. \label{eq:window-main}
            \end{align}
        
        \item
        \label{it:S-correlations}
        For every $0<\theta<1$, there exist constants $C_{R,\theta},c_\theta>0$ such that
        \begin{align}
            \left| \nu\left(\prod_{i=1}^RF_i\circ\sigma^{n_i}\right) -\prod_{i=1}^R\nu(F_i) \right| &\leq C_{R,\theta}e^{-c_\theta\Delta(\mathbf n)} \prod_{i=1}^R\|F_i\|_\theta \label{eq:symbolic-main}
        \end{align}
        for all $F_1,\ldots,F_R\in\mathcal F_\theta(\Sigma)$ and all integers $0\leq n_1<\cdots<n_R$.
        
        \item
        \label{it:S-moments}
        There exist $\delta>0$ and $n_0\geq1$ such that
        \begin{align}
            \E_\nu\left(\frac{S_n}{\sqrt n}\right)^{2R+2} \leq (2R+1)!!-\delta \quad\text{for every }n\geq n_0. \label{eq:symbolic-moment-main}
        \end{align}
        In particular, the laws of $S_n/\sqrt n$ under $\nu$ have no Gaussian weak subsequential limit. Explicitly, for every strictly increasing sequence of positive integers $n_k\to\infty$ and every $\sigma^2\geq0$,
        \begin{align*}
            \frac{S_{n_k}}{\sqrt{n_k}} \centernot{\Longrightarrow}\mathcal N(0,\sigma^2) \quad\text{as }k\to\infty,
        \end{align*}
        where convergence in distribution is taken with respect to $\nu$.

        \item
            \label{it:S-higher-order}
            There exist constants $c,C>0$ and $N_0\geq1$ such that, for every integer $N\geq N_0$, there are integers $0\leq n_1<\cdots<n_{2R+2}<N$ satisfying
            \begin{align}
                \Delta(\mathbf n)\geq\frac{N}{C} \quad\text{and}\quad -\E_\nu\left[\prod_{i=1}^{2R+2} Z_{n_i}\right] \geq cN^{-R-1}. \label{eq:symbolic-higher-order-main}
            \end{align}
            Hence, $(\Sigma,\sigma,\nu)$ is not exponentially mixing of order $2R+2$ in $\mathcal F_\theta(\Sigma)$ for any $0<\theta<1$.
    \end{enumerate}
\end{maintheorem}
The proof of Theorem~\ref{thm:symbolic} is carried out in Part~\ref{part:probabilistic} and completed at the end of Section~\ref{sec:moment}. We now assume Theorem~\ref{thm:symbolic} and prove Theorems~\ref{thm:analytic} and~\ref{thm:physical}. We first record a consequence of the finite-interval independence assertion that will be used for the physical realisation.

\begin{corollary}
    \label{cor:mixing}
    Every stationary probability measure on $\Sigma$ satisfying \eqref{eq:window-main} is mixing, and hence ergodic.
\end{corollary}
\begin{proof}
    Let $A,B$ be cylinder events. They depend on coordinates in two fixed finite intervals. For every sufficiently large $n$, those intervals, after translating the second by $n$, satisfy the separation condition in \eqref{eq:window-main}. Therefore $ \nu(A\cap\sigma^{-n}B)=\nu(A)\nu(B)$ for all sufficiently large $n$. Since the cylinder sets form an algebra generating the Borel $\sigma$-algebra of $\Sigma$, the standard approximation argument for mixing completes the proof (see~\cite[Section~1.7]{Walters}).
\end{proof}

\subsection{Stability of CLT failure under perturbations of the observable}
\label{sec:robustness}
We prove that the absence of Gaussian weak subsequential limits persists under small $\mathcal{C}^1$ perturbations of the observable. The argument requires only summable two-point correlations and will be used in both smooth realisations, including analytic approximation on the torus. We use the following distance to compare the laws of the normalised sums and measure their separation from Gaussian laws.

\begin{notation}[Bounded-Lipschitz distance]
    For probability measures $\mathbb P,\mathbb Q$ on $\mathbb R$, set
    $$
    d_{\mathrm{BL}}(\mathbb P,\mathbb Q):=\sup_{\|u\|_\infty\leq1,\ \Lip(u)\leq1}|\mathbb P(u)-\mathbb Q(u)|.
    $$
    Let $\mathcal G$ denote the set of all Gaussian probability measures on $\mathbb R$, including point masses, and write
    $$
    \dist_{\mathrm{BL}}(\mathbb P,\mathcal G):=\inf_{\mathbb Q\in\mathcal G}d_{\mathrm{BL}}(\mathbb P,\mathbb Q).
    $$
\end{notation}

The following lemma gives a uniform second-moment bound for the normalised sums. We use this bound to obtain tightness and to control how their laws change when the observable is perturbed.

\begin{lemma}
    \label{lem:variance}
    Suppose $T$ preserves a probability measure $\mu$ on a compact manifold and
    \begin{align}
        \left|\int f g\circ T^n \d \mu - \int f \d \mu \int g \d \mu\right| \le C\gamma(n)\|f\|_{\mathcal{C}^1}\|g\|_{\mathcal{C}^1}\quad\text{and}\quad \sum_{n\ge1}\gamma(n)<\infty, \label{eq:summable-correlations}
    \end{align}
    where $\gamma(n)\ge0$. Then every centred $h\in \mathcal{C}^1(M)$ satisfies
    \begin{align}
        \sup_{n\ge1}\frac{\|S_nh\|_{L^2(\mu)}}{\sqrt n} \le C_0\|h\|_{\mathcal{C}^1}\quad \text{with}\quad C_0^2:=1+2C\sum_{i\ge1}\gamma(i). \label{eq:variance-seminorm}
    \end{align}
\end{lemma}
\begin{proof}
    For every $n\geq1$, invariance of $\mu$ and \eqref{eq:summable-correlations} give
    \begin{align*}
        \frac1n\|S_nh\|_{L^2(\mu)}^2 &=\mu(h^2) +2\sum_{i=1}^{n-1} \left(1-\frac{i}{n}\right)\mu(h\,h\circ T^i)\leq \|h\|_\infty^2 +2C\sum_{i\geq1}\gamma(i) \|h\|_{\mathcal{C}^1(M)}^2 \\
        &\leq C_0^2\|h\|_{\mathcal{C}^1(M)}^2.
    \end{align*}
    Taking square roots and the supremum over $n\geq1$ proves \eqref{eq:variance-seminorm}.
\end{proof}

\begin{proposition}
    \label{prop:robustness}
    Assume \eqref{eq:summable-correlations}. Let $f\in \mathcal{C}^1(M)$ be centred and suppose that $S_nf/\sqrt n$ has no Gaussian weak subsequential limit. Then there are $\delta>0$, $c_0>0$ and $n_0\geq1$ such that
    \begin{align}
        \dist_{\mathrm{BL}}\left( \Law_\mu\left(\frac{S_n(g-\mu(g))}{\sqrt n}\right), \mathcal G \right)\geq c_0 \label{eq:robust-separation}
    \end{align}
    for every $n\geq n_0$ and every $g\in \mathcal{C}^1(M)$ with $\|g-f\|_{\mathcal{C}^1}<\delta$.
\end{proposition}
\begin{proof}
    Write $\mathbb P_n:=\Law_\mu(S_nf/\sqrt n)$. Lemma~\ref{lem:variance} gives
    $$
    \sup_{n\geq1}\int x^2\,\mathbb P_n(dx) \leq C_0^2\|f\|_{\mathcal{C}^1(M)}^2,
    $$
    so $(\mathbb P_n)$ is tight. We first show that
    \begin{align}
        \liminf_{n\to\infty} \dist_{\mathrm{BL}}(\mathbb P_n,\mathcal G)>0. \label{eq:gaussian-separation}
    \end{align}
    Otherwise, there would be $n_j\to\infty$ and $\mathbb G_j\in\mathcal G$ with $d_{\mathrm{BL}}(\mathbb P_{n_j},\mathbb G_j)\to0$. By Prohorov's theorem~\cite[Theorem~5.1]{Billingsley}, pass to a subsequence for which $\mathbb P_{n_j}$ converges in the weak topology to a probability measure $\mathbb P$. Since $d_{\mathrm{BL}}$ metrises this topology~\cite[Theorem~11.3.3]{Dudley},
    $$
    d_{\mathrm{BL}}(\mathbb G_j,\mathbb P) \leq d_{\mathrm{BL}}(\mathbb G_j,\mathbb P_{n_j}) +d_{\mathrm{BL}}(\mathbb P_{n_j},\mathbb P)\longrightarrow0.
    $$
    The Gaussian laws, including point masses, are closed in the space of probability measures on $\mathbb R$ with this weak topology (see~\cite[Lemma~B.1, item~1]{PSSW2015} or~\cite[Exercise~2.7.49]{BogachevWeak}). Thus $\mathbb P$ would be a Gaussian weak subsequential limit of $(\mathbb P_n)$, contrary to the hypothesis. This proves \eqref{eq:gaussian-separation}. Choose $c_0>0$ and $n_0$ so that the distance is at least $2c_0$ for every $n\geq n_0$.
    
    Now let $g\in\mathcal{C}^1(M)$ and put $h:=g-\mu(g)-f$. Because $\mu(f)=0$,
    $$
    \mu(h)=0,\quad |\mu(g)|=|\mu(g-f)|\leq\|g-f\|_\infty\quad\text{and}\quad \|h\|_{\mathcal{C}^1(M)}\leq2\|g-f\|_{\mathcal{C}^1(M)}.
    $$
    Write $\mathbb Q_n:=\Law_\mu(S_n(g-\mu(g))/\sqrt n)$. For a test function $u$ with $\|u\|_\infty\leq1$ and $\Lip(u)\leq1$, evaluate both sums on the same probability space. Their difference is $S_nh/\sqrt n$, so
    \begin{align*}
        |\mathbb P_n(u)-\mathbb Q_n(u)| &\leq\frac{\|S_nh\|_{L^1(\mu)}}{\sqrt n} \leq\frac{\|S_nh\|_{L^2(\mu)}}{\sqrt n}\leq C_0\|h\|_{\mathcal{C}^1(M)} \leq2C_0\|g-f\|_{\mathcal{C}^1(M)}.
    \end{align*}
    Taking the supremum over $u$ gives
    \begin{align}
        d_{\mathrm{BL}}(\mathbb P_n,\mathbb Q_n) \leq2C_0\|g-f\|_{\mathcal{C}^1(M)}. \label{eq:observable-law-distance}
    \end{align}
    Choose $\delta<c_0/(2C_0)$. For $n\geq n_0$ and $\|g-f\|_{\mathcal{C}^1(M)}<\delta$,
    $$
    \dist_{\mathrm{BL}}(\mathbb Q_n,\mathcal G) \geq\dist_{\mathrm{BL}}(\mathbb P_n,\mathcal G) -d_{\mathrm{BL}}(\mathbb P_n,\mathbb Q_n) >2c_0-2C_0\delta>c_0.
    $$
    This proves \eqref{eq:robust-separation}.
\end{proof}

\begin{corollary}
    \label{cor:analytic-observable}
    In the setting of Proposition~\ref{prop:robustness} on $\T^d$, a real trigonometric polynomial can be chosen whose centred normalised sums have no Gaussian weak subsequential limit.
\end{corollary}
\begin{proof}
    Real trigonometric polynomials are dense in $\mathcal{C}^1(\T^d)$ in the $\mathcal{C}^1$ norm (see~\cite[Corollary~3.1.11 and its proof]{Grafakos}). Choose such a polynomial in the neighbourhood from Proposition~\ref{prop:robustness}. It is real analytic, and subtracting its mean preserves that class.
\end{proof}
The neighbourhood in Proposition~\ref{prop:robustness} is for a fixed map and a fixed measure. It need not be uniform in the requested mixing order.

\subsection{Realisation on a fixed analytic diffeomorphism}
\label{sec:realisation}
\subsubsection{Hyperbolic coding}
For Theorem~\ref{thm:analytic}, the geometric construction is a two-symbol horseshoe in a fixed iterate of the cat map. We use the following classical results of the theory of hyperbolic dynamical systems.

\begin{theorem}[Homoclinic horseshoe theorem {\cite{KatokHasselblatt}}]
    \label{thm:horseshoe-input}
    Let a $\mathcal{C}^1$ diffeomorphism have a hyperbolic periodic point with a transverse homoclinic intersection. A positive iterate has a compact invariant hyperbolic Cantor set on which the iterate is conjugate to the full two-sided shift on two symbols.
\end{theorem}
\begin{theorem}[Regularity of hyperbolic coding {\cite[Chapter~18]{KatokHasselblatt}}]
    \label{thm:coding-input}
    For a finite Markov coding $\pi$ of a uniformly hyperbolic invariant set, there are $C>0$ and $0<\varrho<1$ such that the image of each cylinder specified on $[-\ell,\ell]$ has diameter at most $C\varrho^\ell$. In the full-shift horseshoe case, the coding may be taken to be a homeomorphism.
\end{theorem}

The following lemma is standard. Smale's horseshoe construction~\cite[Section~I.5, Proposition~5.3 and Theorem~5.5]{Smale1967}, applied to a transverse homoclinic point of the cat map, gives a two-symbol horseshoe after passing to a suitable iterate $T=A^{k_0}$. The estimate \eqref{eq:pullback} follows from Theorem~\ref{thm:coding-input}, and the zero-area conclusion follows from ergodicity of $T$ with respect to Haar area. These choices are independent of $R$.

\begin{lemma}
    \label{lem:fixed-horseshoe}
    There are $k_0\geq2$, a compact two-symbol horseshoe $\Lambda$ for $T=A^{k_0}$, and a coding homeomorphism $\pi:\Sigma\to\Lambda$ such that $T\circ\pi=\pi\circ\sigma$. Moreover, for some $0<\theta<1$ and $C_1<\infty$,
    \begin{align}
        \|f\circ\pi\|_\theta \leq C_1\|f\|_{\mathcal{C}^1(\mathbb T^2)} \quad\text{for every }f\in\mathcal{C}^1(\mathbb T^2). \label{eq:pullback}
    \end{align}
    Finally, $\Lambda$ has zero Haar area.
\end{lemma}

\subsubsection{Transporting the measure and the coordinate observable}

The following lemma transfers the symbolic correlation bounds to the horseshoe and realises the coordinate process by a smooth observable. We use the same argument in both realisations to preserve the laws of the sums and the higher-order correlations.

\begin{lemma}
    \label{lem:transfer}
    Let $T$ be a $\mathcal{C}^1$ diffeomorphism of a compact smooth Riemannian manifold $M$, and let $\pi:\Sigma\to\Lambda$ be a homeomorphism onto a compact invariant set, with $T\circ\pi=\pi\circ\sigma$. Suppose that, for some $0<\theta<1$ and $C_1<\infty$, $ \|f\circ\pi\|_\theta\leq C_1\|f\|_{\mathcal{C}^1(M)} \quad\text{for }f\in\mathcal{C}^1(M).$ For $R\geq2$, let $\nu$ be as in  Theorem~\ref{thm:symbolic} and set $\mu:=\pi_*\nu$. Then the following hold.
    \begin{enumerate}[label=\textup{(\roman*)},ref=\roman*]
        \item\label{it:transfer-i} The measure $\mu$ is $T$-invariant and has support $\Lambda$.
        There are constants $C_R,c_R>0$ such that
        \begin{align}
            \left|\mu\left(\prod_{i=1}^R f_i\circ T^{n_i}\right) -\prod_{i=1}^R\mu(f_i)\right| &\leq C_R e^{-c_R\Delta(\mathbf n)} \prod_{i=1}^R\|f_i\|_{\mathcal{C}^1(M)} \label{eq:correlation-transfer}
        \end{align}
        for all $f_1,\ldots,f_R\in\mathcal{C}^1(M)$ and $0\leq n_1<\cdots<n_R$, where $C_R=C_{R,\theta}C_1^R$ and $c_R=c_\theta$.
        \item\label{it:transfer-ii} There is $f_0\in\mathcal{C}^{\infty}(M)$, depending only on $\pi$, such that
        \begin{align}
            (f_0\circ T^i)\circ\pi=Z_i\quad(i\in\mathbb Z) \quad\text{and}\quad S_nf_0\circ\pi=S_n. \label{eq:sum-transfer}
        \end{align}
        Thus the process $(f_0\circ T^i)_{i\in\mathbb Z}$ under $\mu$ has the same law as the coordinate process under $\nu$.
    \end{enumerate}
\end{lemma}

\begin{proof}
    For every bounded Borel function $f$, the conjugacy and $\sigma$-invariance of $\nu$ give
    $$
    \int f\circ T\,\d\mu =\int(f\circ\pi)\circ\sigma\,\d\nu =\int f\circ\pi\,\d\nu =\int f\,\d\mu.
    $$
    Hence $\mu$ is $T$-invariant. Also $\mu(\Lambda)=1$. Every open neighbourhood of a point of $\Lambda$ has a nonempty open preimage under the homeomorphism $\pi$. Since $\nu$ has full support, that neighbourhood has positive $\mu$-measure, proving $\supp\mu=\Lambda$.
    
    Fix $f_1,\ldots,f_R\in\mathcal{C}^1(M)$ and put $F_i=f_i\circ\pi$. Applying Theorem~\ref{thm:symbolic}\itemref{it:S-correlations} yields
    \begin{align*}
        \left|\mu\left(\prod_{i=1}^R f_i\circ T^{n_i}\right) -\prod_{i=1}^R\mu(f_i)\right| &=\left|\nu\left(\prod_{i=1}^RF_i\circ\sigma^{n_i}\right) -\prod_{i=1}^R\nu(F_i)\right|\leq C_{R,\theta}e^{-c_\theta\Delta(\mathbf n)} \prod_{i=1}^R\|F_i\|_\theta \\
        &\leq C_{R,\theta}C_1^R e^{-c_\theta\Delta(\mathbf n)} \prod_{i=1}^R\|f_i\|_{\mathcal{C}^1(M)}.
    \end{align*}
    This gives \eqref{eq:correlation-transfer} with the stated constants.
    
    For the observable, the sets $\Lambda_\pm:=\pi\{\omega:\omega_0=\pm1\}$ are disjoint and compact. Choose smooth cutoffs $\chi_\pm$, equal to one near $\Lambda_\pm$ and supported in disjoint neighbourhoods, and set $f_0:=\chi_+-\chi_-$. Then $f_0\circ\pi=Z_0$. For every $i\in\mathbb Z$,
    $$
    (f_0\circ T^i)\circ\pi =f_0\circ\pi\circ\sigma^i =Z_0\circ\sigma^i=Z_i.
    $$
    Summing over $i=0,\ldots,n-1$ proves \eqref{eq:sum-transfer}. Because $\mu=\pi_*\nu$, these identities identify every finite-dimensional distribution of the two processes. The cutoffs depend only on $\pi$, so the same $f_0$ works for all $R$ and $\nu$.
\end{proof}

\begin{proof}[Proof of Theorem~\ref{thm:analytic}]
    Fix $T,\Lambda,\pi$ from Lemma~\ref{lem:fixed-horseshoe} and the observable $f_0$ from Lemma~\ref{lem:transfer}. For a given $R\geq2$, take $\nu$ from Theorem~\ref{thm:symbolic} and set $\mu:=\pi_*\nu$. Lemma~\ref{lem:transfer}\itemref{it:transfer-i} gives the support and correlation conclusions \itemref{it:analytic-i}--\itemref{it:analytic-ii}.
    
    By \eqref{eq:sum-transfer}, the process of $f_0$ under $\mu$ has the same law as the symbolic coordinates under $\nu$. Thus Theorem~\ref{thm:symbolic}\itemref{it:S-independence}, \itemref{it:S-moments} gives \itemref{it:analytic-iii}--\itemref{it:analytic-iv}, including $\mu(f_0)=0$. The same identity gives, for any observation times,
    $$
    \mu\left(\prod_{i=1}^{2R+2} f_0\circ T^{n_i}\right) =\nu\left(\prod_{i=1}^{2R+2} Z_{n_i}\right).
    $$
    Item~\itemref{it:S-higher-order} of Theorem~\ref{thm:symbolic} therefore provides the separated-time lower bound in \itemref{it:analytic-vi}.
    
    Finally, \itemref{it:analytic-ii} implies exponential mixing and hence \eqref{eq:summable-correlations}. Proposition~\ref{prop:robustness} gives the neighbourhood in \itemref{it:analytic-v}, and Corollary~\ref{cor:analytic-observable} provides a real trigonometric polynomial in it. The map, coding and original observable were all fixed before $R$ was chosen.
\end{proof}

\subsection{Realisation with a physical measure}
\label{sec:physical}

We use the following smooth physical realisation result.
To avoid interrupting the flow of the argument, its proof
is given in Appendix~\ref{app:smooth-realisation-short}.

\begin{proposition}[Smooth physical realisation of symbolic measures]
\label{prop:cv-all-measures}
Let $M$ be a nonempty compact smooth surface without boundary.
There exist a $\mathcal C^\infty$ diffeomorphism  $F_0:M\to M$, a two-symbol horseshoe $\Lambda$ for $F_0$, and a homeomorphism $\pi:\Sigma\to\Lambda$ satisfying $F_0\circ\pi=\pi\circ\sigma$, such that the following holds.

For every ergodic $\sigma$-invariant Borel probability measure $\nu$ on $\Sigma$, there are a $\mathcal C^\infty$ diffeomorphism   $T:M\to M$, equal to $F_0$ in a neighbourhood of $\Lambda$, and a nonempty open set $W$ such that $\mu=\pi_*\nu$ is physical and $W$ is contained in its basin. In fact, for every $v\in\mathcal C^0(M)$,
\begin{equation}
\label{eq:cv-uniform-physical}
 \lim_{N\to\infty}\sup_{y\in W}
 \left|\frac{1}{N}\sum_{j=0}^{N-1}v(T^jy)-\int v\,d\mu\right|=0.
\end{equation}
Moreover, $T(\Lambda)=\Lambda$, $T\circ\pi=\pi\circ\sigma$,
and there exist $0<\theta<1$ and $C_1<\infty$ such that
$$
 \|v\circ\pi\|_\theta\leq C_1\|v\|_{\mathcal C^1(M)}
 \quad\text{for every }v\in\mathcal C^1(M).
$$
\end{proposition}

\begin{proof}[Proof of Theorem~\ref{thm:physical}]
Fix $M$ and $R\geq2$. Let $F_0$, $\Lambda$ and $\pi$
be given by Proposition~\ref{prop:cv-all-measures}.
Take $\nu$ from Theorem~\ref{thm:symbolic}.
It is ergodic by Corollary~\ref{cor:mixing}.

Proposition~\ref{prop:cv-all-measures} gives a $\mathcal C^\infty$ diffeomorphism $T:M\to M$ and a nonempty open
set $W$ such that $\mu=\pi_*\nu$ satisfies \itemref{it:physical-v}.
The same proposition gives the conjugacy and the pullback
bound required by Lemma~\ref{lem:transfer}.
That lemma proves \itemref{it:physical-i}--\itemref{it:physical-iii} and \itemref{it:physical-vi}
by transferring the symbolic estimates and the
coordinate-process law to $M$.
Finally, Proposition~\ref{prop:robustness} gives \itemref{it:physical-iv}.
\end{proof}

\part{Probabilistic construction}
\label{part:probabilistic}
\addtocontents{toc}{\protect\addvspace{4pt}}

In this part, we construct the stationary process $(Z_i)_{i\in\mathbb Z}$ used in Part~\ref{part:dynamical} and prove Theorem~\ref{thm:symbolic}. We begin with a sketch of the construction. The details and proofs follow in Sections~\ref{sec:geometry}--\ref{sec:moment}.

\subsection*{Sketch of the construction}

Fix an integer $R\geq2$ and set $L:=2R+2$. We construct the stationary process $(Z_i)_{i\in\Z}$ by defining its probability measure $\nu$ on $\Sigma=\{-1,1\}^{\Z}$. The construction has three steps, followed by references to the proofs of its properties.

\paragraph{\textbf{Step~1. Random interval partitions.}}
We first construct a nested sequence of random interval partitions of $\Z$. Starting from $\mathcal M_0=\Z$, we choose random sets of markers $\mathcal M_k\subset\mathcal M_{k-1}\subset\Z$. The intervals between consecutive markers in $\mathcal M_k$ form a partition $\mathcal P_k$ of $\Z$. Its elements, called level-$k$ blocks, are the intervals $[a,b)\cap\Z$ with $a,b\in\mathcal M_k$, $a<b$, and $(a,b)\cap\mathcal M_k=\varnothing$. Since the marker sets are nested, every level-$k$ block is a union of consecutive level-$(k-1)$ blocks, which we call its children.

The sets $\{\mathcal M_k\}_{k\geq 0}$ are constructed so that, with probability one, every level-$k$ block $B\in\mathcal P_k$ satisfies
$$
(4L)^k<|B|\leq3(4L)^k, \quad\text{for every }k\geq1,
$$
(Lemma~\ref{lem:geometry}\itemref{it:geometry-iii}), and has at least $L$ children (Lemma~\ref{lem:geometry}\itemref{it:geometry-iv}). Thus, the block lengths grow at a controlled rate, and each parent has enough children for the interaction introduced in the next step.

The nested marker sets $\mathcal H=(\mathcal M_k)_{k\geq0}$ are stationary, in the sense that translating all of them by the same integer leaves their joint law unchanged (Lemma~\ref{lem:stationarity-geometry}), i.e.
$$
(\mathcal M_k+m)_{k\geq0} \stackrel{d}{=}(\mathcal M_k)_{k\geq0} \quad\text{for every }m\in\Z,
$$
where $\mathcal M_k+m:=\{x+m:x\in\mathcal M_k\}$.

\paragraph{\textbf{Step~2. Recursive construction.}}
Fix a realisation $h$ of the nested marker sets $\mathcal H$, and write $\mathcal P_k(h)$ for the corresponding level-$k$ partition. We recursively define a probability measure $\rho_B$ on $\{-1,1\}^B$ for each block $B\in\mathcal P_k(h)$.

At level $0$, every block is a singleton $B=\{j\}$, and $\rho_B$ assigns probability $1/2$ to each of the values $X_j=-1$ and $X_j=1$.

Let $k\geq1$ and $B\in\mathcal P_k(h)$. Write $B_1,\ldots,B_r\in\mathcal P_{k-1}(h)$ for its children (i.e. $B = B_1\cup B_2\cup \ldots \cup B_r$), ordered from left to right. Assume that the laws $\rho_{B_1},\ldots,\rho_{B_r}$ have already been defined. Observe that, from Step~1, that $r\geq L$.

For $1\leq i\leq L$, Lemma~\ref{lem:bias} gives an odd function $\psi_i=\psi_{\rho_{B_i}}$ taking values in $[-1,1]$ such that
$$
\E_{\rho_{B_i}}\bigl[\psi_i(X[B_i])\bigr]=0 \quad\text{and}\quad \E_{\rho_{B_i}}\bigl[X_j\psi_i(X[B_i])\bigr] =\frac{1}{\sqrt{3|B_i|}} \quad\text{for every }j\in B_i.
$$
Here $X_j(x)=x_j$ and $X[B_i]=(X_j)_{j\in B_i}$. With coordinates ordered from left to right, we use the same function $\psi_\rho$ whenever the same law $\rho$ occurs.

We define the probability $\rho_B$ on $\{-1,1\}^B$ as
$$
\rho_B(dx)= \left(1-\frac12\prod_{i=1}^{L}\psi_i(x[B_i])\right) \bigotimes_{i=1}^{r}\rho_{B_i}(dx[B_i]).
$$
The product measure includes all $r$ blocks, whereas the density depends only on the first $L$. The density lies between $1/2$ and $3/2$ and integrates to one, so $\rho_B$ is a probability measure with full support (Lemma~\ref{lem:block-properties}\itemref{it:block-properties-i}). Integrating over any one of the first $L$ blocks eliminates the product term because the corresponding $\psi_i$ has mean zero. Thus, under $\rho_B$, any collection of at most $L-1$ vectors $X[B_i]$ is independent, and each has law $\rho_{B_i}$ (Lemma~\ref{lem:block-properties}\itemref{it:block-properties-iii}).

The law $\rho_B$ remains invariant under $x\mapsto-x$, and its coordinates remain $(L-1)$-tuplewise independent and uniform on $\{-1,1\}$ (Lemma~\ref{lem:block-properties}\itemref{it:block-properties-i}--\itemref{it:block-properties-ii}). These properties allow Lemma~\ref{lem:bias} to be applied at every level. Choosing one coordinate in each of the first $L$ blocks gives a negative $L$-point correlation (Lemma~\ref{lem:block-correlations}).

\paragraph{\textbf{Step~3. The stationary process.}}
Once again we fix a realisation $h$ of the nested marker sets $\mathcal H$. For each $k\geq0$, define the probability $\nu_h^{(k)}$ on $\{-1,1\}^\mathbb Z$ as
$$
\nu_h^{(k)}:=\bigotimes_{B\in\mathcal P_k(h)}\rho_B.
$$
For every nonempty finite set $E\subset\Z$, choose $k$ such that $(4L)^k\geq\max E-\min E+1$. Then
$$
\nu_h^{(\ell)}\big|_E=\nu_h^{(k)}\big|_E \quad\text{for every }\ell\geq k
$$
(Lemma~\ref{lem:stabilisation}). These stabilised finite-dimensional distributions are consistent and, by the Kolmogorov extension theorem, determine a probability measure $\nu_h$ on $\Sigma$ (Proposition~\ref{prop:limit}\itemref{it:limit-i}).

Averaging over the nested marker sets gives
$$
\nu(A):=\int\nu_h(A)\,\P_{\mathcal H}(dh) \quad\text{for }A\in\mathcal B(\Sigma),
$$
where $\P_{\mathcal H}$ is the law of $\mathcal H$. On the probability space $(\Sigma,\mathcal B(\Sigma),\nu)$, define
$$
Z_i(\omega):=\omega_i \quad\text{for every }\omega\in\Sigma\text{ and }i\in\Z.
$$

\paragraph{\textbf{Properties of the process.}}
The coordinate process $(Z_i)_{i\in\Z}$ is stationary, and its law $\nu$ has full support on $\Sigma$ (Proposition~\ref{prop:limit}\itemref{it:limit-ii}). Item~\itemref{it:S-independence} of Theorem~\ref{thm:symbolic} follows from Proposition~\ref{prop:limit}\itemref{it:limit-iii}. Items~\itemref{it:S-windows} and~\itemref{it:S-correlations} of Theorem~\ref{thm:symbolic} are proved in Propositions~\ref{prop:windows} and~\ref{prop:holder}, respectively.

The key properties of the process are stated in items~\itemref{it:S-moments}--\itemref{it:S-higher-order}. Item~\itemref{it:S-moments} gives the failure of the CLT for any subsequence (Corollaries~\ref{cor:moment} and~\ref{cor:no-gaussian}). Item~\itemref{it:S-higher-order} shows that exponential mixing does not extend to all orders. It holds at order $R$ but fails at order $L=2R+2$ (Corollary~\ref{cor:higher-order}).

\section{Stationary nested marker sets and interval partitions}
\label{sec:geometry}
\label{sec:construction}

Fix an integer $R\geq2$ and set $L:=2R+2$, so that $L\geq6$ is even. We construct stationary nested random marker sets and the associated interval partitions on which the coordinate process will be built. At level $k$, the block lengths and the dependence range of the marker configurations are comparable to $(4L)^k$. The factor $4$ in this scale ensures that each parent contains at least $L$ children, allowing the interaction in Section~\ref{sec:block-laws} to involve $L$ distinct children. The bounds on block lengths and dependence ranges will yield finite-window independence and exponential correlation bounds. Proposition~\ref{prop:hierarchy} records the four properties used in the later sections; those sections do not use the particular marker selection rule.

We use the finitely dependent colourings of Holroyd and Liggett~\cite[Corollary~20]{HolroydLiggett} to construct nested marker sets. Our selection of markers is similar to the construction of nets in~\cite[proof of Corollary~11]{HolroydSchrammWilson2017}. Here we repeat the selection within each preceding marker set and establish the block-length and finite-dependence bounds needed for our construction.
\subsection{Nested marker sets and associated partitions}
We consider a sequence $\mathcal H=(\mathcal M_k)_{k\geq0}$ of nested bi-infinite subsets of $\mathbb Z$, with
$$
\mathbb Z=\mathcal M_0\supset\mathcal M_1 \supset\mathcal M_2\supset\cdots.
$$
We use the following notation.

\begin{itemize}
    \item The elements of $\mathcal M_k$ are the \emph{level-$k$ markers}.
    The intervals between consecutive markers form the partition
    $$
    \mathcal P_k:= \bigl\{[u,v)\cap\mathbb Z: u,v\in\mathcal M_k,\ u<v,\ \mathcal M_k\cap(u,v)=\varnothing\bigr\}
    $$
    of $\mathbb Z$ into finite intervals, called \emph{level-$k$ blocks}.
    
    \item For $k\geq1$, each block $B\in\mathcal P_k$ is a union of consecutive blocks $B_1,\ldots,B_r\in\mathcal P_{k-1}$, because $\mathcal M_k\subset\mathcal M_{k-1}$.
    We call these blocks the \emph{children} of $B$ and write $N_k(B):=r$ for their number. Thus $ B=B_1\cup\cdots\cup B_{N_k(B)}$, with $B_i\in\mathcal P_{k-1}.$
    
    \item The \emph{configuration} of $\mathcal H$ at $i\in\mathbb Z$ through level $k$ is
    \begin{align}
        \mathcal H_i^{(k)} := \bigl( \one_{\{i\in\mathcal M_0\}}, \ldots, \one_{\{i\in\mathcal M_k\}} \bigr). \label{eq:hierarchy-configuration}
    \end{align}
    We write $\mathcal H^{(k)}:=(\mathcal H_i^{(k)})_{i\in\mathbb Z}$ for the process recording marker membership through level $k$.
    
    \item For a finite interval $I=[u,v]\cap\mathbb Z$, write
    $$
    \mathcal H^{(k)}[I] := (\mathcal H_u^{(k)},\ldots,\mathcal H_v^{(k)}).
    $$
    It determines all blocks and their subdivisions in $I$ up to level $k$.
    
    \item For an interval $I$ and $r\geq0$, write
\begin{equation}
I^{+r}:=\{j\in\Z:\dist(j,I)\leq r\}. \label{eq:interval-enlargement}
\end{equation}
\end{itemize}

Finite dependence will allow us to make sufficiently separated local marker configurations independent.
\begin{definition}[Finite dependence]
    For $d\geq0$, a process indexed by $\mathbb Z$ is \emph{$d$-dependent} if its restrictions to any two sets of indices at distance greater than $d$ generate independent $\sigma$-algebras.
\end{definition}
Figures~\ref{fig:nested-markers-blocks} and~\ref{fig:marker-configurations} illustrate the notation introduced above.

\begin{figure}[htbp]
    \centering
    \begin{tikzpicture}[
        x=.74cm,y=1cm,
        font=\small,
        line width=.45pt,
        marker/.style={circle,fill=black,inner sep=1.65pt},
        guide/.style={densely dotted,black!40},
        block/.style={draw=black,fill=black!5},
        brace/.style={decorate,decoration={brace,amplitude=4pt}}
        ]
        \def\labelx{-1.25}
        
        \foreach \x in {0,9}
        \draw[guide] (\x,.35) -- (\x,5.32);
        \foreach \x in {3,5}
        \draw[guide] (\x,.35) -- (\x,3.32);
        
        \node[anchor=east] at (\labelx,4.55)
        {$\mathcal M_2$};
        \node[anchor=east] at (\labelx,5.05)
        {$\mathcal P_2$};
        
        \draw[black!45] (-.25,4.55) -- (12.3,4.55);
        
        \foreach \x in {0,9}
        \node[marker] at (\x,4.55) {};
        
        \draw[block,fill=black!12]
        (0,4.8) rectangle (9,5.32);
        \node at (4.5,5.06)
        {$B=[0,9)\cap\mathbb Z$};
        
        \draw (9,4.8) -- (12.3,4.8);
        \draw (9,5.32) -- (12.3,5.32);
        \node at (11,5.06) {$\cdots$};
        \node at (-.65,5.06) {$\cdots$};
        
        \node[anchor=east] at (\labelx,2.55)
        {$\mathcal M_1$};
        \node[anchor=east] at (\labelx,3.05)
        {$\mathcal P_1$};
        
        \draw[black!45] (-.25,2.55) -- (12.3,2.55);
        
        \foreach \x in {0,3,5,9,12}
        \node[marker] at (\x,2.55) {};
        
        \foreach \a/\b/\lab in {0/3/B_1,3/5/B_2,5/9/B_3}
        \draw[block] (\a,2.8) rectangle (\b,3.32)
        node[midway] {$\lab$};
        
        \draw (9,2.8) rectangle (12,3.32);
        \node at (10.5,3.06) {$\cdots$};
        \node at (-.65,3.06) {$\cdots$};
        
        \draw[brace] (0,3.52) -- (9,3.52)
        node[midway,above=6pt]
        {$B=B_1\cup B_2\cup B_3,\quad N_2(B)=3$};
        
        \node[anchor=east] at (\labelx,.55)
        {$\mathcal M_0=\mathbb Z$};
        \node[anchor=east] at (\labelx,1.05)
        {$\mathcal P_0$};
        
        \draw[black!45] (-.25,.55) -- (12.3,.55);
        
        \foreach \x in {0,...,12} {
            \node[marker] at (\x,.55) {};
            \node[below=5pt,font=\scriptsize] at (\x,.55)
            {$\x$};
        }
        
        \foreach \x in {0,...,11}
        \draw[block] (\x,.8) rectangle (\x+1,1.32);
        
        \foreach \x in {0,...,8}
        \node[font=\footnotesize] at (\x+.5,1.06) {$\{\x\}$};
        
        \foreach \y in {.55,2.55,4.55} {
            \node at (-.65,\y) {$\cdots$};
            \node at (12.7,\y) {$\cdots$};
        }
    \end{tikzpicture}
    
    \caption{Visual representation of $\mathcal M_0,\mathcal M_1,\mathcal M_2,\mathcal P_0,\mathcal P_1,\mathcal P_2$. The level-$2$ block $B$ has the three level-$1$ children
        $B_1=[0,3)\cap\mathbb Z$, $B_2=[3,5)\cap\mathbb Z$ and $B_3=[5,9)\cap\mathbb Z$. Each level-$1$ block is a union of level-$0$ singletons. The rows continue in both directions.}
    \label{fig:nested-markers-blocks}
\end{figure}
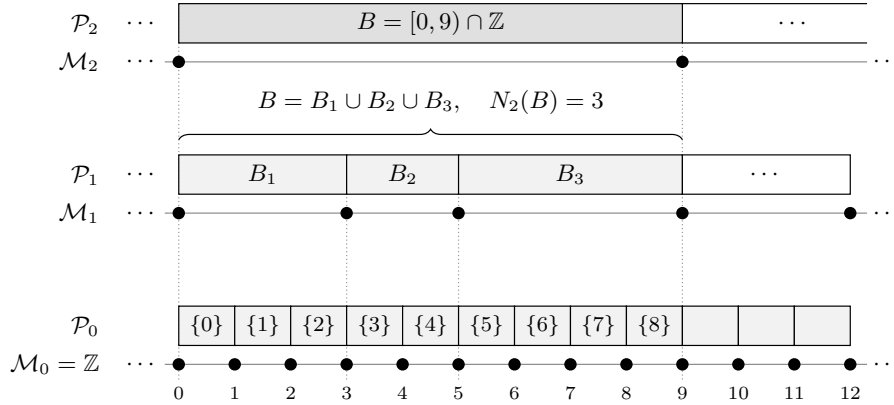

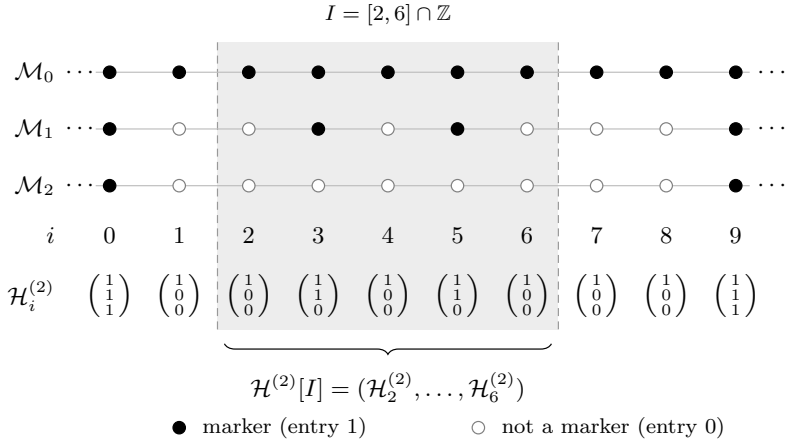
\begin{figure}[htbp]
    \centering
    \begin{tikzpicture}[
        x=.92cm,y=1cm,
        font=\small,
        line width=.45pt,
        present/.style={
            circle,draw=black,fill=black,inner sep=1.65pt
        },
        absent/.style={
            circle,draw=black!50,fill=white,inner sep=1.65pt
        },
        brace/.style={
            decorate,
            decoration={brace,amplitude=4pt,mirror}
        }
        ]
        \fill[black!7] (1.55,1.55) rectangle (6.45,5.35);
        \draw[black!40,densely dashed]
        (1.55,1.55) -- (1.55,5.35);
        \draw[black!40,densely dashed]
        (6.45,1.55) -- (6.45,5.35);
        
        \node[anchor=south,font=\footnotesize] at (4,5.45)
        {$I=[2,6]\cap\mathbb Z$};
        
        \foreach \k/\y in {0/4.95,1/4.2,2/3.45} {
            \node[anchor=east] at (-.65,\y)
            {$\mathcal M_{\k}$};
            \draw[black!25] (-.2,\y) -- (9.2,\y);
            \node at (-.4,\y) {$\cdots$};
            \node at (9.55,\y) {$\cdots$};
        }
        
        \foreach \x in {0,...,9}
        \node[present] at (\x,4.95) {};
        
        \foreach \x in {0,3,5,9}
        \node[present] at (\x,4.2) {};
        \foreach \x in {1,2,4,6,7,8}
        \node[absent] at (\x,4.2) {};
        
        \foreach \x in {0,9}
        \node[present] at (\x,3.45) {};
        \foreach \x in {1,...,8}
        \node[absent] at (\x,3.45) {};
        
        \node[anchor=east] at (-.65,2.8) {$i$};
        \foreach \x in {0,...,9}
        \node at (\x,2.8) {$\x$};
        
        \node[anchor=east] at (-.65,2.0)
        {$\mathcal H_i^{(2)}$};
        
        \foreach \x/\b/\c in {
            0/1/1,
            1/0/0,
            2/0/0,
            3/1/0,
            4/0/0,
            5/1/0,
            6/0/0,
            7/0/0,
            8/0/0,
            9/1/1
        }
        \node at (\x,2.0)
        {$\left(
            \begin{smallmatrix}
                1\\ \b\\ \c
            \end{smallmatrix}
            \right)$};
        
        \draw[brace] (1.65,1.36) -- (6.35,1.36)
        node[midway,below=7pt]
        {$\mathcal H^{(2)}[I]
            =(\mathcal H_2^{(2)},\ldots,\mathcal H_6^{(2)})$};
        
        \node[present] at (1,.25) {};
        \node[anchor=west,font=\footnotesize] at (1.2,.25)
        {marker (entry $1$)};
        
        \node[absent] at (5.3,.25) {};
        \node[anchor=west,font=\footnotesize] at (5.5,.25)
        {not a marker (entry $0$)};
    \end{tikzpicture}
    
    \caption{Configurations of the nested marker sets through level $2$.
        For each integer $i$, the tuple $\mathcal H_i^{(2)}$ records membership in $\mathcal M_0$, $\mathcal M_1$ and $\mathcal M_2$, with its entries displayed vertically in this order. The shaded columns form the restriction $\mathcal H^{(2)}[I]$ to the interval $I=[2,6]\cap\mathbb Z$.}
    \label{fig:marker-configurations}
\end{figure}

The following proposition gives the required stationary nested marker sets.

\begin{proposition}
    \label{prop:hierarchy}
    There exist a finite constant $D=D(L)$ and a random sequence of nested marker sets $\mathcal H=(\mathcal M_k)_{k\geq0}$, with law $\mathbb P_{\mathcal H}$, i.e.\ the distribution of $\mathcal H$ on $\bigl(\{0,1\}^{\mathbb Z}\bigr)^{\mathbb Z_{\geq0}}$ with its product $\sigma$-algebra, such that the following hold.
    \begin{enumerate}[label=\textup{(H\arabic*)},ref=H\arabic*]
        \item
        \label{it:H-lengths}
        Almost surely, for every $k\geq1$ and every $B\in\mathcal P_k$, $ (4L)^k<|B|\leq3(4L)^k. $
        
        \item
        \label{it:H-children}
        Almost surely, for every $k\geq1$ and every $B\in\mathcal P_k$, $N_k(B)>4L/3>L.$
        
        \item
        \label{it:H-stationary}
        For every $k\geq0$, $p\geq1$ and $i_1,\ldots,i_p,h\in\mathbb Z$,
        $$
        \Law\bigl(\mathcal H_{i_1+h}^{(k)},\ldots,\mathcal H_{i_p+h}^{(k)}\bigr) = \Law\bigl(\mathcal H_{i_1}^{(k)},\ldots,\mathcal H_{i_p}^{(k)}\bigr).
        $$
        
        \item
        \label{it:H-dependence}
        For every $k\geq0$, the process $\mathcal H^{(k)}$ is $(4L)^k D$-dependent.
    \end{enumerate}
\end{proposition}
We now construct the marker sets and prove Proposition~\ref{prop:hierarchy}. A reader who assumes the proposition may proceed directly to Section~\ref{sec:block-laws}.

\subsection{A stationary finitely dependent colouring}
Recall that $L=2R+2$ is fixed throughout. We use the following consequence of~\cite[Corollary~20]{HolroydLiggett}.

\begin{theorem}[{\cite[Corollary~20]{HolroydLiggett}}]
    \label{thm:colouring}
    There are an integer $q=q(L)\geq2$ and a process $V=(V_i)_{i\in\mathbb Z}$ taking values in the set of colours $\{1,\ldots,q\}$ with the following properties.
    \begin{enumerate}[label=\textup{(\roman*)},ref=\roman*]
        \item\label{it:colouring-i} The process is strictly stationary, meaning that for every $p\geq1$ and every $j_1,\ldots,j_p,h\in\mathbb Z$, $\Law(V_{j_1+h},\ldots,V_{j_p+h}) = \Law(V_{j_1},\ldots,V_{j_p}).$
        
        \item\label{it:colouring-ii} The process is $(4L)$-dependent.
        
        \item\label{it:colouring-iii} Almost surely,
        \begin{align}
            V_i\neq V_j \quad\text{whenever}\quad 0<|i-j|\leq4L. \label{eq:colouring}
        \end{align}
    \end{enumerate}
\end{theorem}
Following~\cite{HolroydLiggett}, we refer to the values of $V$ as colours. Property \eqref{eq:colouring} says that nearby indices receive different colours, which is the feature used in the selection rule below. Fix this process once and for all. Only the fact that $q<\infty$ is used.

\subsection{The marker selection rule}
\label{subsec:markers}
We construct the marker sets recursively. For each level $\ell\geq0$, we denote the increasing enumeration of $\mathcal M_\ell$ by $(t_j^{(\ell)})_{j\in\mathbb Z}$, chosen so that $t_0^{(\ell)}\leq0<t_1^{(\ell)}$.
\begin{itemize}
    \item \emph{Initial level.} Set $\mathcal M_0:=\mathbb Z$, so that $t_j^{(0)}=j$ and $\mathcal P_0=\bigl\{\{j\}:j\in\mathbb Z\bigr\}$.
    
    \item \emph{Enumerate the old markers.} Fix $k\geq1$ and suppose that $\mathcal M_{k-1}$ has been constructed.
    Write its increasing enumeration as
    $$
    \mathcal M_{k-1}=\left\{t_j^{(k-1)}:j\in\mathbb Z\right\}\quad\text{with}\quad \cdots<t_{-1}^{(k-1)}<t_0^{(k-1)}\leq0<t_1^{(k-1)}<t_2^{(k-1)}<\cdots.
    $$
    \item \emph{Assign colours.} Let $V^{(k)}=(V_i^{(k)})_{i\in\mathbb Z}$ be an independent copy of the process in Theorem~\ref{thm:colouring}, independent of all randomness used at the previous levels, and assign the colour $V_j^{(k)}$ to the marker $t_j^{(k-1)}$.
    The colour attached to $t_j^{(k-1)}$ is $V_j^{(k)}$, not $V_{t_j^{(k-1)}}^{(k)}$.
    
    \item \emph{Select the new markers.} Set $A_0:=\varnothing$, with $\dist(t,\varnothing):=+\infty$, and, for $a=1,\ldots,q$, define
    \begin{align}
        A_a := A_{a-1} \cup \left\{ t_j^{(k-1)}\in\mathcal M_{k-1}: V_j^{(k)}=a,\ \dist\left(t_j^{(k-1)},A_{a-1}\right)>(4L)^k \right\}. \label{eq:marker-selection}
    \end{align}
    Set $\mathcal M_k:=A_q$. In particular, $\mathcal M_k\subset\mathcal M_{k-1}$.
\end{itemize}
The enumeration $t_0^{(k-1)}\leq0<t_1^{(k-1)}$ is used only to assign the colouring.

The following lemma controls the geometry of the partitions at every level. The block lengths are comparable to $(4L)^k$, and each parent $B\in\mathcal P_k$ contains at least $L$ children $B_i\in\mathcal P_{k-1}$, allowing the recursive construction of the block laws.

\begin{lemma}
    \label{lem:geometry}
    For every $k\geq1$, the marker set $\mathcal M_k$ is almost surely a bi-infinite subset of $\mathcal M_{k-1}$. Moreover, the following hold.
    \begin{enumerate}[label=\textup{(\roman*)},ref=\roman*]
        \item\label{it:geometry-i} The set $\mathcal M_k$ is $(4L)^k$-separated, that is,
        $$
        |u-v|>(4L)^k \quad \text{for all distinct }u,v\in\mathcal M_k.
        $$
        \item\label{it:geometry-ii} The set $\mathcal M_k$ is maximal with this property inside $\mathcal M_{k-1}$, meaning that for every $t\in\mathcal M_{k-1}$ there exists $u\in\mathcal M_k$ with $|t-u|\leq(4L)^k$.
        Equivalently, $\mathcal M_k$ is a $(4L)^k$-net in $\mathcal M_{k-1}$.
        \item\label{it:geometry-iii} The family $\mathcal P_k$ is a partition of $\mathbb Z$ into finite intervals, and every $B\in\mathcal P_k$ satisfies
        \begin{align}
            (4L)^k<|B| \leq 2(4L)^k+3(4L)^{k-1} \leq 3(4L)^k. \label{eq:block-lengths}
        \end{align}
        \item\label{it:geometry-iv} Every $B\in\mathcal P_k$ is the union of consecutive blocks of $\mathcal P_{k-1}$, and
        $$
        N_k(B)>\frac{4L}{3}.
        $$
        In particular, every level-$k$ block has at least $L$ children.
    \end{enumerate}
\end{lemma}

\begin{proof}
    We argue by induction on $k\geq1$, recalling that $\mathcal M_0=\mathbb Z$ and $\mathcal P_0=\bigl\{\{j\}:j\in\mathbb Z\bigr\}$.
    
    \begin{step}[1]
        The conclusions \itemref{it:geometry-i}--\itemref{it:geometry-iv} hold for $k=1$.
    \end{step}
    \begin{stepproof}[Proof of Step 1]
        Write $\mathcal M_0=\{t_j^{(0)}:j\in\mathbb Z\}$ with $t_j^{(0)}=j$. We first prove \itemref{it:geometry-i}. We claim that if $V_i^{(1)}=V_j^{(1)}$ with $i\neq j$, then $|t_i^{(0)}-t_j^{(0)}|>4L$. Indeed, if $|t_i^{(0)}-t_j^{(0)}|\leq4L$, then $|i-j|=|t_i^{(0)}-t_j^{(0)}|\leq4L$, and hence $V_i^{(1)}\neq V_j^{(1)}$ by \eqref{eq:colouring}. Let $u\neq v$ belong to $\mathcal M_1$. If they are selected at the same stage of \eqref{eq:marker-selection}, they have the same colour, so the claim gives $|u-v|>4L$. If they are selected at different stages, the marker selected later is added only if its distance from every previously selected marker is greater than $4L$. Thus $\mathcal M_1$ is $4L$-separated, proving \itemref{it:geometry-i}. We next prove \itemref{it:geometry-ii}. Take $t_j^{(0)}\in\mathcal M_0$. If $t_j^{(0)}\in\mathcal M_1$, then $\dist(t_j^{(0)},\mathcal M_1)=0$. Otherwise, let $a:=V_j^{(1)}$. Since $t_j^{(0)}\notin A_a$, the definition of $A_a$ gives $u\in A_{a-1}$ with $|t_j^{(0)}-u|\leq4L$. As $A_{a-1}\subset\mathcal M_1$, we obtain
        \begin{align}
            \dist\left(t_j^{(0)},\mathcal M_1\right)\leq4L. \label{eq:marker-maximality-base}
        \end{align}
        This proves \itemref{it:geometry-ii}, and it implies that $\mathcal M_1$ is unbounded in both directions. We now prove \itemref{it:geometry-iii}. Since $\mathcal M_1$ is bi-infinite, its consecutive points determine a partition $\mathcal P_1$ of $\mathbb Z$ into finite intervals. Let $B=[u,v)\cap\mathbb Z\in\mathcal P_1$. By \itemref{it:geometry-i}, $v-u>4L$. Suppose, for a contradiction, that $v-u>2(4L)+3$. Then $(u+4L,v-4L)$ has length greater than $3$ and therefore contains an integer $t\in\mathcal M_0$. Since $u,v$ are consecutive points of $\mathcal M_1$, this gives $\dist(t,\mathcal M_1)>4L$, contradicting \eqref{eq:marker-maximality-base}. Thus
        $$
        4L<|B|\leq2(4L)+3\leq3(4L),
        $$
        which proves \itemref{it:geometry-iii}. Finally, the level-$0$ blocks are singletons, so $N_1(B)=|B|>4L>4L/3$, proving \itemref{it:geometry-iv}.
    \end{stepproof}
    
    \begin{step}[2]
        Assume the conclusions at level $k-1$, with $k\geq2$. Then \itemref{it:geometry-i}--\itemref{it:geometry-ii} hold at level $k$.
    \end{step}
    \begin{stepproof}[Proof of Step 2]
        Write the increasing enumeration of $\mathcal M_{k-1}$ as $(t_j^{(k-1)})_{j\in\mathbb Z}$. We first claim that if $V_i^{(k)}=V_j^{(k)}$ with $i\neq j$, then $|t_i^{(k-1)}-t_j^{(k-1)}|>(4L)^k$. Indeed, suppose that $|t_i^{(k-1)}-t_j^{(k-1)}|\leq(4L)^k$. By \itemref{it:geometry-iii} at level $k-1$, every level-$(k-1)$ block has length greater than $(4L)^{k-1}$, so
        $$
        (4L)^k\geq\left|t_i^{(k-1)}-t_j^{(k-1)}\right|>|i-j|(4L)^{k-1},
        $$
        whence $|i-j|<4L$ and, by \eqref{eq:colouring}, $V_i^{(k)}\neq V_j^{(k)}$. Let $u\neq v$ belong to $\mathcal M_k$. If they are selected at the same stage, the claim gives $|u-v|>(4L)^k$. If they are selected at different stages, the marker selected later is retained only if its distance from every previously selected marker is greater than $(4L)^k$. This proves \itemref{it:geometry-i}. For \itemref{it:geometry-ii}, take $t_j^{(k-1)}\in\mathcal M_{k-1}$. If $t_j^{(k-1)}\in\mathcal M_k$ there is nothing to prove. Otherwise, let $a:=V_j^{(k)}$. Since $t_j^{(k-1)}\notin A_a$, \eqref{eq:marker-selection} gives $u\in A_{a-1}$ with $|t_j^{(k-1)}-u|\leq(4L)^k$. As $A_{a-1}\subset\mathcal M_k$, we obtain
        \begin{align}
            \dist\left(t_j^{(k-1)},\mathcal M_k\right)\leq(4L)^k. \label{eq:marker-maximality}
        \end{align}
        In particular $\mathcal M_k$ is maximal among the $(4L)^k$-separated subsets of $\mathcal M_{k-1}$, and is unbounded in both directions.
    \end{stepproof}
    
    \begin{step}[3]
        Under the same induction hypothesis, \itemref{it:geometry-iii}--\itemref{it:geometry-iv} hold at level $k$.
    \end{step}
    \begin{stepproof}[Proof of Step 3]
        Since $\mathcal M_k$ is bi-infinite, its consecutive points partition $\mathbb Z$ into finite intervals, so $\mathcal P_k$ is a partition. Let $B=[u,v)\cap\mathbb Z\in\mathcal P_k$. From  \itemref{it:geometry-i}, $v-u>(4L)^k$. Suppose, for a contradiction, that
        $$
        v-u>2(4L)^k+3(4L)^{k-1}.
        $$
        Then $\bigl(u+(4L)^k,v-(4L)^k\bigr)$ has length greater than $3(4L)^{k-1}$. By \itemref{it:geometry-iii} at level $k-1$, every level-$(k-1)$ block has length at most $3(4L)^{k-1}$, so this interval contains some $t\in\mathcal M_{k-1}$. As $u$ and $v$ are consecutive points of $\mathcal M_k$, this gives $\dist(t,\mathcal M_k)>(4L)^k$, contradicting \eqref{eq:marker-maximality}. Therefore $v-u\leq2(4L)^k+3(4L)^{k-1}\leq3(4L)^k$, which proves \itemref{it:geometry-iii}. For \itemref{it:geometry-iv}, write $u=t_{j_0}^{(k-1)}<t_{j_0+1}^{(k-1)}<\cdots<t_{j_1}^{(k-1)}=v$ for the level-$(k-1)$ markers between $u$ and $v$, which exist because $\mathcal M_k\subset\mathcal M_{k-1}$. Then $B$ is the union of the $N_k(B)=j_1-j_0$ consecutive level-$(k-1)$ blocks $[t_j^{(k-1)},t_{j+1}^{(k-1)})\cap\mathbb Z$. Each has length at most $3(4L)^{k-1}$ by \itemref{it:geometry-iii} at level $k-1$, so $|B|\leq3(4L)^{k-1}N_k(B)$, and the lower bound in \itemref{it:geometry-iii} gives
        $$
        N_k(B)>\frac{(4L)^k}{3(4L)^{k-1}}=\frac{4L}{3}.
        $$
    \end{stepproof}
    Steps~2--3 complete the induction.
\end{proof}

\subsection{Stationarity of the nested marker sets}

The following lemma shows that, for every $k\geq0$, the process $(\mathcal H_i^{(k)})_{i\in\mathbb Z}$ is stationary.

\begin{lemma}
    \label{lem:stationarity-geometry}
    For every $k\geq0$, $p\geq1$ and $i_1,\ldots,i_p,h\in\mathbb Z$,
    \begin{align}
        \Law\bigl( \mathcal H_{i_1+h}^{(k)},\ldots,\mathcal H_{i_p+h}^{(k)} \bigr) = \Law\bigl( \mathcal H_{i_1}^{(k)},\ldots,\mathcal H_{i_p}^{(k)} \bigr). \label{eq:hierarchy-stationarity}
    \end{align}
\end{lemma}

\begin{proof}
    We prove by induction on $k$ that
    \begin{align}
        \bigl(\mathcal M_0+h,\ldots,\mathcal M_k+h\bigr) \stackrel{d}{=} \bigl(\mathcal M_0,\ldots,\mathcal M_k\bigr) \quad\text{for every }h\in\mathbb Z. \label{eq:marker-hierarchy-stationarity}
    \end{align}
    The case $k=0$ is immediate since $\mathcal M_0=\mathbb Z$. Assume the assertion through level $k-1$ and fix $h\in\mathbb Z$. If $(t_j^{(k-1)})_{j\in\mathbb Z}$ is the increasing enumeration of $\mathcal M_{k-1}$ anchored at zero, then $(t_j^{(k-1)}+h)_{j\in\mathbb Z}$ enumerates $\mathcal M_{k-1}+h$, but need not be anchored at zero. The anchored enumeration of the translated set has the form $(t_{j+r}^{(k-1)}+h)_{j\in\mathbb Z}$, where $r$ is determined by $\mathcal M_{k-1}$ and $h$. Conditional on $\mathcal M_0,\ldots,\mathcal M_{k-1}$, the integer $r$ is fixed. Independence and stationarity of $V^{(k)}$ give
    $$
    \Law\bigl((V_{i+r}^{(k)})_{i\in\mathbb Z}\mid\mathcal M_0,\ldots,\mathcal M_{k-1}\bigr)=\Law\bigl((V_i)_{i\in\mathbb Z}\bigr).
    $$
    In this way, the conditional law of the colouring on the translated markers is the same whether we use the anchored enumeration or $(t_j^{(k-1)}+h)_{j\in\mathbb Z}$. We may therefore compare the selections from $(t_j^{(k-1)},V_j^{(k)})$ and $(t_j^{(k-1)}+h,V_j^{(k)})$ on the same probability space. Let $A_0,\ldots,A_q$ be obtained from $(t_j^{(k-1)},V_j^{(k)})$ by \eqref{eq:marker-selection}, and let $A_0',\ldots,A_q'$ be obtained from $(t_j^{(k-1)}+h,V_j^{(k)})$. We claim that
    \begin{align}
        A_a'=A_a+h \quad\text{for each }0\leq a\leq q. \label{eq:selection-translation}
    \end{align}
    For $a=0$ this is immediate. If $A_{a-1}'=A_{a-1}+h$, then, since $\dist(\cdot,\cdot)$ is translation invariant,
    $$
    A_a' = \left( A_{a-1} \cup \left\{ t_j^{(k-1)}: V_j^{(k)}=a,\ \dist\left(t_j^{(k-1)},A_{a-1}\right)>(4L)^k \right\} \right)+h = A_a+h.
    $$
    Thus \eqref{eq:selection-translation} holds for every $a$, and in particular $\mathcal M_k+h$ is the marker set produced from the translated old markers with the same colours. Since $V^{(k)}$ is independent of the previous levels and $(\mathcal M_0,\ldots,\mathcal M_{k-1})$ has the same law after translation, \eqref{eq:marker-hierarchy-stationarity} follows. Finally, $\mathcal H_i^{(k)}$ is determined by the memberships $\one_{\{i\in\mathcal M_j\}}$, $j\leq k$, so \eqref{eq:marker-hierarchy-stationarity} gives \eqref{eq:hierarchy-stationarity}.
\end{proof}

\subsection{Finite dependence of the marker sets at each scale}
Recall that $I^{+r}$ denotes the set of integers at distance at most $r$ from $I$, as defined in \eqref{eq:interval-enlargement}. The following lemma shows that $\mathcal M_k\cap I$ is determined by $\mathcal M_{k-1}\cap I^{+q(4L)^k}$ and the colours attached to those markers.

\begin{lemma}
    \label{lem:local-selection}
    Fix $k\geq1$ and an interval $I\subset\mathbb Z$. Consider two level-$(k-1)$ marker sets $\mathcal M_{k-1}$ and $\widetilde{\mathcal M}_{k-1}$, with increasing enumerations
    $$
    \mathcal M_{k-1}=\left\{t_j^{(k-1)}:j\in\mathbb Z\right\}\quad\text{and}\quad \widetilde{\mathcal M}_{k-1}=\left\{\widetilde t_j^{\,(k-1)}:j\in\mathbb Z\right\},
    $$
    and attach to their markers the colours $(V_i^{(k)})_{i\in\mathbb Z}$ and $(\widetilde V_i^{(k)})_{i\in\mathbb Z}$ respectively. Let $\mathcal M_k$ and $\widetilde{\mathcal M}_k$ be the marker sets produced from these two configurations by \eqref{eq:marker-selection}. If
    \begin{align}
        \left\{\left(t_j^{(k-1)},V_j^{(k)}\right):t_j^{(k-1)}\in I^{+q(4L)^k}\right\} = \left\{\left(\widetilde t_j^{\,(k-1)},\widetilde V_j^{(k)}\right):\widetilde t_j^{\,(k-1)}\in I^{+q(4L)^k}\right\}, \label{eq:local-selection-assumption}
    \end{align}
    then $\mathcal M_k\cap I=\widetilde{\mathcal M}_k\cap I$.
\end{lemma}
\begin{proof}
    We prove by induction on $a=0,\ldots,q$ that
    \begin{align}
        A_a\cap I^{+(q-a)(4L)^k} = \widetilde A_a\cap I^{+(q-a)(4L)^k}. \label{eq:local-selection-induction}
    \end{align}
    For $a=0$ this follows from $A_0=\widetilde A_0=\varnothing$. Fix $1\leq a\leq q$ and assume \eqref{eq:local-selection-induction} with $a-1$ in place of $a$. Take $t\in A_a\cap I^{+(q-a)(4L)^k}$. If $t\in A_{a-1}$, then $t\in I^{+(q-a+1)(4L)^k}$ and the induction hypothesis gives $t\in\widetilde A_{a-1}\subset\widetilde A_a$. Suppose instead that $t\notin A_{a-1}$. By the definition of $A_a$ there is $j$ with $t=t_j^{(k-1)}$, $V_j^{(k)}=a$ and $\dist(t,A_{a-1})>(4L)^k$. Since $t\in I^{+q(4L)^k}$, assumption \eqref{eq:local-selection-assumption} provides $m$ with $\widetilde t_m^{(k-1)}=t$ and $\widetilde V_m^{(k)}=a$. We claim that $\dist(t,\widetilde A_{a-1})>(4L)^k$. Otherwise some $u\in\widetilde A_{a-1}$ satisfies $|u-t|\leq(4L)^k$, hence $u\in I^{+(q-a+1)(4L)^k}$, and the induction hypothesis puts $u\in A_{a-1}$, contradicting $\dist(t,A_{a-1})>(4L)^k$. The definition of $\widetilde A_a$ therefore gives $t\in\widetilde A_a$. This proves one inclusion in \eqref{eq:local-selection-induction}. The reverse is obtained by interchanging the two configurations. Taking $a=q$ gives $\mathcal M_k\cap I=\widetilde{\mathcal M}_k\cap I$.
\end{proof}

The following lemma is a standard consequence of the properties of conditional independence~\cite[Theorems~8.9 and~8.12]{Kallenberg}. We state and prove it in the form needed below.

\begin{lemma}
    \label{lem:local-kernels}
    Let $\mathcal H$ be a random element in a measurable space, and let $H_1,\ldots,H_s,X_1,\ldots,X_s$ be random elements taking values in finite sets. Suppose that $H_1,\ldots,H_s$ are jointly independent and measurable with respect to $\sigma(\mathcal H)$, and that $X_1,\ldots,X_s$ are conditionally independent given $\mathcal H$, with $\Law(X_i\mid\mathcal H)=\Law(X_i\mid H_i)$ almost surely for every $1\leq i\leq s.$ Then the pairs $(H_1,X_1),\ldots,(H_s,X_s)$ are jointly independent. In particular, $X_1,\ldots,X_s$ are jointly independent.
\end{lemma}

\begin{proof}
    We use the following standard properties of conditional independence.
    \begin{enumerate}[label=\textup{(P\arabic*)},ref=P\arabic*]
        \item
        \label{it:local-kernels-characterisation}
        $U$ and $V$ are conditionally independent given $W$ if and only if
        $$
        \Law(U\mid V,W)=\Law(U\mid W)
        $$
        almost surely (see~\cite[Theorem~8.9]{Kallenberg}).
        
        \item
        \label{it:local-kernels-chain-rule}
        If $U$ is independent of $V$, and $U$ and $W$ are conditionally independent given $V$, then $U$ is independent of $(V,W)$ (see~\cite[Theorem~8.12]{Kallenberg}).
    \end{enumerate}
    
    The case $s=1$ is immediate. Fix $2\leq i\leq s$ and write
    $$
    H_{<i}:=(H_1,\ldots,H_{i-1}) \quad\text{and}\quad X_{<i}:=(X_1,\ldots,X_{i-1}).
    $$
    Conditional independence given $\mathcal H$ and the hypothesis on the conditional laws imply
    $$
    \Law(X_{<i}\mid\mathcal H) =\bigotimes_{j<i}\Law(X_j\mid H_j).
    $$
    The right-hand side is measurable with respect to $\sigma(H_{<i})$, so the tower property gives
    $$
    \Law(X_{<i}\mid H_{<i},H_i) =\bigotimes_{j<i}\Law(X_j\mid H_j) =\Law(X_{<i}\mid H_{<i}).
    $$
    By \itemref{it:local-kernels-characterisation}, $H_i$ and $X_{<i}$ are conditionally independent given $H_{<i}$. Since $H_i$ is independent of $H_{<i}$, \itemref{it:local-kernels-chain-rule} shows that $H_i$ is independent of $(H_{<i},X_{<i})$.
    
    Next, conditional independence given $\mathcal H$ and \itemref{it:local-kernels-characterisation} yield
    $$
    \Law(X_i\mid\mathcal H,X_{<i}) =\Law(X_i\mid\mathcal H) =\Law(X_i\mid H_i).
    $$
    Since $H_i$ and $H_{<i}$ are measurable with respect to $\sigma(\mathcal H)$, conditioning this identity on $(H_i,H_{<i},X_{<i})$ gives
    $$
    \Law(X_i\mid H_i,H_{<i},X_{<i}) =\Law(X_i\mid H_i).
    $$
    Thus \itemref{it:local-kernels-characterisation} shows that $X_i$ and $(H_{<i},X_{<i})$ are conditionally independent given $H_i$. Together with the independence established above, \itemref{it:local-kernels-chain-rule} implies that $(H_i,X_i)$ is independent of $(H_{<i},X_{<i})$.
    
    Hence each pair $(H_i,X_i)$ is independent of all preceding pairs jointly. This proves joint independence of the pairs, and therefore of their second components.
\end{proof}

The following proposition bounds the dependence range of $(\mathcal H_i^{(k)})_{i\in\mathbb Z}$ by $4q(4L)^k$ for each $k\geq0$.

\begin{proposition}
    \label{prop:hierarchy-dependence}
    Let $q=q(L)$ be as in Theorem~\ref{thm:colouring}. For every $k\geq0$ and all finite intervals $I,J\subset\mathbb Z$ with $\dist(I,J)>4q(4L)^k$, the random vectors
    $$
    \bigl(\mathcal H_i^{(k)}\bigr)_{i\in I} \quad\text{and}\quad \bigl(\mathcal H_j^{(k)}\bigr)_{j\in J}
    $$
    are independent. In particular, $\mathcal H^{(k)}$ is $4q(4L)^k$-dependent.
\end{proposition}

\begin{proof}
    We argue by induction. For $k=0$ the marker configuration is deterministic. Suppose the assertion holds at level $k-1$, and let $\dist(I,J)>4q(4L)^k$. Set
    $$
    I':=I^{+q(4L)^k}\quad\text{and}\quad J':=J^{+q(4L)^k}.
    $$
    Their distance exceeds $2q(4L)^k>4q(4L)^{k-1}$, so the old restrictions $\mathcal H^{(k-1)}[I']$ and $\mathcal H^{(k-1)}[J']$ are independent.
    
    Enumerate $\mathcal M_{k-1}$ as $(t_m^{(k-1)})_{m\in\mathbb Z}$ and put $K_I:=\{m:t_m^{(k-1)}\in I'\}$ and $K_J:=\{m:t_m^{(k-1)}\in J'\}$. For $m\in K_I$ and $n\in K_J$, $|t_m^{(k-1)}-t_n^{(k-1)}|>2q(4L)^k\geq4(4L)^k$. If $|m-n|\leq4L$, the interval between $t_m^{(k-1)}$ and $t_n^{(k-1)}$ would consist of at most $4L$ blocks in $\mathcal P_{k-1}$. Each has length at most $3(4L)^{k-1}$, including the singleton case $k=1$, so
    $$
    \left|t_m^{(k-1)}-t_n^{(k-1)}\right|\leq |m-n|\,3(4L)^{k-1} \leq3(4L)^k,
    $$
    a contradiction. Thus $\dist(K_I,K_J)>4L$.
    
    The colouring $V^{(k)}$ is independent of all marker sets constructed at the previous levels. Conditioned on those markers, its law is therefore unchanged, and its finite dependence makes the two colour vectors independent. List each vector in increasing order, without its absolute enumeration indices. Each $K_I$ is an interval of indices. Stationarity of $V^{(k)}$ therefore makes its conditional law depend only on the number of markers in $I'$, which is determined by $\mathcal H^{(k-1)}[I']$. The same holds for $J'$. Lemma~\ref{lem:local-kernels} makes the two pairs consisting of an old local configuration and its colour vector independent. By Lemma~\ref{lem:local-selection}, the new configurations on $I$ and $J$ are functions of the respective pairs. This proves the induction step.

To extend the conclusion to arbitrary index sets, first observe that restrictions to finitely many ordered intervals with successive distances greater than $4q(4L)^k$ are jointly independent. Indeed, each interval is independent of the convex hull of all preceding intervals, so induction gives joint independence. Now let $E,F\subset\mathbb Z$ be finite sets with $\operatorname{dist}(E,F)>4q(4L)^k$. Order the points of $E\cup F$ and group consecutive points according to whether they belong to $E$ or to $F$. The convex hulls of these groups are ordered intervals with successive distances greater than $4q(4L)^k$. Their restrictions are therefore jointly independent, which implies independence of the restrictions to $E$ and $F$. For infinite index sets, the conclusion follows because cylinder events generate the corresponding $\sigma$-algebras.
\end{proof}

The local dependence of the conditional laws is essential in Lemma~\ref{lem:local-kernels}, since conditional independence alone would not justify averaging.

\begin{remark}
    \label{rem:enumeration}
    The colouring is attached to the indices $j$ in the enumeration $(t_j^{(k-1)})$, not to the positions $t_j^{(k-1)}\in\mathbb Z$. The block-length lower bound gives
    $$
    0<\left|t_i^{(k-1)}-t_j^{(k-1)}\right|\leq(4L)^k\quad\Longrightarrow\quad0<|i-j|\leq4L.
    $$
    Thus the same colouring, with a fixed finite number $q=q(L)$ of colours, can be used at every level. Lemma~\ref{lem:local-selection} then requires information only within distance $q(4L)^k$, and Proposition~\ref{prop:hierarchy-dependence} gives dependence range $4q(4L)^k$. Both distances are fixed multiples of the block scale. This is the estimate used later to choose a cylinder radius proportional to the separation of the observation times.
\end{remark}

\begin{proof}[Proof of Proposition~\ref{prop:hierarchy}]
    Take the nested marker sets $\mathcal H=(\mathcal M_k)_{k\geq0}$ constructed in Section~\ref{subsec:markers}. Lemma~\ref{lem:geometry} gives \itemref{it:H-lengths} and \itemref{it:H-children}, and Lemma~\ref{lem:stationarity-geometry} gives \itemref{it:H-stationary}. Proposition~\ref{prop:hierarchy-dependence} gives \itemref{it:H-dependence}, with $D:=4q$.
\end{proof}

\section{A finite-dimensional biasing lemma and the block laws}
\label{sec:block-laws}
Throughout this section and Sections~\ref{sec:limit}--\ref{sec:moment}, we work with the random nested marker sets $\mathcal H$ of Proposition~\ref{prop:hierarchy}, with law $\mathbb P_{\mathcal H}$. We use only \itemref{it:H-lengths}--\itemref{it:H-dependence}. For each block $B$ of a fixed realisation of $\mathcal H$, we construct a law $\rho_B$ on its coordinates. These laws preserve every marginal on at most $L-1$ children (Lemma~\ref{lem:block-properties}) and produce negative $L$-point correlations across the first $L$ children (Lemma~\ref{lem:block-correlations}). The recursive construction follows the principle in~\cite[Lemma~2.8]{BradleyPruss} of introducing higher-order dependence while preserving lower-order marginals.

The nested partitions are now fixed independently of the coordinate variables. The children of a parent may have different lengths and may be subdivided differently at lower levels. We therefore need a function positively correlated with every coordinate, even when the child laws are not exchangeable.

The following notation identifies a block law with a finite-dimensional law, allowing Lemma~\ref{lem:bias} to be applied to children with different lengths and different laws.
\begin{notation}[Coordinates on a finite block]
For a finite block $B\subset\mathbb Z$, write
$$
\{-1,1\}^B:=\{x=(x_j)_{j\in B}:x_j\in\{-1,1\}\}.
$$
For each $j\in B$, define the coordinate map
$$
X_j:\{-1,1\}^B\to\{-1,1\}\quad\text{with}\quad X_j(x):=x_j.
$$
Thus, whenever $\rho$ is a probability measure on $\{-1,1\}^B$, we regard $ X=(X_j)_{j\in B}$ as the corresponding random vector under $\rho$. For $C\subset B$, write
$$
x[C]:=(x_j)_{j\in C}\quad\text{and}\quad X[C]:=(X_j)_{j\in C}.
$$
We identify $\{-1,1\}^B$ with $\{-1,1\}^{|B|}$ by ordering the coordinates of $B$ from left to right.
\end{notation}

\subsection{Construction of the auxiliary function $\psi_\rho$}

The following lemma provides the functions used to introduce the dependence between blocks that leads to failure of the CLT.

\begin{lemma}
    \label{lem:bias}
    Let $\rho$ be a probability measure on $\{-1,1\}^d$, invariant under the global sign reversal $x\mapsto-x$, and let $X=(X_1,\ldots,X_d)$ be a random vector with law $\rho$. Suppose that
    \begin{align}
        \E_\rho(a\cdot X)^2=\|a\|_2^2\quad\text{and}\quad \E_\rho(a\cdot X)^4\le3\|a\|_2^4 \quad\text{for every }a\in\R^d. \label{eq:linear-form-moments}
    \end{align}
    There exists an odd function $\psi_\rho:\{-1,1\}^d\to[-1,1]$ such that
    \begin{align}
        \E_\rho[X_j\psi_\rho(X)]=\frac1{\sqrt{3d}}, \quad\text{for each }1\le j\le d. \label{eq:coordinate-bias}
    \end{align}
    In particular $\E_\rho\psi_\rho=0$. Both conditions in \eqref{eq:linear-form-moments} hold when the coordinates are four-tuplewise independent and satisfy $\rho(X_j=-1)=\rho(X_j=1)=1/2$ for every $1\leq j\leq d$, since then $\E_\rho(a\cdot X)^4=3\|a\|_2^4-2\sum_ia_i^4$.
\end{lemma}
\begin{proof}
    Fix $a\in\R^d\setminus\{0\}$ and set $Y:=a\cdot X$. The case $a=0$ is immediate. Writing $Y^2=|Y|^{2/3}\cdot|Y|^{4/3}$ and applying H\"older's inequality with the conjugate exponents $3/2$ and $3$ gives
    \begin{align}
        \E_\rho Y^2\leq\bigl(\E_\rho|Y|\bigr)^{2/3}\bigl(\E_\rho Y^4\bigr)^{1/3},\label{eq:ey2}
    \end{align}
    since $\bigl\||Y|^{2/3}\bigr\|_{3/2}=(\E_\rho|Y|)^{2/3}$ and $\bigl\||Y|^{4/3}\bigr\|_{3}=(\E_\rho Y^4)^{1/3}$. Raising both sides of \eqref{eq:ey2} to the power $3/2$ gives $\bigl(\E_\rho Y^2\bigr)^{3/2}\leq\E_\rho|Y|\,\bigl(\E_\rho Y^4\bigr)^{1/2}$. Thus
    \begin{align}
        \E_\rho|a\cdot X| \geq\frac{\bigl(\E_\rho(a\cdot X)^2\bigr)^{3/2}}{\bigl(\E_\rho(a\cdot X)^4\bigr)^{1/2}} \geq\frac{\|a\|_2^3}{\sqrt3\,\|a\|_2^2} =\frac{\|a\|_2}{\sqrt3}, \label{eq:linear-form-lower-bound}
    \end{align}
    where the second inequality uses \eqref{eq:linear-form-moments}. The bound also holds for $a=0$.
    
    Consider the set $\mathcal K:= \bigl\{\E_\rho[X\psi(X)]:\psi:\{-1,1\}^d\to[-1,1]\bigr\}\subset\R^d,$ which is clearly convex and compact. We claim that $\mathcal K$ contains every $v\in\R^d$ with $\|v\|_2\leq1/\sqrt3$. If some such $v$ were not in $\mathcal K$, then, by the Hahn--Banach separation theorem, there would exist $a\in\mathbb R^d$ such that $ a\cdot v>\sup_{w\in\mathcal K}a\cdot w.$ But the choice $\psi=\operatorname{sgn}(a\cdot X)$ produces the point $w:=\E_\rho[X\operatorname{sgn}(a\cdot X)]$ of $\mathcal K$, and for it
    $$
    a\cdot w=\E_\rho\bigl[(a\cdot X)\operatorname{sgn}(a\cdot X)\bigr]=\E_\rho|a\cdot X| \geq\frac{\|a\|_2}{\sqrt3}\geq\|a\|_2\|v\|_2\geq a\cdot v,
    $$
    by \eqref{eq:linear-form-lower-bound} and the Cauchy--Schwarz inequality, a contradiction. The vector $v:=(1,\ldots,1)/\sqrt{3d}$ satisfies $\|v\|_2=1/\sqrt3$, so $v\in\mathcal K$. Thus there is $\psi$ with values in $[-1,1]$ such that $\E_\rho[X_j\psi(X)]=1/\sqrt{3d}$ for every $j$. Replacing $\psi$ by its odd part $x\mapsto(\psi(x)-\psi(-x))/2$, which again takes values in $[-1,1]$, leaves these expectations unchanged, because $X$ and $-X$ have the same law.
\end{proof}

\begin{remark}
The convex set $\mathcal K$ used in the proof is a scalar multiple of a zonoid~\cite[Section~2]{KoshevoyMosler1998}. A zonoid is a convex set associated with a probability distribution, obtained by collecting weighted averages of the values of a random vector. For a random vector $X$ with law $\rho$, it is defined by $\mathcal Z(\rho)=\bigl\{\E_\rho[Xg(X)]:g:\mathbb R^d\to[0,1]\text{ measurable}\bigr\}.$ In our proof, $\mathcal K$ collects all vectors of correlations with the coordinates that can be obtained from a single function $\psi$ taking values in $[-1,1]$. Since $\E_\rho X=0$, the substitution $\psi=2g-1$ gives $\mathcal K=2\mathcal Z(\rho)$.
\end{remark}

The following lemma introduces dependence between factors while preserving every marginal on at most $L-1$ factors. It uses a product of centred functions to perturb a product measure, a construction used in the Sarmanov family of distributions~\cite{Sarmanov1966,LeeSarmanov1996}.

\begin{lemma}
    \label{lem:product-perturbation}
    Let $\rho_1,\ldots,\rho_r$ be probability measures, with $r\geq L$, and let $\psi_i$ satisfy $|\psi_i|\leq1$ and $\rho_i(\psi_i)=0$ for $1\leq i\leq L$. Then
    $$
    d\rho=\left(1-\frac12\prod_{i=1}^L\psi_i(x_i)\right) \,d(\rho_1\otimes\cdots\otimes\rho_r)
    $$
    defines a probability measure with the same null sets as the product measure. For bounded measurable $F_i$,
    \begin{align}
        \rho\left(\prod_{i=1}^rF_i(x_i)\right) &=\prod_{i=1}^r\rho_i(F_i) -\frac12\prod_{i=1}^L\rho_i(F_i\psi_i) \prod_{i=L+1}^r\rho_i(F_i). \label{eq:product-perturbation}
    \end{align}
    Any collection of at most $L-1$ coordinates is independent under $\rho$, with the original marginal laws. That is,
    $$
    \rho(x_i\in A_i\text{ for all }i\in J) = \prod_{i\in J}\rho_i(A_i)
    $$
    for every nonempty $J\subset\{1,\ldots,r\}$ with $|J|\leq L-1$ and measurable sets $A_i$.
\end{lemma}
\begin{proof}
    Write $\rho_0:=\rho_1\otimes\cdots\otimes\rho_r$. Since $|\psi_i|\leq1$, the proposed density lies between $1/2$ and $3/2$. Since $\rho_i(\psi_i)=0$ for each $1\leq i\leq L$ we obtain that
    $$
    \int\left(1-\frac12\prod_{i=1}^L\psi_i(x_i)\right) \,d\rho_0 =1-\frac12\prod_{i=1}^L\rho_i(\psi_i)=1.
    $$
    Thus $\rho$ is a probability measure, and the positive lower bound on the density gives equivalence with $\rho_0$.
    
    To obtain \eqref{eq:product-perturbation}, observe that Fubini's theorem gives
    \begin{align}
        \int\prod_{i=1}^rF_i(x_i)\,d\rho_0 =\prod_{i=1}^r\rho_i(F_i)\ \text{and}\ \int\left(\prod_{i=1}^rF_i(x_i)\right) \left(\prod_{i=1}^L\psi_i(x_i)\right)\,d\rho_0 =\prod_{i=1}^L\rho_i(F_i\psi_i) \prod_{i=L+1}^r\rho_i(F_i).\label{eq:prod2}
    \end{align}
    Subtracting half the second term from the first, we obtain \eqref{eq:product-perturbation}.
    
    Fix a nonempty $J\subset\{1,\ldots,r\}$ with $|J|\leq L-1$, and choose measurable sets $A_i$ for $i\in J$. In \eqref{eq:product-perturbation}, take $F_i=\mathbf{1}_{A_i}$ for $i\in J$ and $F_i=1$ otherwise. Since $J$ contains at most $L-1$ indices, there exists $s\in\{1,\ldots,L\}\setminus J$. For such an index, $F_s=1$, hence $\rho_s(F_s\psi_s)=\rho_s(\psi_s)=0.$ Thus the second term in \eqref{eq:product-perturbation} vanishes. The remaining term gives
    $$
    \rho(x_i\in A_i\text{ for all }i\in J) = \prod_{i\in J}\rho_i(A_i),
    $$
\end{proof}

\subsection{Recursive definition of the block distributions}

We now define the block laws recursively along the subdivision tree of each block.

Fix a realisation $h$ of the nested marker sets $\mathcal H$ given by Proposition~\ref{prop:hierarchy}, and write $\mathcal P_k(h)$ for its level-$k$ block partition.

For each $B\in\mathcal P_k(h)$, consider the following tree. Its root is $B$. Each vertex $C\in\mathcal P_\ell(h)$ with $\ell\geq1$ is joined to the blocks of $\mathcal P_{\ell-1}(h)$ contained in $C$. These are its children, ordered from left to right. The subdivision stops at level $0$, where each block consists of a single index $j\in B$ (see Figure~\ref{fig:block-law-tree}).

We now define, recursively over the levels, a probability measure $\rho_B$ on the coordinates of every block $B\in\mathcal P_k(h)$. At level $0$, every $B=\{j\}\in\mathcal P_0(h)$ receives the uniform law on $\{-1,1\}$.

Suppose that the laws have been defined for all blocks in $\mathcal P_{k-1}(h)$. Let $B\in\mathcal P_k(h)$ have consecutive children $B_1,\ldots,B_r\in\mathcal P_{k-1}(h)$. For each $1\leq i\leq r$, let $\psi_i:=\psi_{\rho_{B_i}}$ be a function given by Lemma~\ref{lem:bias}. Define
\begin{align}
    \rho_B(dx):= \left(1-\frac12\prod_{i=1}^L\psi_i(x[B_i])\right) \bigotimes_{i=1}^r\rho_{B_i}(dx[B_i]). \label{eq:parent-law}
\end{align}
There are at least $L$ children by \itemref{it:H-children}, and the density in \eqref{eq:parent-law} lies between $1/2$ and $3/2$.

\begin{figure}[htbp]
    \centering
    \begin{tikzpicture}[bl base,x=1cm,y=1cm]
        \path[use as bounding box]
        (-.2,-.25) rectangle (13.9,8);
        
        \node[anchor=east] at (1.2,7.5) {$\mathcal P_2(h)$};
        \node[anchor=east] at (1.2,6.5) {$\mathcal P_1(h)$};
        \node[anchor=east] at (1.2,5.5) {$\mathcal P_0(h)$};
        
        \path[bl block] (1.6,7.15) rectangle (13.5,7.85);
        \node at (7.55,7.5) {$B$};
        
        \foreach \a/\b/\lab in {
            1.6/4/B_1,
            4/6.4/B_2,
            7.2/9.6/B_i,
            10.4/13.5/B_r
        }{
            \path[bl block]
            (\a,6.15) rectangle (\b,6.85);
            
            \node at ({(\a+\b)/2},6.5) {$\lab$};
            
            \draw[black,dashed] (\a,6.1) -- (\a,5.05);
            \draw[black,dashed] (\b,6.1) -- (\b,5.05);
            
            \foreach \xx in {\a+.2,\a+.6,\b-.6,\b-.2}{
                \path[bl block]
                (\xx-.2,5.2) rectangle (\xx+.2,5.8);
                
                \node[bl dot] at (\xx,5.5) {};
            }
            
            \node at ({(\a+\b)/2},5.5) {$\cdots$};
        }
        
        \foreach \xx in {6.8,10}{
            \node at (\xx,6.5) {$\cdots$};
            \node at (\xx,5.5) {$\cdots$};
        }
        
        \node[bl note,text width=12cm] at (7.55,4.5)
        {At level $0$, the blocks inside $B$ are the singletons
            $\{j\}$ with $j\in B$, ordered from left to right.};
        
        \node[bl node,label=above:{$B$}]
        (root) at (7.55,3.1) {};
        
        \foreach \a/\b/\name/\lab/\side/\offset in {
            1.6/4/c1/B_1/east/-.15,
            4/6.4/c2/B_2/east/-.15,
            7.2/9.6/ci/B_i/west/.15,
            10.4/13.5/cr/B_r/west/.15
        }{
            \node[bl node]
            (\name) at ({(\a+\b)/2},1.95) {};
            
            \node[anchor=\side]
            at ({(\a+\b)/2+\offset},2.12) {$\lab$};
            
            \draw[bl tree] (root) -- (\name);
            
            \foreach \xx [count=\n] in {
                \a+.2,\a+.6,\b-.6,\b-.2
            }{
                \node[bl dot]
                (\name-leaf-\n) at (\xx,.65) {};
                
                \draw[bl tree]
                (\name) -- (\name-leaf-\n);
            }
            
            \node at ({(\a+\b)/2},.65) {$\cdots$};
        }
        
        \foreach \xx in {6.8,10}{
            \node at (\xx,1.95) {$\cdots$};
            \node at (\xx,.65) {$\cdots$};
        }
        
        \node[bl note] at (7.55,0)
        {At each leaf,
            $\rho_{\{j\}}=\tfrac12(\delta_{-1}+\delta_1)$.};
    \end{tikzpicture}

    \caption{A level-$2$ block $B$, its children $B_1,\ldots,B_r$,
        and its singleton leaves at level $0$, shown as nested intervals above and as a tree below.}
    \label{fig:block-law-tree}
\end{figure}

The construction of $\rho_B$ depends only on this tree, including the left-to-right order of its children. There are only countably many finite trees of this form. We fix, once and for all, one admissible function $\psi_\rho$ from Lemma~\ref{lem:bias} for each law $\rho$ arising in the construction, and use the same choice whenever that law occurs again. Thus the choice of $\psi_{\rho_B}$ introduces no additional randomness.

The recursion preserves symmetry and every marginal on at most $L-1$ children, and the coordinates remain $(L-1)$-tuplewise independent. These properties ensure that Lemma~\ref{lem:bias} remains applicable at every level and yield exact stabilisation of the finite-dimensional distributions.

\begin{lemma}
    \label{lem:block-properties}
    For every $k\geq0$ and every $B\in\mathcal P_k$, the law $\rho_B$ satisfies \itemref{it:block-properties-i}--\itemref{it:block-properties-ii} below. For $k\geq1$, write $B_1,\ldots,B_r\in\mathcal P_{k-1}$ for its consecutive children. Then \itemref{it:block-properties-iii} also holds.
    \begin{enumerate}[label=\textup{(\roman*)},ref=\roman*]
        \item\label{it:block-properties-i} $\rho_B$ is a probability measure on $\{-1,1\}^B$ with full support, and
        $$
        \rho_B(\{x\})=\rho_B(\{-x\})\quad\text{for every }x\in\{-1,1\}^B.
        $$
        \item\label{it:block-properties-ii} For every $1\leq m\leq L-1$, every distinct $j_1,\ldots,j_m\in B$, and every $\varepsilon_1,\ldots,\varepsilon_m\in\{-1,1\}$,
        $$
        \rho_B\bigl(X_{j_1}=\varepsilon_1,\ldots,X_{j_m}=\varepsilon_m\bigr)=2^{-m}.
        $$
        \item\label{it:block-properties-iii} For every distinct $i_1,\ldots,i_m\in\{1,\ldots,r\}$ with $1\leq m\leq L-1$ and every $E_\ell\subset\{-1,1\}^{B_{i_\ell}}$,
        $$
        \rho_B\bigl(x[B_{i_1}]\in E_1,\ldots,x[B_{i_m}]\in E_m\bigr)=\prod_{\ell=1}^m\rho_{B_{i_\ell}}(E_\ell).
        $$
    \end{enumerate}
    In particular, for every $1\leq m\leq L-1$ and every distinct $j_1,\ldots,j_m\in B$,
    $$
    \E_{\rho_B}\left[\prod_{\ell=1}^mX_{j_\ell}\right]=0.
    $$
\end{lemma}

\begin{proof}
    We induct on the level of $B$. At level $0$, each singleton block $B=\{p\}$ is assigned the law $\frac12(\delta_{-1}+\delta_1)$, so \itemref{it:block-properties-i}--\itemref{it:block-properties-ii} hold. Now fix $k\geq1$ and let $B\in\mathcal P_k$ have consecutive children $B_1,\ldots,B_r\in\mathcal P_{k-1}$. Assume the conclusions for each child law $\rho_{B_i}$. Since $L-1\geq5$, their coordinates are four-tuplewise independent and symmetric. Together with central symmetry, this verifies the hypotheses of Lemma~\ref{lem:bias}. Hence the functions $\psi_i$ used in \eqref{eq:parent-law} are available and satisfy $\rho_{B_i}(\psi_i)=0$.
    
    Apply Lemma~\ref{lem:product-perturbation} with $\rho_i=\rho_{B_i}$. It gives normalisation and the marginal identity \itemref{it:block-properties-iii}. Each child law has full support, so their product assigns positive mass to every configuration. The density is at least $1/2$, proving full support of $\rho_B$. Under $x\mapsto-x$ each child law is invariant, while
    $$
    \prod_{i=1}^L\psi_i(-x[B_i]) =(-1)^L\prod_{i=1}^L\psi_i(x[B_i]) =\prod_{i=1}^L\psi_i(x[B_i]).
    $$
    Thus the density is also invariant, proving central symmetry and completing \itemref{it:block-properties-i}.
    
    For \itemref{it:block-properties-ii}, fix $E=\{j_1,\ldots,j_m\}$ with $m\leq L-1$. Let $B_{i_1},\ldots,B_{i_s}\in\mathcal P_{k-1}$ be the children of $B\in\mathcal P_k$ intersecting $E$. Then $s\leq m\leq L-1$, so \itemref{it:block-properties-iii} makes these child vectors independent. Within each child, the selected coordinates are independent and symmetric by induction. Therefore, for any prescribed values on $E$,
    $$
    \rho_B(X_j=\varepsilon_j\text{ for }j\in E) =\prod_{a=1}^s2^{-|E\cap B_{i_a}|} =2^{-|E|}.
    $$
    This is \itemref{it:block-properties-ii}. The expectation of a product of distinct selected coordinates is therefore the product of their zero means, proving the final assertion.
\end{proof}

Using Lemmas~\ref{lem:product-perturbation} and~\ref{lem:bias}, the following lemma computes the $L$-point correlations across the children $B_i\in\mathcal P_{k-1}$ of a parent $B\in\mathcal P_k$. These correlations are nonpositive, so their negative contributions accumulate without cancellation. They yield the $L$-th moment deficit used in Corollary~\ref{cor:no-gaussian} to exclude Gaussian weak subsequential limits.

\begin{lemma}
    \label{lem:block-correlations}
    Let $k\geq1$, let $B\in\mathcal P_k$ have consecutive children $B_1,\ldots,B_r\in\mathcal P_{k-1}$, and let $X=(X_j)_{j\in B}$ have law $\rho_B$. Let $j_1<\cdots<j_L$ be indices in $B$. Assume that some child $B_i\in\mathcal P_{k-1}$ contains at least one, but not all, of these indices, that is,
    $$
    1\leq\bigl|B_i\cap\{j_1,\ldots,j_L\}\bigr|\leq L-1\quad\text{for some }i\in\{1,\ldots,r\}.
    $$
    Then
    \begin{align}
        \E_{\rho_B}\left[\prod_{\ell=1}^L X_{j_\ell}\right] =
        \begin{cases}
            \displaystyle-\frac{1}{2\,3^{L/2}\sqrt{|B_1|\cdots|B_L|}}, &\text{if }|B_i\cap\{j_1,\ldots,j_L\}|=1\text{ for every }1\leq i\leq L, \\[2ex]
            0, &\text{otherwise.}
        \end{cases}
        \label{eq:cross-correlation}
    \end{align}
\end{lemma}

\begin{proof}
    Recall that $B\in\mathcal P_k$ has consecutive children $B_1,\ldots,B_r\in\mathcal P_{k-1}$, ordered from left to right. Put $E:=\{j_1,\ldots,j_L\}$ and $F_i:=\prod_{j\in E\cap B_i}X_j$, interpreting an empty product as one. Lemma~\ref{lem:product-perturbation} gives
    \begin{align}
        \E_{\rho_B}\prod_{\ell=1}^L X_{j_\ell} &=\prod_{i=1}^r\rho_{B_i}(F_i) -\frac12\prod_{i=1}^L\rho_{B_i}(F_i\psi_i) \prod_{i=L+1}^r\rho_{B_i}(F_i). \label{eq:block-correlation-expansion}
    \end{align}
    By hypothesis some child $B_s\in\mathcal P_{k-1}$ contains between $1$ and $L-1$ selected indices. Lemma~\ref{lem:block-properties} \itemref{it:block-properties-ii} makes their coordinates independent with zero means, so
    $$
    \rho_{B_s}(F_s) =\prod_{j\in E\cap B_s}\rho_{B_s}(X_j)=0.
    $$
    Hence the independent term in \eqref{eq:block-correlation-expansion} vanishes. Hence,
    \begin{align}
        \E_{\rho_B}\prod_{\ell=1}^L X_{j_\ell} &=-\frac12\prod_{i=1}^L\rho_{B_i}(F_i\psi_i) \prod_{i=L+1}^r\rho_{B_i}(F_i). \label{eq:block-correlation-reduced}
    \end{align}
    
    If one of the first $L$ children $B_i\in\mathcal P_{k-1}$ contains no selected coordinate, its factor in the correction is $\rho_{B_i}(\psi_i)=0$, so the correlation is zero. In the remaining case, each of the blocks $B_1,\ldots,B_L$ contains at least one of the indices $j_1,\ldots,j_L$. There are only $L$ selected coordinates, so each of those children contains exactly one and none lies in a later child. In particular,
    \begin{align}
        E\subset B_1\cup\cdots\cup B_L\quad \text{and}\quad |E\cap B_i|&=1\quad\text{for every }1\leq i\leq L. \label{eq:block-correlation-one-per-child}
    \end{align}
    Since both the children and the selected indices are ordered, $j_i\in B_i$ for $1\leq i\leq L$. Equation~\eqref{eq:coordinate-bias} then gives
    \begin{align}
        \rho_{B_i}(F_i\psi_i) =\rho_{B_i}(X_{j_i}\psi_i)=\frac1{\sqrt{3|B_i|}}. \label{eq:block-correlation-bias-factor}
    \end{align}
    For every later child $B_i\in\mathcal P_{k-1}$ with $i>L$, the factor $\rho_{B_i}(F_i)$ equals one. Substituting into \eqref{eq:block-correlation-expansion} therefore yields
    $$
    \E_{\rho_B}\prod_{\ell=1}^L X_{j_\ell} =-\frac12\prod_{i=1}^L\frac1{\sqrt{3|B_i|}} =-\frac1{2\,3^{L/2}\sqrt{|B_1|\cdots|B_L|}}.
    $$
    This proves both cases of \eqref{eq:cross-correlation}.
\end{proof}

\begin{remark}
    \label{rem:block-correlation-sign}
    The non-zero value in \eqref{eq:cross-correlation} depends only on the lengths of the first $L$ children $B_1,\ldots,B_L\in\mathcal P_{k-1}$, and not on the particular coordinate chosen in each child. Moreover, every $L$-point correlation covered by Lemma~\ref{lem:block-correlations} is nonpositive. Thus, when these correlations are summed over a block, there is no cancellation between non-zero terms.
\end{remark}

\section{Construction of the stationary process}
\label{sec:limit}
For fixed markers, we combine the level-$k$ block laws independently. Exact stabilisation of their finite-dimensional marginals defines a limiting law. Averaging over the markers then gives the stationary process.

Related ideas appear in finitary coding. The markers and skeleta of del Junco~\cite[Sections~2--5]{DelJunco1981} and the cells of Spinka's cell process~\cite[Sections~2 and~4]{Spinka2020} play a role analogous to that of our nested marker sets and interval partitions. Here we use nested interval partitions to construct a probability measure on $\{-1,1\}^{\mathbb Z}$ recursively. The preservation of marginals on at most $L-1$ blocks ensures consistency throughout the construction.
\subsection{Finite-level processes and stabilisation of finite-dimensional distributions}
Fix a realisation $h$ of $\mathcal H=(\mathcal M_k)_{k\geq0}$, and write $\mathcal P_k(h)$ for its level-$k$ block partition. For each $B\in\mathcal P_k(h)$ with $k\geq1$, the law $\rho_B$ was constructed from the laws of its consecutive children $B_1,\ldots,B_r\in\mathcal P_{k-1}(h)$. At level $0$, the law on each singleton block $\{p\}\in\mathcal P_0(h)$ is $\rho_{\{p\}} = \frac12(\delta_{-1}+\delta_1)$. On $\Sigma=\{-1,1\}^{\mathbb Z}$, define
\begin{align}
    \nu_h^{(k)}:=\bigotimes_{B\in\mathcal P_k(h)}\rho_B \quad\text{and}\quad \nu^{(k)}:=\int\nu_h^{(k)}\,\mathbb P_{\mathcal H}(dh). \label{eq:finite-level-law}
\end{align}
Write $Z^{(k)}$ for the coordinate process under $\nu^{(k)}$. For $m\in\mathbb Z$, denote by $h+m$ the configuration obtained by translating every marker set by $m$. Translation preserves the subdivision of each block into children at every level, including their left-to-right order. Thus the laws of the translated blocks are translates of the original laws. Hence,
$$
\nu_{h+m}^{(k)}=(\sigma^{-m})_*\nu_h^{(k)}.
$$
By Lemma~\ref{lem:stationarity-geometry}, the nested marker sets are stationary, so
\begin{align}
    (\mathcal M_0+m,\ldots,\mathcal M_k+m) \stackrel{d}{=}(\mathcal M_0,\ldots,\mathcal M_k) \quad\text{for }m\in\mathbb Z. \label{eq:stationarym}
\end{align}
Averaging gives
$$
(\sigma^{-m})_*\nu^{(k)} =\int\nu_{h+m}^{(k)}\,\mathbb P_{\mathcal H}(dh) =\nu^{(k)},
$$
which proves stationarity.

The following lemma shows that each finite-level process is finitely dependent after averaging over the marker sets, with dependence range bounded by a fixed multiple of the block scale.

\begin{lemma}
    \label{lem:finite-output}
    For every $k\geq0$ and finite interval $I$, the marginal $\nu_h^{(k)}|_I$ depends only on the marker configuration through level $k$ in $I^{+3(4L)^k}$. Moreover, $Z^{(k)}$ is $(D+6)(4L)^k$-dependent.
\end{lemma}
\begin{proof}
    For a realisation $h=(M_j)_{j\geq0}$ of $\mathcal H$, write
    $$
    h_i^{(k)}:=\bigl(\one_{\{i\in M_0\}},\ldots, \one_{\{i\in M_k\}}\bigr)\quad\text{and}\quad h^{(k)}:=(h_i^{(k)})_{i\in\mathbb Z}.
    $$
    Every block $B\in\mathcal P_k(h)$ intersecting $I$ has length at most $3(4L)^k$, including the singleton case $k=0$. Both endpoints therefore lie in $I^{+3(4L)^k}$. The marker indicators through level $k$ in this enlarged interval determine those endpoints and all lower-level markers between them. Thus they determine all subdivisions of each block intersecting $I$, and hence its law. Since
    $$
    \nu_h^{(k)}|_I =\bigotimes_{\substack{B\in\mathcal P_k(h) \\
            B\cap I\neq\varnothing}} \rho_B|_{B\cap I},
    $$
    they also determine the marginal on $I$.
    
    Now suppose $\dist(I,J)>(D+6)(4L)^k$. No block $B\in\mathcal P_k(h)$ can intersect both intervals, because its length would exceed $3(4L)^k$. Under the joint law $\mathbb P_{\mathcal H}(dh)\nu_h^{(k)}(d\omega)$, the coordinate vectors on $I$ and $J$ are therefore independent conditionally on $\mathcal H$. Their respective conditional laws depend only on $\mathcal H^{(k)}[I^{+3(4L)^k}]$ and $\mathcal H^{(k)}[J^{+3(4L)^k}]$, by the locality just proved. These local configurations are independent by \itemref{it:H-dependence}, since
    $$
    \dist\bigl(I^{+3(4L)^k},J^{+3(4L)^k}\bigr) \geq\dist(I,J)-6(4L)^k>D(4L)^k.
    $$
    These are precisely the hypotheses of Lemma~\ref{lem:local-kernels}, with the two local marker configurations as the $H_i$ and the coordinate vectors as the $X_i$. It gives independence under the averaged law $\nu^{(k)}$. For bounded functions $F,G$ of the coordinates on $I,J$, respectively, the factorisation reads
    \begin{align}
        \nu^{(k)}(FG) &=\int\nu_h^{(k)}(F)\nu_h^{(k)}(G)\,\mathbb P_{\mathcal H}(dh)\notag =\left(\int\nu_h^{(k)}(F)\,\mathbb P_{\mathcal H}(dh)\right) \left(\int\nu_h^{(k)}(G)\,\mathbb P_{\mathcal H}(dh)\right)\notag \\=\nu^{(k)}(F)\nu^{(k)}(G). \label{eq:finite-output-factorisation}
    \end{align}
    The extension to arbitrary index sets follows by the same argument as in the proof of Proposition~\ref{prop:hierarchy-dependence}.
\end{proof}
In particular, restrictions to any finite family of ordered intervals with successive distances greater than $(D+6)(4L)^k$ are jointly independent.

For fixed markers, the joint law on at most $L-1$ level-$k$ blocks is unchanged at all higher levels, regardless of their separation. This preservation allows us to define the limiting law and identify its finite-dimensional distributions with those of a finite-level process.

\begin{lemma}
    \label{lem:stabilisation}
    Fix a realisation $h$ of the nested marker sets $\mathcal H$ and integers $\ell\geq k\geq0$. Suppose that $E\subset\mathbb Z$ intersects at most $L-1$ level-$k$ blocks, that is,
    $$
    \#\{B\in\mathcal P_k(h):E\cap B\neq\varnothing\}\leq L-1.
    $$
    Then
    \begin{align}
        \nu_h^{(\ell)}\big|_E = \nu_h^{(k)}\big|_E = \bigotimes_{\substack{B\in\mathcal P_k(h) \\
                E\cap B\neq\varnothing}} \rho_B\big|_{E\cap B}. \label{eq:conditional-stabilisation}
    \end{align}
\end{lemma}
\begin{proof}
    We prove by induction on $\ell-k$ that the marginal on any $t\leq L-1$ distinct level-$k$ blocks $B_1,\ldots,B_t$ is $\bigotimes_{i=1}^t\rho_{B_i}$. For $\ell=k$, this is the definition of $\nu_h^{(k)}$.
    
    Suppose the assertion holds for $\ell-1$, where $\ell>k$. First consider the selected blocks contained in one level-$\ell$ block $B$. Let $C_1,\ldots,C_m\in\mathcal P_{\ell-1}(h)$ be the distinct children of $B$ containing them. Since $m\leq t\leq L-1$, Lemma~\ref{lem:block-properties}\itemref{it:block-properties-iii} gives
    $$
    \rho_B\big|_{C_1\cup\cdots\cup C_m}=\bigotimes_{a=1}^m\rho_{C_a}.
    $$
    Within each $C_a$, the induction hypothesis applies to the selected level-$k$ blocks it contains. Therefore
    $$
    \rho_B\big|_{B_1\cup\cdots\cup B_t}=\bigotimes_{a=1}^m\left(\bigotimes_{\{i:B_i\subset C_a\}}\rho_{B_i}\right)=\bigotimes_{i=1}^t\rho_{B_i}.
    $$
    If the selected blocks belong to different level-$\ell$ blocks, apply the same calculation in each of those blocks and use their independence under $\nu_h^{(\ell)}$. This completes the induction.
    
    Now take the level-$k$ blocks intersecting $E$. There are at most $L-1$ of them by hypothesis, so the identity just proved applies to their union. Taking the marginal on $E$ gives \eqref{eq:conditional-stabilisation}.
\end{proof}
Two cases are used below. Since every level-$k$ block has length at least $(4L)^k$, by \itemref{it:H-lengths} for $k\geq1$ and by the singleton construction for $k=0$, a set contained in an interval of length at most $(4L)^k$ intersects at most two level-$k$ blocks, and a union of $s\le R$ such intervals intersects at most $2s\le2R=L-2$ of them. Lemma~\ref{lem:stabilisation} applies to each fixed marker configuration. Since the block lengths and subdivisions are random, conditional independence does not imply independence after averaging over the markers. Lemma~\ref{lem:finite-output} establishes independence after averaging for coordinate vectors on sufficiently separated intervals.

\subsection{Existence and stationarity of the limiting process}

The following proposition constructs the stationary process by averaging the limiting laws over the marker configurations. It establishes full support and $(L-1)$-tuplewise independence. For each fixed marker configuration $h$, the coordinate vector on every block $B\in\mathcal P_k(h)$ has law $\rho_B$ under $\nu_h$.

\begin{proposition}
    \label{prop:limit}
    There exists a family $(\nu_h)_h$ of Borel probability measures on $\Sigma=\{-1,1\}^{\mathbb Z}$ such that $h\mapsto\nu_h(A)$ is measurable for every Borel set $A\subset\Sigma$. Define the averaged probability measure on $\Sigma$ by
    \begin{align}
        \nu(A):=\int\nu_h(A)\,\mathbb P_{\mathcal H}(dh), \quad\text{for }A\in\mathcal B(\Sigma). \label{eq:limit-law}
    \end{align}
    Then the following hold.
    \begin{enumerate}[label=\textup{(\roman*)},ref=\roman*]
        \item\label{it:limit-i} For every finite set $E\subset\mathbb Z$, there exists an integer $k_E\geq0$, depending only on $E$ and not on $h$, such that, for $\mathbb P_{\mathcal H}$-almost every $h$,
        $$
        \nu_h^{(\ell)}\big|_E=\nu_h\big|_E\quad\text{for every }\ell\geq k_E.
        $$
        Thus the joint distribution of the coordinates indexed by $E$ is unchanged after level $k_E$ and equals their joint distribution under $\nu_h$.
        \item\label{it:limit-ii} The measure $\nu$ is $\sigma$-invariant and has full support.
        More precisely, $\nu(\sigma^{-1}A)=\nu(A)$ for every Borel set $A\subset\Sigma$, and $ \nu(\omega_i=\varepsilon_i\text{ for every }i\in E)>0$ for every finite $E\subset\mathbb Z$ and every choice $(\varepsilon_i)_{i\in E}\in\{-1,1\}^E$.
        \item\label{it:limit-iii} Let $Z_i(\omega):=\omega_i$.
        For every $1\leq m\leq L-1$, distinct indices $j_1,\ldots,j_m\in\mathbb Z$, and values $\varepsilon_1,\ldots,\varepsilon_m\in\{-1,1\}$, one has, for $\mathbb P_{\mathcal H}$-almost every $h$,
        $$
        \nu_h(Z_{j_1}=\varepsilon_1,\ldots,Z_{j_m}=\varepsilon_m)=2^{-m}.
        $$
        Therefore, $\nu(Z_{j_1}=\varepsilon_1,\ldots,Z_{j_m}=\varepsilon_m)=2^{-m}.$ Hence the coordinates are $(L-1)$-tuplewise independent under both the conditional laws $\nu_h$ and the averaged law $\nu$, and take the values $-1$ and $1$ with probability $1/2$ each.
        \item\label{it:limit-iv} For $\mathbb P_{\mathcal H}$-almost every $h$, every $k\geq0$ and every block $B\in\mathcal P_k(h)$, $ \nu_h\big|_B=\rho_B,$ where $\rho_B$ is the block law constructed in Section~\ref{sec:block-laws}, using the uniform law on $\{-1,1\}$ at level $0$ and the recursion \eqref{eq:parent-law} at higher levels.
    \end{enumerate}
\end{proposition}

\begin{proof}
    We divide the proof into four steps. We work with realisations $h$ of $\mathcal H=(\mathcal M_k)_{k\geq0}$ for which $\mathcal M_0=\mathbb Z$, the marker sets are nested and bi-infinite, and the bounds \itemref{it:H-lengths}--\itemref{it:H-children} hold at every level. By Proposition~\ref{prop:hierarchy}, these properties hold $\mathbb P_{\mathcal H}$-almost surely and define a measurable set of marker configurations.
    \begin{step}[1]
        We construct $\nu_h$, establish its measurable dependence on $h$, and prove \itemref{it:limit-i}.
    \end{step}
    \begin{stepproof}[Proof of Step 1]
        Fix a nonempty finite set $E\subset\mathbb Z$, let $I=[\min E,\max E]\cap\mathbb Z$, and choose an integer $k_E\geq1$ such that $|I|\leq(4L)^{k_E}$. This choice depends only on $E$, not on $h$. Every level-$k_E$ block has length greater than $(4L)^{k_E}$ by \itemref{it:H-lengths}. Hence, $I$ intersects at most two such blocks. Otherwise it would contain a complete block between the first and last blocks it intersects, contradicting $|I|\leq(4L)^{k_E}$. Hence
        $$
        \#\{B\in\mathcal P_{k_E}(h):B\cap E\neq\varnothing\}\leq2\leq L-1.
        $$
        Lemma~\ref{lem:stabilisation} gives
        $$
        \nu_h^{(\ell)}\big|_E=\nu_h^{(k_E)}\big|_E\quad\text{for every }\ell\geq k_E.
        $$
        We take $\nu_h^{(k_E)}|_E$ as the proposed limiting marginal on $\{-1,1\}^E$. The identity above shows that choosing a larger level gives the same marginal.
        
        These proposed marginals are consistent. Indeed, if $E\subset F$ are finite and nonempty, choose $\ell\geq\max\{k_E,k_F\}$. Then
        $$
        \left(\nu_h^{(k_F)}\big|_F\right)\big|_E=\left(\nu_h^{(\ell)}\big|_F\right)\big|_E=\nu_h^{(\ell)}\big|_E=\nu_h^{(k_E)}\big|_E.
        $$
        The Kolmogorov extension theorem~\cite[Theorem~8.23]{Kallenberg} therefore gives a unique Borel probability measure $\nu_h$ on $\Sigma$ whose marginal on every finite nonempty $E$ is $\nu_h^{(k_E)}|_E$. Thus
        $$
        \nu_h^{(\ell)}\big|_E=\nu_h\big|_E\quad\text{for every }\ell\geq k_E,
        $$
        proving \itemref{it:limit-i}.
        
        We next verify measurability. Fix a cylinder event
        $$
        A:=\{\omega\in\Sigma:\omega_i=\varepsilon_i\text{ for every }i\in E\},
        $$
        where $E$ is finite and nonempty. By the construction of $\nu_h$ and the product definition of $\nu_h^{(k_E)}$,
        $$
        \nu_h(A)=\nu_h^{(k_E)}(A)=\prod_{\substack{B\in\mathcal P_{k_E}(h) \\
                B\cap E\neq\varnothing}}\rho_B(X_j=\varepsilon_j\text{ for every }j\in B\cap E).
        $$
        The locality assertion in Lemma~\ref{lem:finite-output} shows that this quantity is determined by the finite configuration $h^{(k_E)}[I^{+3(4L)^{k_E}}]$. That configuration consists of finitely many measurable marker indicators and takes only finitely many values. Therefore $h\mapsto\nu_h(A)$ is measurable on the full-measure set under consideration.
        
        Outside the measurable set of marker configurations specified at the start of the proof, define $\nu_h$ as the product of the uniform laws on $\{-1,1\}$. This choice on a null set does not change the averaged measure.
        
        Since cylinder sets form a generating class for the Borel $\sigma$-algebra of $\Sigma$, the kernel criterion in~\cite[Lemma~3.1]{Kallenberg} implies that $h\mapsto\nu_h$ is a probability kernel. Hence
        $$
        \nu(A):=\int\nu_h(A)\,\mathbb P_{\mathcal H}(dh), \quad\text{for }A\in\mathcal B(\Sigma),
        $$
        defines a Borel probability measure on $\Sigma$.
    \end{stepproof}
    
    \begin{step}[2]
        We prove \itemref{it:limit-ii}. The coordinate process $(Z_i)_{i\in\Z}$ is stationary under $\nu$, and $\nu$ has full support on $\Sigma$.
    \end{step}
    \begin{stepproof}[Proof of Step 2]
        Let $A$ be a cylinder event depending on coordinates in a finite nonempty set $E$, and fix $m\in\mathbb Z$. The event $\sigma^{-m}A$ depends on coordinates in $E+m$. Choose $\ell\geq\max\{k_E,k_{E+m}\}$. By Step~1, the conditional probabilities of both events agree with their level-$\ell$ probabilities for almost every $h$. Integrating and using \eqref{eq:finite-level-law} gives
        $$
        \nu(A)=\nu^{(\ell)}(A)\quad\text{and}\quad \nu(\sigma^{-m}A)=\nu^{(\ell)}(\sigma^{-m}A).
        $$
        The coordinate process $Z^{(\ell)}$ is stationary under $\nu^{(\ell)}$, as proved after \eqref{eq:stationarym}. Hence
        $$
        \nu(\sigma^{-m}A)=\nu^{(\ell)}(\sigma^{-m}A)=\nu^{(\ell)}(A)=\nu(A).
        $$
        The probability measures $(\sigma^m)_*\nu$ and $\nu$ agree on cylinder events. Since these form a $\pi$-system generating $\mathcal B(\Sigma)$, the two measures agree on every Borel set~\cite[Lemma~1.18]{Kallenberg}. This proves stationarity.
        
        For full support, fix a nonempty cylinder event $A=\{\omega:\omega_i=\varepsilon_i\text{ for every }i\in E\}$. For $\mathbb P_{\mathcal H}$-almost every realisation $h$ of $\mathcal H$, the probability $\nu_h(A)$ is the finite product displayed in Step~1. Each factor is positive because Lemma~\ref{lem:block-properties}\itemref{it:block-properties-i} gives full support of $\rho_B$, and any prescription of coordinate values on $B\cap E$ extends to a coordinate configuration on $B$. Thus $\nu_h(A)>0$ for $\mathbb P_{\mathcal H}$-almost every $h$, and
        $$
        \nu(A)=\int\nu_h(A)\,\mathbb P_{\mathcal H}(dh)>0.
        $$
        Every nonempty open subset of $\Sigma$ contains a nonempty cylinder event, so $\supp\nu=\Sigma$.
    \end{stepproof}
    
    \begin{step}[3]
        We prove \itemref{it:limit-iii}. For any distinct indices $j_1,\ldots,j_m$, with $1\leq m\leq L-1$, the variables $Z_{j_1},\ldots,Z_{j_m}$ are independent and have mean zero. This holds with respect to $\nu$ and, for $\mathbb P_{\mathcal H}$-almost every realisation $h$ of $\mathcal H$, with respect to $\nu_h$.
    \end{step}
    \begin{stepproof}[Proof of Step 3]
        Fix $1\leq m\leq L-1$, distinct indices $j_1,\ldots,j_m\in\mathbb Z$, and values $\varepsilon_1,\ldots,\varepsilon_m\in\{-1,1\}$. Put $E=\{j_1,\ldots,j_m\}$. Since the level-$0$ blocks are singletons, $E$ intersects exactly $m\leq L-1$ of them. Lemma~\ref{lem:stabilisation}, applied with $k=0$, therefore gives
        $$
        \nu_h^{(\ell)}\big|_E=\nu_h^{(0)}\big|_E=\bigotimes_{i=1}^m\rho_{\{j_i\}}\quad\text{for every }\ell\geq0.
        $$
        For $\ell\geq k_E$, Step~1 identifies the left-hand side with $\nu_h|_E$. Each singleton law is the uniform law on $\{-1,1\}$, so
        $$
        \nu_h(Z_{j_1}=\varepsilon_1,\ldots,Z_{j_m}=\varepsilon_m)=\prod_{i=1}^m\rho_{\{j_i\}}(X_{j_i}=\varepsilon_i)=2^{-m}.
        $$
        Integrating this identity over $h$ yields
        $$
        \nu(Z_{j_1}=\varepsilon_1,\ldots,Z_{j_m}=\varepsilon_m)=\int2^{-m}\,\mathbb P_{\mathcal H}(dh)=2^{-m}.
        $$
        These are precisely the joint distributions asserted in \itemref{it:limit-iii}.
    \end{stepproof}
    
    \begin{step}[4]
        We prove \itemref{it:limit-iv}. For every $k\geq0$, every block $B\in\mathcal P_k(h)$ and every $\varepsilon=(\varepsilon_i)_{i\in B}\in\{-1,1\}^B$,
        $$
        \nu_h(Z_i=\varepsilon_i\text{ for every }i\in B) =\rho_B(\{\varepsilon\}).
        $$
    \end{step}
    \begin{stepproof}[Proof of Step 4]
        Fix $h$, an integer $k\geq0$, and a block $B\in\mathcal P_k(h)$. The set $B$ intersects only one level-$k$ block, namely itself. Lemma~\ref{lem:stabilisation}, applied with $E=B$, gives
        $$
        \nu_h^{(\ell)}\big|_B=\nu_h^{(k)}\big|_B=\rho_B\quad\text{for every }\ell\geq k,
        $$
        where the last equality follows from the product definition of $\nu_h^{(k)}$. Since $B$ is finite, from Step~1 we obtain that there exists $k_B$ such that $\nu_h^{(\ell)}|_B=\nu_h|_B$ for every $\ell\geq k_B$. Taking $\ell\geq\max\{k,k_B\}$ proves
        $$
        \nu_h\big|_B=\nu_h^{(\ell)}\big|_B=\rho_B.
        $$
    \end{stepproof}
    This completes the proof.
\end{proof}

We now consider $\mathcal H$ and $Z$ on the same probability space, with joint probability measure
$$
\mathbb P(dh,d\omega) :=\mathbb P_{\mathcal H}(dh)\,\nu_h(d\omega).
$$
The process $Z=(Z_i)_{i\in\Z}$ has law $\nu$. Knowing $\mathcal H$ determines the blocks, but the coordinates $Z_i$ remain random. For a realisation $h$ of $\mathcal H$, their law is $\nu_h$. Thus the conditional expectation (see~\cite[Chapter~8]{Kallenberg} for a review) is
$$
\E[F(Z)\mid\mathcal H] =\int_\Sigma F(\omega)\,\nu_{\mathcal H}(d\omega) \quad\text{almost surely},
$$
for every bounded Borel function $F:\Sigma\to\mathbb R$. Taking $F=\one_A$ gives the probability that $Z\in A$ when the marker sets are known, namely
$$
\mathbb P(Z\in A\mid\mathcal H) =\E[\one_A(Z)\mid\mathcal H] =\nu_{\mathcal H}(A) \quad\text{almost surely},
$$
for every Borel set $A\subset\Sigma$.

The following lemma is the counterpart of Lemma~\ref{lem:block-correlations} for the limiting process $(Z_i)_{i\in\mathbb Z}$, transferring the correlation formula for the block coordinates $X_i$ to this process. The $L$-point correlations are nonpositive, so their negative contributions add without cancellation and are used to prove failure of the CLT for $(Z_i)_{i\in\mathbb Z}$.

\begin{lemma}
    \label{lem:correlation-classification}
    For $\mathbb P_{\mathcal H}$-almost every realisation $h$ of $\mathcal H$, the following statements hold for every $j_1<\cdots<j_L$ in $\mathbb Z$. Write $E:=\{j_1,\ldots,j_L\}$.
    
    Suppose that $E\subset B$ for some block $B\in\mathcal P_k(h)$. Choose $k$ minimal, and write $B_1,\ldots,B_r\in\mathcal P_{k-1}(h)$ for the children of the unique block $B\in\mathcal P_k(h)$ containing $E$, ordered from left to right. Then
    \begin{align}
        \int_\Sigma\prod_{i=1}^L\omega_{j_i}\,\nu_h(d\omega) =
        \begin{cases}
            \displaystyle -\frac{1}{2\,3^{L/2}\sqrt{|B_1|\cdots|B_L|}}, &\text{if }|E\cap B_i|=1\text{ for every }1\leq i\leq L, \\[2ex]
            0, &\text{otherwise.}
        \end{cases}
        \label{eq:correlation-classification}
    \end{align}
    If no block at any level contains $E$, the integral is zero.
    
    In particular, for every $j_1<\cdots<j_L$,
    $$
    \E\left[\prod_{i=1}^L Z_{j_i}\,\middle|\,\mathcal H\right] \leq0 \quad\text{almost surely},\ \text{and } \E\left[\prod_{i=1}^L Z_{j_i}\right]\leq0.
    $$
\end{lemma}

\begin{proof}
    Fix a realisation $h$ of $\mathcal H$ for which Proposition~\ref{prop:limit} holds and put $E=\{j_1,\ldots,j_L\}$. Suppose first that a block contains $E$, and choose the smallest such level $k$. The block $B\in\mathcal P_k(h)$ is unique because $\mathcal P_k(h)$ is a partition. Moreover, $k\geq1$, since a singleton cannot contain $L$ distinct indices. Write $B_1,\ldots,B_r\in\mathcal P_{k-1}(h)$ for the children of $B\in\mathcal P_k(h)$, ordered from left to right. The set $E$ intersects at least two children of $B$. Otherwise one child $B_i\in\mathcal P_{k-1}(h)$ would contain it, contradicting minimality.
    
    Proposition~\ref{prop:limit}\itemref{it:limit-iv} identifies the marginal on $B$ with $\rho_B$. Thus
    $$
    \int_\Sigma\prod_{i=1}^L\omega_{j_i}\,\nu_h(d\omega) =\E_{\rho_B}\prod_{i=1}^L X_{j_i}.
    $$
    The selected coordinates intersect at least two children, so Lemma~\ref{lem:block-correlations} applies and gives \eqref{eq:correlation-classification}.
    
    Now suppose that no block contains $E$. At each level $\ell$, the blocks $C\in\mathcal P_\ell(h)$ intersecting $E$ satisfy $1\leq|C\cap E|\leq L-1$. By Lemma~\ref{lem:block-properties}\itemref{it:block-properties-ii}, the selected coordinates within each such block are independent and centred, so
    $$
    \E_{\rho_C}\prod_{j\in C\cap E}X_j =\prod_{j\in C\cap E}\E_{\rho_C}X_j=0.
    $$
    Since the blocks are independent under $\nu_h^{(\ell)}$,
    $$
    \int_\Sigma\prod_{j\in E}\omega_j\,\nu_h^{(\ell)}(d\omega) =\prod_{\substack{C\in\mathcal P_\ell(h) \\
            C\cap E\neq\varnothing}} \E_{\rho_C}\prod_{j\in C\cap E}X_j=0.
    $$
    For $\ell\geq k_E$, the marginal on $E$ agrees with that of $\nu_h$, so the same integral under $\nu_h$ is zero.
    
    In both cases the integral under $\nu_h$ is nonpositive. The conditional-expectation identity above therefore gives $\E[\prod_iZ_{j_i}\mid\mathcal H]\leq0$ almost surely, and
    $$
    \E\prod_{i=1}^L Z_{j_i} =\int\left(\int_\Sigma\prod_{i=1}^L\omega_{j_i}\, \nu_h(d\omega)\right)\mathbb P_{\mathcal H}(dh)\leq0.
    $$
    There are countably many increasing integer tuples. Intersecting their probability-one sets makes all these assertions simultaneous.
\end{proof}
The simultaneous assertion allows \eqref{eq:correlation-classification} to be applied to a block selected as a function of the markers.

\section{Exponential mixing of order \texorpdfstring{$R$}{R}}
\label{sec:correlations}
We study the process $Z=(Z_i)_{i\in\mathbb Z}$ of Proposition~\ref{prop:limit}, with law $\nu$, using the notation introduced before Theorem~\ref{thm:symbolic}. The finite-level dependence of Lemma~\ref{lem:finite-output} gives joint independence of the coordinate vectors on at most $R$ intervals separated by a fixed multiple of their lengths (Proposition~\ref{prop:windows}). This implies that $\sigma$ is mixing with respect to $\nu$ (Corollary~\ref{cor:mixing}) and yields exponential mixing of order $R$ for regular observables (Proposition~\ref{prop:holder}).

\begin{proposition}
    \label{prop:windows}
    Let $(Z_i)_{i\in\mathbb Z}$ be the process given in Proposition~\ref{prop:limit} and  define $K:=4L(D+7)$, where $D$ is the constant in Proposition~\ref{prop:hierarchy}. Let $2\le s\le R$ and let $I_1,\ldots,I_s$ be ordered disjoint nonempty intervals. If
    \begin{align}
        \min_{i<s}\dist(I_i,I_{i+1})>K\max_i|I_i|, \label{eq:window-separation}
    \end{align}
    then $Z[I_1],\ldots,Z[I_s]$ are jointly independent.
\end{proposition}
\begin{proof}
    Put $m:=\max_i|I_i|$ and choose the least $k\geq0$ with $(4L)^k\geq m$. Minimality gives
    $$
    m\leq(4L)^k\leq4Lm.
    $$
    The same inequalities hold when $k=0$, in which case $m=1$. For $k\geq1$, each interval intersects at most two blocks of $\mathcal P_k(h)$ by their length lower bound. For $k=0$, each window is a singleton. Hence $E:=I_1\cup\cdots\cup I_s$ intersects at most $2s\leq2R=L-2$ blocks. Lemma~\ref{lem:stabilisation} gives $\nu_h^{(\ell)}|_E=\nu_h^{(k)}|_E$ for every $\ell\geq k$. Taking $\ell$ large enough to apply Proposition~\ref{prop:limit}\itemref{it:limit-i}, and then averaging, we obtain
    $$
    \nu_h|_E=\nu_h^{(k)}|_E\quad\text{and}\quad \nu|_E=\nu^{(k)}|_E.
    $$
    On the other hand, \eqref{eq:window-separation} implies
    $$
    \min_{i<s}\dist(I_i,I_{i+1}) >4L(D+7)m\geq(D+7)(4L)^k>(D+6)(4L)^k.
    $$
    Lemma~\ref{lem:finite-output} makes the finite-level restrictions to these windows jointly independent. For events $A_i\subset\{-1,1\}^{I_i}$, equality of the joint laws on $E$ and of their marginals gives
    \begin{align}
        \nu(Z[I_i]\in A_i\text{ for all }i) &=\nu^{(k)}(Z^{(k)}[I_i]\in A_i\text{ for all }i)\notag \\
        &=\prod_{i=1}^s\nu^{(k)}(Z^{(k)}[I_i]\in A_i) =\prod_{i=1}^s\nu(Z[I_i]\in A_i). \label{eq:window-factorisation}
    \end{align}
    This is the required independence.
\end{proof}

To pass from window independence to exponential mixing, we approximate observables in $\mathcal F_\theta(\Sigma)$ by cylinder functions with window radius proportional to the smallest time gap. The cylinder functions are independent and the approximation error decays exponentially, giving the bounds needed for the horseshoe realisations.

\begin{proposition}
    \label{prop:holder}
    For every $0<\theta<1$, the measure $\nu$ satisfies
    \begin{align}
        \left|\nu\left(\prod_{i=1}^RF_i\circ\sigma^{n_i}\right) -\prod_{i=1}^R\nu(F_i)\right| \le C_{R,\theta}\,e^{-c_\theta\Delta(\mathbf n)} \prod_{i=1}^R\|F_i\|_\theta\quad\text{where}\quad c_\theta:=\frac{|\log\theta|}{4(K+1)} .\label{eq:holder-multiple}
    \end{align}
\end{proposition}
\begin{proof}
    Set $s:=R$ and fix $F_1,\ldots,F_s\in\mathcal F_\theta(\Sigma)$ and the observation times. For $\ell\geq0$, let $\omega^{(\ell)}$ agree with $\omega$ on $[-\ell,\ell]$ and equal $1$ elsewhere, and put $F_{i,\ell}(\omega):=F_i(\omega^{(\ell)})$. This function depends only on the coordinates in that window. Since $\omega^{(\ell)}\in\Sigma$ and $d_\theta(\omega,\omega^{(\ell)})\leq\theta^{\ell+1}$,
    \begin{align}
        \|F_{i,\ell}\|_\infty\leq\|F_i\|_\infty\quad\text{and}\quad \|F_i-F_{i,\ell}\|_\infty \leq\theta^{\ell+1}\|F_i\|_\theta. \label{eq:cylinder-approximation}
    \end{align}
    Write $\Delta:=\Delta(\mathbf n)$. For $\Delta>2(K+1)$, choose $\ell=\lfloor\Delta/(4(K+1))\rfloor$. The function $F_{i,\ell}\circ\sigma^{n_i}$ depends on $I_i=[n_i-\ell,n_i+\ell]\cap\mathbb Z$. These windows have length $2\ell+1$ and successive distances at least $\Delta-2\ell$. The choice of $\ell$ gives
    $$
    \Delta-2\ell-K(2\ell+1) =\Delta-2(K+1)\ell-K\geq\Delta/2-K>0.
    $$
    Thus Proposition~\ref{prop:windows} applies. Independence and stationarity yield
    $$
    \nu\left(\prod_{i=1}^sF_{i,\ell}\circ\sigma^{n_i}\right) =\prod_{i=1}^s\nu(F_{i,\ell}).
    $$
    
    We now estimate the two approximation errors. For any $a_1,\ldots,a_s,b_1,\ldots,b_s$, the telescoping identity is
    $$
    \prod_{i=1}^s a_i-\prod_{i=1}^s b_i =\sum_{i=1}^s \left(\prod_{j<i}a_j\right)(a_i-b_i) \left(\prod_{j>i}b_j\right).
    $$
    Apply it pointwise with $a_i=F_i\circ\sigma^{n_i}$ and $b_i=F_{i,\ell}\circ\sigma^{n_i}$. Using \eqref{eq:cylinder-approximation}, we obtain
    \begin{align*}
        \left|\nu\left(\prod_iF_i\circ\sigma^{n_i}\right) -\nu\left(\prod_iF_{i,\ell}\circ\sigma^{n_i}\right)\right| &\leq\sum_{i=1}^s\|F_i-F_{i,\ell}\|_\infty \prod_{j\neq i}\|F_j\|_\infty\leq s\theta^{\ell+1}\prod_{i=1}^s\|F_i\|_\theta.
    \end{align*}
    For the product of expectations, use the same identity with $a_i=\nu(F_i)$ and $b_i=\nu(F_{i,\ell})$. Here $|a_i|,|b_i|\leq\|F_i\|_\infty$ and $|a_i-b_i|\leq\|F_i-F_{i,\ell}\|_\infty$, so
    $$
    \left|\prod_i\nu(F_i)-\prod_i\nu(F_{i,\ell})\right| \leq s\theta^{\ell+1}\prod_i\|F_i\|_\theta.
    $$
    The cylinder factorisation and triangle inequality now give
    \begin{align*}
        \left|\nu\left(\prod_iF_i\circ\sigma^{n_i}\right) -\prod_i\nu(F_i)\right| &\leq2s\theta^{\ell+1}\prod_i\|F_i\|_\theta\leq2se^{-c_\theta\Delta}\prod_i\|F_i\|_\theta,
    \end{align*}
    where $\ell+1>\Delta/(4(K+1))$ gives $c_\theta=|\log\theta|/(4(K+1))$. For $\Delta\leq2(K+1)$, the left side is at most $2\prod_i\|F_i\|_\theta$, so all cases are covered by $ C_{s,\theta}:=\max\{2s,2e^{2c_\theta(K+1)}\}. $
\end{proof}

\section{Failure of the CLT and of higher-order exponential mixing}
\label{sec:moment}
Recall that $(Z_i)_{i\in\Z}$ is the stationary process of Proposition~\ref{prop:limit}, with $Z_i(\omega)=\omega_i$ under $\nu$. We use $S_n=\sum_{i=0}^{n-1}Z_i$ and $L=2R+2$. Proposition~\ref{prop:key-estimate} finds a contribution of order $n^{L/2}$ to the sum of negative $L$-point correlations, using tuples whose successive gaps are proportional to $n$. This gives both the moment deficit excluding Gaussian weak subsequential limits and a polynomial lower bound for a separated correlation, excluding exponential mixing of order $L$.

\begin{proposition}
    \label{prop:key-estimate}
    Let $n\ge6(4L)^2$ and let $k\ge2$ be the integer with
    \begin{align}
        6(4L)^k\le n<6(4L)^{k+1}. \label{eq:moment-scale}
    \end{align}
    Then, writing $g_n:=(4L)^{k-1}/2$,
    \begin{align}
        -\sum_{\substack{0\le j_1<\cdots<j_L<n \\
                \Delta(\mathbf j)\ge g_n}} \E\left[\prod_{i=1}^L Z_{j_i}\right] \ \ge\ \frac{(4L)^{(k-1)L/2}}{2^{L+1}3^{L/2}} \ \ge\ \frac1{2^{L+1}3^{L/2}}\left(\frac n{6(4L)^2}\right)^{L/2}, \label{eq:key-estimate}
    \end{align}
    and $g_n\ge n/(12(4L)^2)$.
\end{proposition}
\begin{proof}
    The proof is divided into two steps.
    \begin{step}[1]
        Choose a complete block and separated subintervals of its children.
    \end{step}
    \begin{stepproof}[Proof of Step 1]
        Let $A=[u,v)\cap\mathbb Z\in\mathcal P_k$ be the block containing $0$. Since $u\leq0<v$, \itemref{it:H-lengths} gives $v\leq v-u\leq3(4L)^k$. The next block $A'\in\mathcal P_k$ has length at most $3(4L)^k$, so its right endpoint is at most $6(4L)^k\leq n$. Thus at least one complete level-$k$ block lies in $[0,n)$. Choose the leftmost such block $B=B(\mathcal H)\in\mathcal P_k$. Its first $L$ children are $B_1,\ldots,B_L\in\mathcal P_{k-1}$, ordered from left to right. Write $B_i=[u_i,v_i)\cap\mathbb Z$. There are at least $L$ children by \itemref{it:H-children}, and consecutive children satisfy $v_i=u_{i+1}$.
        
        Set $a:=(4L)^{k-1}/4$. This is an integer because $k\geq2$ and $4L$ is divisible by $4$. Define
        $$
        C_i:=[u_i+a,v_i-a)\cap\mathbb Z.
        $$
        Since $|B_i|>(4L)^{k-1}=4a$,
        $$
        |C_i|=|B_i|-2a\geq|B_i|/2>0.
        $$
        If $j_i\in C_i$, then $u_i+a\leq j_i\leq v_i-a-1$. Hence
        $$
        j_{i+1}-j_i \geq(u_{i+1}+a)-(v_i-a-1) =2a+1=g_n+1.
        $$
        Every selected coordinate lies in $B\subset[0,n)$, so these tuples are increasing and have $\Delta(\mathbf j)\geq g_n$. The upper bound in \eqref{eq:moment-scale} also gives
        $$
        g_n=\frac{(4L)^{k-1}}2>\frac n{12(4L)^2}.
        $$
    \end{stepproof}
    
    \begin{step}[2]
        Sum the conditional correlations and average over the markers.
    \end{step}
    \begin{stepproof}[Proof of Step 2]
        Fix a realisation $h$ of $\mathcal H$ for which Step~1 and Lemma~\ref{lem:correlation-classification} hold. Let $B=B(h)\in\mathcal P_k(h)$ be the block chosen in Step~1, with first $L$ children $B_1,\ldots,B_L\in\mathcal P_{k-1}(h)$ and trimmed intervals $C_i\subset B_i$. For this fixed $h$, these blocks and intervals are fixed. If $j_i\in C_i$, then the selected coordinates intersect distinct children $B_i\in\mathcal P_{k-1}(h)$, so no block of that level or below contains all of them. Their smallest containing block is therefore $B\in\mathcal P_k(h)$. The correlation classification gives
        $$
        \E_{\nu_h}\prod_{i=1}^L Z_{j_i} =-\frac1{2\,3^{L/2}\prod_{i=1}^L|B_i|^{1/2}}.
        $$
        This value is the same for all $\prod_i|C_i|$ choices. Using $|C_i|\geq|B_i|/2$ and $|B_i|>(4L)^{k-1}$, we obtain
        \begin{align}
            -\sum_{j_1\in C_1,\ldots,j_L\in C_L} \E_{\nu_h}\prod_{i=1}^L Z_{j_i} &=\frac{\prod_i|C_i|} {2\,3^{L/2}\prod_i|B_i|^{1/2}}\geq\frac{\prod_i|B_i|^{1/2}}{2^{L+1}3^{L/2}} \geq\frac{(4L)^{(k-1)L/2}}{2^{L+1}3^{L/2}}. \label{eq:separated-conditional-correlation-sum}
        \end{align}
        
        By Step~1, all these tuples belong to the deterministic set
        $$
        \mathcal T_n:=\{\mathbf j:0\leq j_1<\cdots<j_L<n,\ \Delta(\mathbf j)\geq g_n\}.
        $$
        Every other conditional $L$-point correlation is also nonpositive. Adding its negative therefore preserves the lower bound and gives
        $$
        -\sum_{\mathbf j\in\mathcal T_n} \E_{\nu_h}\prod_{i=1}^LZ_{j_i} \geq\frac{(4L)^{(k-1)L/2}}{2^{L+1}3^{L/2}}.
        $$
        This holds for almost every $h$. Although the block and the $C_i$ depend on $h$, the enlarged index set $\mathcal T_n$ does not. We may therefore interchange its finite sum with the integral over the marker law, obtaining
        \begin{align*}
            -\sum_{\mathbf j\in\mathcal T_n}\E\prod_{i=1}^LZ_{j_i} &=\int\left(-\sum_{\mathbf j\in\mathcal T_n} \E_{\nu_h}\prod_{i=1}^LZ_{j_i}\right) \mathbb P_{\mathcal H}(dh) \geq\frac{(4L)^{(k-1)L/2}}{2^{L+1}3^{L/2}}.
        \end{align*}
        Finally, $n<6(4L)^{k+1}$ implies $(4L)^{k-1}>n/(6(4L)^2)$, proving the second inequality in \eqref{eq:key-estimate}.
    \end{stepproof}
\end{proof}

\subsection{No Gaussian subsequential limits}
Tuplewise independence identifies the low-order moments, while Proposition~\ref{prop:key-estimate} gives a uniform deficit in the normalised $L$-th moment. Together these exclude every Gaussian weak subsequential limit.

The Khintchine inequality bounds the moments of independent symmetric $\{-1,1\}$-valued random variables by the corresponding Gaussian moments. We combine this bound with the negative $L$-point correlations of our process to obtain a strict $L$-th moment deficit.

\begin{lemma}[{\cite{Haagerup}}]
    \label{lem:khintchine}
    Let $\xi_1,\ldots,\xi_n$ be independent random variables such that
    $$
    \mathbb P(\xi_i=1)=\mathbb P(\xi_i=-1)=\frac12\quad\text{for each }1\leq i\leq n.
    $$
    For every $p\geq2$ and every $a_1,\ldots,a_n\in\mathbb R$,
    $$
    \left(\E\left|\sum_{i=1}^n a_i\xi_i\right|^p\right)^{1/p} \leq \left(\E|G|^p\right)^{1/p} \left(\sum_{i=1}^n a_i^2\right)^{1/2},
    $$
    where $G\sim\mathcal N(0,1)$. In particular, if $L\geq2$ is even, then $ \E(\xi_1+\cdots+\xi_n)^L\leq(L-1)!!\,n^{L/2}. $
\end{lemma}

\begin{corollary}
    \label{cor:moment}
    Let $n\geq1$ and let $\xi_1,\ldots,\xi_n$ be independent random variables such that
    $$
    \mathbb P(\xi_i=1)=\mathbb P(\xi_i=-1)=\frac12\quad\text{for each }1\leq i\leq n.
    $$
    Then
    $$
    \E S_n^p=\E(\xi_1+\cdots+\xi_n)^p\quad\text{for every integer }1\leq p\leq L-1.
    $$
    In particular,
    \begin{align}
        \E S_n=0,\quad \E S_n^2=n\quad\text{and}\quad \E S_n^4=3n^2-2n. \label{eq:low-moments}
    \end{align}
    Moreover, for every $n\geq6(4L)^2$,
    \begin{align}
        \E\left(\frac{S_n}{\sqrt n}\right)^L \leq (L-1)!!-\delta_L\quad\text{and}\quad \delta_L:= \frac{L!}{2^{L+1}3^{L/2}\bigl(6(4L)^2\bigr)^{L/2}}>0. \label{eq:moment-deficit}
    \end{align}
\end{corollary}

\begin{proof}
    Fix $1\leq p\leq L-1$. Each monomial $Z_{j_1}\cdots Z_{j_p}$ in $S_n^p$ involves at most $L-1$ distinct coordinates. By Proposition~\ref{prop:limit}\itemref{it:limit-iii}, the corresponding expectations agree even when some indices are repeated:
$$
\E[Z_{j_1}\cdots Z_{j_p}] =\E[\xi_{j_1+1}\cdots\xi_{j_p+1}].
$$
Summing over $j_1,\ldots,j_p\in\{0,\ldots,n-1\}$ gives the asserted moment equality. Taking $p=1,2,4$ gives \eqref{eq:low-moments}.
    
    For the $L$-th moment, subtracting the two expansions gives
    \begin{align*}
        \E S_n^L-\E(\xi_1+\cdots+\xi_n)^L =\sum_{j_1,\ldots,j_L=0}^{n-1} \left(\E\prod_{i=1}^LZ_{j_i} -\E\prod_{i=1}^L\xi_{j_i+1}\right).
    \end{align*}
    Terms with repeated indices involve at most $L-1$ distinct coordinates and cancel by the same argument. For $L$ distinct indices, independence and centring give $\E\prod_i\xi_{j_i+1}=0$. Each increasing tuple has $L!$ permutations, all with the same coordinate product. Thus
    \begin{align}
        \E S_n^L =\E(\xi_1+\cdots+\xi_n)^L +L!\sum_{0\leq j_1<\cdots<j_L<n} \E\prod_{i=1}^L Z_{j_i}. \label{eq:moment-expansion}
    \end{align}
    By Lemma~\ref{lem:correlation-classification}, all summands in the last sum are nonpositive. Restricting to the separated tuples in Proposition~\ref{prop:key-estimate} therefore gives the upper bound
    $$
    \sum_{0\leq j_1<\cdots<j_L<n}\E\prod_iZ_{j_i} \leq-\frac1{2^{L+1}3^{L/2}} \left(\frac n{6(4L)^2}\right)^{L/2}.
    $$
    Lemma~\ref{lem:khintchine} bounds the independent moment by $(L-1)!!\,n^{L/2}$. Substituting both bounds into \eqref{eq:moment-expansion} yields
    $$
    \E S_n^L\leq(L-1)!!\,n^{L/2} -\frac{L!}{2^{L+1}3^{L/2}} \left(\frac n{6(4L)^2}\right)^{L/2}.
    $$
    Dividing by $n^{L/2}$ proves \eqref{eq:moment-deficit}.
\end{proof}

The exclusion of Gaussian weak subsequential limits from the moment estimates follows closely the argument in~\cite[Step~3.11]{BradleyPruss}. The fourth-moment bound identifies the mean and variance of any hypothetical Gaussian limit, while the strict deficit in the $L$-th moment rules out that limit. Our estimate on correlations at widely separated coordinates also yields the failure of exponential mixing at order $L$.
\begin{corollary}
    \label{cor:no-gaussian}
    The sequence $S_n/\sqrt n$ has no Gaussian weak subsequential limit.
\end{corollary}

\begin{proof}
    Suppose $Y_j:=S_{n_j}/\sqrt{n_j}$ converges in distribution to a Gaussian random variable $G$. By \eqref{eq:low-moments},
    $$
    \E Y_j=0,\quad \E Y_j^2=1\quad\text{and}\quad \E Y_j^4=3-\frac2{n_j}\leq3.
    $$
    In particular, for $A>0$,
    $$
    \E\bigl[|Y_j|\one_{\{|Y_j|>A\}}\bigr]\leq\frac3{A^3} \quad\text{and}\quad \E\bigl[Y_j^2\one_{\{|Y_j|>A\}}\bigr]\leq\frac3{A^2}.
    $$
    These uniform tail bounds give uniform integrability of both $(Y_j)$ and $(Y_j^2)$ by~\cite[condition~(3.18)]{Billingsley}. The continuous mapping theorem~\cite[Theorem~2.7]{Billingsley} gives $Y_j^2\mathrel{\xRightarrow{j\to\infty}}G^2$. Uniform integrability then passes the first two moments to the limit by~\cite[Theorem~3.5]{Billingsley}, giving
    $$
    \E G=\lim_j\E Y_j=0\quad\text{and}\quad \E G^2=\lim_j\E Y_j^2=1.
    $$
    Thus the only possible Gaussian limit is $\mathcal N(0,1)$. In particular, a degenerate Gaussian is already excluded.
    
    For the higher moment, fix $M>0$. The function $x\mapsto\min\{|x|^L,M\}$ is bounded and continuous, so the continuous mapping theorem~\cite[Theorem~2.7]{Billingsley}, boundedness and \eqref{eq:moment-deficit} imply
    $$
    \E\min\{|G|^L,M\} =\lim_j\E\min\{|Y_j|^L,M\} \leq(L-1)!!-\delta_L.
    $$
    Let $M\to\infty$ and apply monotone convergence. This is the usual lower-semicontinuity argument for a nonnegative moment~\cite[Theorem~3.4]{Billingsley}, and it gives
    $$
    \E|G|^L\leq(L-1)!!-\delta_L.
    $$
    But $L$ is even and $G\sim\mathcal N(0,1)$, so $\E|G|^L=(L-1)!!$, which yields a contradiction.
\end{proof}

\begin{remark}
    For $R=2$, Proposition~\ref{prop:holder} gives exponential two-point mixing, while Corollary~\ref{cor:no-gaussian} excludes the CLT for the coordinate observable $\varphi(\omega)=\omega_0$.
\end{remark}

\subsection{Failure of exponential mixing at order \texorpdfstring{$L$}{L}}
Dividing the lower bound in Proposition~\ref{prop:key-estimate} by the number of available tuples gives a lower bound for a single correlation. We again use the coordinate observable $\varphi(\omega)=\omega_0$.

\begin{corollary}
    \label{cor:higher-order}
    For every $N\ge6(4L)^2$ there are integers $0\le n_1<\cdots<n_L<N$ with
    \begin{align}
        \Delta(\mathbf n)\ \ge\ \frac N{12(4L)^2}\quad\text{and}\quad -\E\left[\prod_{i=1}^L Z_{n_i}\right]\ \ge\ \frac{N^{-L/2}}{2^{L+1}3^{L/2}\bigl(6(4L)^2\bigr)^{L/2}} . \label{eq:higher-order-lower-bound}
    \end{align}
    Therefore $(\Sigma,\nu,\sigma)$ is not exponentially mixing of order $L$ in $\mathcal F_\theta(\Sigma)$ for any $0<\theta<1$.
\end{corollary}

\begin{proof}
    We divide the proof into two steps.
    
    \begin{step}[1]
        We find a tuple $\mathbf n=(n_1,\ldots,n_L)$ with $0\leq n_1<\cdots<n_L<N$ satisfying \eqref{eq:higher-order-lower-bound}.
    \end{step}
    
    \begin{stepproof}[Proof of Step 1]
        Let $ \mathcal T_N:=\bigl\{(j_1,\ldots,j_L):0\leq j_1<\cdots<j_L<N,\ \Delta(\mathbf j)\geq g_N\bigr\}.$ Recall the definition of $g_N$ from Proposition~\ref{prop:key-estimate}. From Proposition~\ref{prop:key-estimate}, we obtain
        \begin{align}
            \sum_{\mathbf j\in\mathcal T_N}\left(-\E\left[\prod_{i=1}^L Z_{j_i}\right]\right)\geq\frac{N^{L/2}}{2^{L+1} 3^{L/2}\bigl(6(4L)^2\bigr)^{L/2}}. \label{eq:ineq2}
        \end{align}
        Each summand is nonnegative by Lemma~\ref{lem:correlation-classification}. Since $\mathcal T_N$ contains at most $N^L$ tuples, the sum is at most $N^L$ times its largest summand,
        \begin{align}
            \sum_{\mathbf j\in\mathcal T_N}\left(-\E\left[\prod_{i=1}^L Z_{j_i}\right]\right)\leq N^L\max_{\mathbf j\in\mathcal T_N}\left(-\E\left[\prod_{i=1}^L Z_{j_i}\right]\right).\label{eq:ineq1}
        \end{align}
        Choose $\mathbf n=(n_1,\ldots,n_L)\in\mathcal T_N$ at which this maximum is attained. Combining \eqref{eq:ineq2} and \eqref{eq:ineq1} and dividing by $N^L$ we obtain
        \begin{align*}
            -\E\left[\prod_{i=1}^L Z_{n_i}\right]\geq\frac{N^{-L/2}}{2^{L+1} 3^{L/2}\bigl(6(4L)^2\bigr)^{L/2}}.
        \end{align*}
        From the definition of $\mathcal T_N$ and Proposition~\ref{prop:key-estimate} we obtain $ \Delta(\mathbf n)\geq g_N\geq {N}/({12(4L)^2}). $ These are precisely the two inequalities in \eqref{eq:higher-order-lower-bound}.
    \end{stepproof}
    
    \begin{step}[2]
        The system $(\Sigma,\nu,\sigma)$ is not exponentially mixing of order $L$.
    \end{step}
    
    \begin{stepproof}[Proof of Step 2]
        Fix $0<\theta<1$. The coordinate observable $\varphi(\omega):=\omega_0$ belongs to $\mathcal F_\theta(\Sigma)$ and has mean zero by Proposition~\ref{prop:limit}\itemref{it:limit-iii}. Since $\varphi\circ\sigma^{n_i}=Z_{n_i}$, we have
        $$
        \nu\left(\prod_{i=1} ^L\varphi\circ\sigma^{n_i}\right)-\nu(\varphi)^L=\E\left[\prod_{i=1}^L Z_{n_i}\right].
        $$
        
        If exponential mixing of order $L$ held, applying its defining estimate to the tuples from Step~1 would give
        \begin{align*}
            \frac{N^{-L/2}}{2^{L+1}3^{L/2}\bigl(6(4L)^2\bigr)^{L/2}} &\leq \left|\E\left[\prod_{i=1}^L Z_{n_i}\right]\right|\leq Ce^{-c\Delta(\mathbf n)}\|\varphi\|_\theta^L \leq Ce^{-cN/(12(4L)^2)}\|\varphi\|_\theta^L,
        \end{align*}
        with $C,c>0$ independent of $N$. Multiplying by $N^{L/2}$ gives a fixed positive number bounded above by a quantity tending to zero. This contradiction proves the assertion.
    \end{stepproof}
\end{proof}

\begin{remark}
Corollary~\ref{cor:higher-order} excludes exponential mixing at order $L$, not qualitative mixing at that order. The subintervals $C_i$ in Proposition~\ref{prop:key-estimate} force the successive gaps to be proportional to the observation scale, a property that does not follow from the moment deficit alone.
\end{remark}

\begin{proof}[Proof of Theorem~\ref{thm:symbolic}]
    Fix $R\geq2$, set $L:=2R+2$, and take the measure $\nu$ constructed in Proposition~\ref{prop:limit}.
    
    Proposition~\ref{prop:limit}\itemref{it:limit-ii}--\itemref{it:limit-iii} gives stationarity, full support and the coordinate distributions and tuplewise independence in item~\itemref{it:S-independence}. Propositions~\ref{prop:windows} and~\ref{prop:holder} give items~\itemref{it:S-windows} and~\itemref{it:S-correlations}, with $K_R=K$.
    
    Corollary~\ref{cor:moment} gives the moment bound in item~\itemref{it:S-moments}, with $\delta=\delta_L$ and $n_0=6(4L)^2$, and Corollary~\ref{cor:no-gaussian} excludes Gaussian weak subsequential limits. Finally, Corollary~\ref{cor:higher-order} gives item~\itemref{it:S-higher-order}, including failure of exponential mixing at order $L$. All properties hold for the same measure $\nu$, completing the proof.
\end{proof}

\addtocontents{toc}{\protect\addvspace{9pt}}
\section{Lean formalisation}\label{sec:lean-formalisation}
Theorems~\ref{thm:analytic}, \ref{thm:physical} and~\ref{thm:symbolic} have been formalised in Lean~4 using Mathlib. The complete source code is available at \url{https://github.com/Manzatto/finite-order-mixing-does-not-imply-clt}.

Each main theorem file contains its statement and proof. All declarations below belong to the namespace \texttt{FiniteOrderMixing}, and the paths are relative to \texttt{FiniteOrderMixing/}.
\begin{itemize}
\item Theorem~\ref{thm:analytic}: statement \texttt{TheoremA} and proof \texttt{theoremA}, in \texttt{TheoremA.lean}.
\item Theorem~\ref{thm:physical}: statement \texttt{TheoremB}
and proof \texttt{theoremB}, in \texttt{TheoremB.lean}.
\item Theorem~\ref{thm:symbolic}: statement \texttt{TheoremC} and proof \texttt{theoremC}, in \texttt{TheoremC.lean}.
\end{itemize}
The formalisation establishes all six assertions of Theorems~\ref{thm:analytic} and~\ref{thm:physical}, and all five assertions of Theorem~\ref{thm:symbolic}.

The external results cited in the paper and used in the formal proofs are also formalised in Lean, in the forms required by our argument rather than in their full published generality. Their hypotheses are verified at each application. For example, the Kolmogorov extension theorem is proved for a finite alphabet, which suffices to construct the symbolic process.

The project contains no \texttt{sorry} placeholders or additional axioms. The three main proofs depend only on Lean's standard axioms \texttt{propext}, \texttt{Classical.choice} and \texttt{Quot.sound}. With the versions pinned by \texttt{lean-toolchain} and \texttt{lake-manifest.json}, compilation and axiom dependencies can be checked from the repository root by running
\begin{verbatim}
lake build
lake env lean CheckAxioms.lean
\end{verbatim}
\vspace{1em}
\paragraph{\textbf{Declaration of AI use.}} OpenAI's ChatGPT 6.0 (Astra) was used as a research assistant during the development of this work.
AI assistance was also used to develop and revise the Lean formalisation of the main results. The resulting proofs were checked by Lean's kernel, with the axiom dependencies described above.
The authors take full responsibility for the correctness of all mathematical statements and proofs, the accuracy of the references, and the final content of the manuscript.

\section*{Acknowledgements}
BBC is supported by a Chapman Fellowship at Imperial College London. MMC is supported by the São Paulo Research Foundation (FAPESP, grant no. 2025/26997-9).

\bibliographystyle{plainnat}
\bibliography{refs}

\appendix 

\section{Smooth physical realisation of symbolic measures}
\label{app:smooth-realisation-short}

In this appendix, we prove Proposition~\ref{prop:cv-all-measures}, which realises every prescribed ergodic invariant measure of the full two-symbol shift as a physical measure of a $\mathcal{C}^{\infty}$ diffeomorphism on any compact smooth surface without boundary.

Colli and Vargas construct a $\mathcal{C}^{\infty}$ surface diffeomorphism with a wandering domain whose orbits are all generic for a particular invariant measure supported on a horseshoe \cite[Theorem~1]{ColliVargas2001}.
Kiriki, Li, Nakano, Soma and Vargas \cite[Theorem~A]{KLNSV} develop methods for realising any prescribed ergodic invariant measure of the full
two-symbol shift as a physical measure of a $\mathcal{C}^r$ surface diffeomorphism, for $2\leq r<\infty$ (see also \cite{KirikiSoma2017,KirikiNakanoSoma2024}). Our proof adapts these methods to a symmetric affine horseshoe.

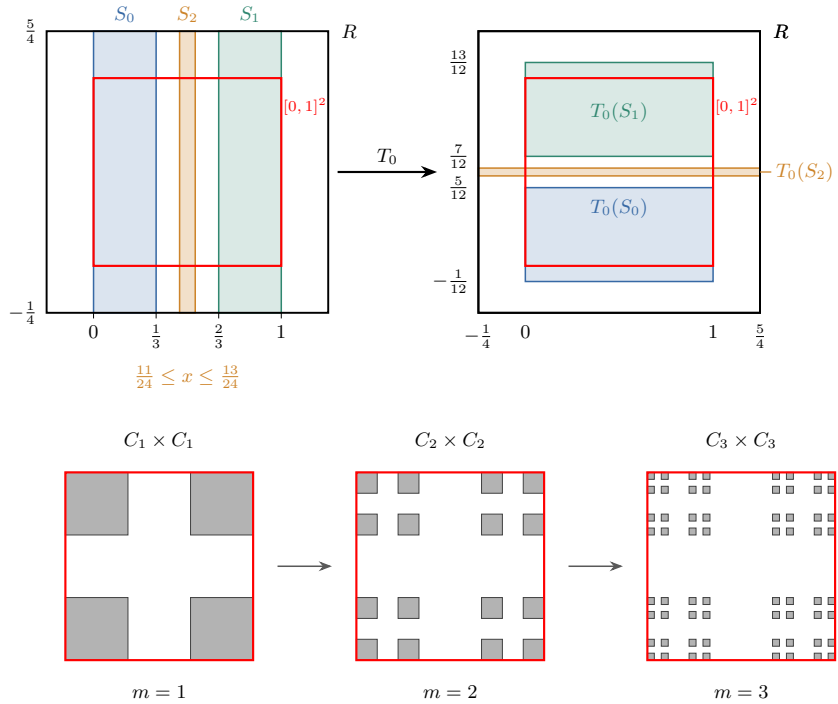
\begin{figure}[htbp]
\centering
\definecolor{tzerocol}{RGB}{55,105,165}
\definecolor{tonecol}{RGB}{48,135,112}
\definecolor{tgapcol}{RGB}{205,132,42}
\resizebox{.75\textwidth}{!}{%
\begin{tikzpicture}[
    x=3.05cm,y=3.05cm,
    >=Stealth,
    font=\small,
    zero strip/.style={draw=tzerocol,fill=tzerocol!18,line width=.7pt},
    one strip/.style={draw=tonecol,fill=tonecol!18,line width=.7pt},
    gap strip/.style={draw=tgapcol,fill=tgapcol!25,line width=.7pt},
    unit square/.style={draw=red,line width=1pt},
    square label/.style={font=\small,text=red,fill=white,
        inner sep=2pt,anchor=north east},
    common rectangle/.style={black,line width=.9pt},
    cantor cell/.style={fill=black!30,draw=black!75,line width=.25pt},
    tick label/.style={font=\small,inner sep=2pt}
]

\begin{scope}
    \path[zero strip] (0,-1/4) rectangle (1/3,5/4);
    \path[one strip] (2/3,-1/4) rectangle (1,5/4);
    \path[gap strip] (11/24,-1/4) rectangle (13/24,5/4);

    \draw[unit square] (0,0) rectangle (1,1);
\node[anchor=west,text=red,font=\scriptsize]
    at (0.97,.85) {$[0,1]^2$};
    \draw[common rectangle] (-1/4,-1/4) rectangle (5/4,5/4);
    \node[anchor=west] at (1.28,1.25) {$R$};

    \node[text=tzerocol] at (1/6,1.34) {$S_0$};
    \node[text=tgapcol] at (1/2,1.34) {$S_2$};
    \node[text=tonecol] at (5/6,1.34) {$S_1$};

    \foreach \x/\lab in {
        0/0,{1/3}/{\frac{1}{3}},{2/3}/{\frac{2}{3}},1/1
    }{
        \draw (\x,-1/4) -- (\x,-.28);
        \node[tick label,anchor=north] at (\x,-.29) {$\lab$};
    }
    \node[tick label,anchor=east] at (-.275,-1/4)
        {$-\frac{1}{4}$};
    \node[tick label,anchor=east] at (-.275,5/4)
        {$\frac{5}{4}$};
    \node[tick label,text=tgapcol] at (.5,-.59)
        {$\frac{11}{24}\leq x\leq\frac{13}{24}$};
\end{scope}

\draw[->,line width=1pt] (1.3,.5) --
    node[above] {$T_0$} (1.83,.5);

\begin{scope}[shift={(2.30,0)}]
    \path[zero strip] (0,-1/12) rectangle (1,5/12);
    \path[one strip] (0,7/12) rectangle (1,13/12);

    \path[gap strip] (-1/4,{1/2-1/48})
        rectangle (5/4,{1/2+1/48});

    \draw[unit square] (0,0) rectangle (1,1);
\node[anchor=west,text=red,font=\scriptsize]
    at (0.97,.85) {$[0,1]^2$};
    \draw[common rectangle] (-1/4,-1/4) rectangle (5/4,5/4);
    \node[anchor=west] at (1.28,1.25) {$R$};
    \draw[common rectangle] (-1/4,-1/4) rectangle (5/4,5/4);
    \node[anchor=west] at (1.28,1.25) {$R$};

    \node[text=tzerocol] at (.5,.31) {$T_0(S_0)$};
    \node[text=tonecol] at (.5,.82) {$T_0(S_1)$};

    \draw[tgapcol,line width=.6pt] (5/4,.5) -- (1.31,.5);
    \node[text=tgapcol,anchor=west,inner sep=2pt] at (1.31,.5)
        {$T_0(S_2)$};

    \foreach \y/\lab in {
        {-1/12}/{-\frac{1}{12}},
        {5/12}/{\frac{5}{12}},
        {7/12}/{\frac{7}{12}},
        {13/12}/{\frac{13}{12}}
    }{
        \node[tick label,anchor=east] at (-.275,\y) {$\lab$};
    }
    \foreach \x/\lab in {
        {-1/4}/{-\frac{1}{4}},0/0,1/1,{5/4}/{\frac{5}{4}}
    }{
        \node[tick label,anchor=north] at (\x,-.29) {$\lab$};
    };
\end{scope}

\newcommand{\CantorPanel}[4]{%
    \begin{scope}[shift={(#1,-2.10)}]
        \foreach \a in {#4}{
            \foreach \b in {#4}{
                \path[cantor cell]
                    ({\a/#3},{\b/#3}) rectangle
                    ({(\a+1)/#3},{(\b+1)/#3});
            }
        }
        \draw[unit square] (0,0) rectangle (1,1);
        \node at (.5,1.17) {$C_{#2}\times C_{#2}$};
        \node[tick label] at (.5,-.17) {$m=#2$};
    \end{scope}
}

\CantorPanel{-.15}{1}{3}{0,2}
\CantorPanel{1.40}{2}{9}{0,2,6,8}
\CantorPanel{2.95}{3}{27}{0,2,6,8,18,20,24,26}

\draw[->,black!65,line width=.7pt]
    (.98,-1.60) -- (1.27,-1.60);
\draw[->,black!65,line width=.7pt]
    (2.53,-1.60) -- (2.82,-1.60);

\end{tikzpicture}%
}
 \caption{The strips $S_0,S_1,S_2$ and their images under $T_0$ (top), together with the sets $\bigcap_{i=-m}^{m-1}T_0^{-i}(S_0\cup S_1\cap [0,1])=C_m\times C_m$ for $m=1,2,3$ (bottom). The black frame is $R=[-\frac14,\frac54]^2$ and the red frame is $[0,1]^2$. }
\label{fig:t0-strips}
\end{figure}

We first explain the idea of the construction. Fix a nonempty compact smooth surface $M$ without boundary.
Set $\gamma=1/225$ and consider the disjoint rectangles
$$
 \begin{gathered}
 S_0=\left[0,\frac{1}{3}\right]\times
     \left[-\frac{1}{4},\frac{5}{4}\right],\quad
 S_1=\left[\frac{2}{3},1\right]\times
     \left[-\frac{1}{4},\frac{5}{4}\right],\ \text{and } S_2=\left[\frac{11}{24},\frac{13}{24}\right]\times
     \left[-\frac{1}{4},\frac{5}{4}\right].
 \end{gathered}
$$
Choose a smooth diffeomorphism $T_0:\mathbb R^2 \to \mathbb R^2$ equal to the identity outside a compact set ($R$ in Figure~\ref{fig:t0-strips}), with the following expressions near $S_0,S_1$ and $S_2$
\begin{equation}
\label{eq:short-affine}
 T_0(x,y)=
\begin{cases}
  \displaystyle \left(3x,y/3\right),&(x,y)\in S_0,\\[3pt]
  \left(3x-2, (y+2)/{3}\right),&(x,y)\in S_1,\\[3pt]
  \left(1-y,\frac{1}{2}+\gamma\left(x-\frac{1}{2}\right)\right),
       &(x,y)\in S_2.
 \end{cases}
\end{equation}

Choose a chart $\chi:V\subset M\to\Omega\subset\mathbb R^2$, where $\Omega$
is a disc in which $[-2,2]^2$ and the closure of the support of $T_0$ are compactly contained. Define $F_0=\chi^{-1}\circ T_0\circ\chi$ on $V$ and extend it by
the identity outside $V$. Then $F_0:M\to M$ is a $\mathcal C^\infty $ diffeomorphism satisfying $F_0(V)=V$.

Let $C\subset [0,1]$ be the (classical) middle-thirds Cantor set and let
$\Sigma=\{-1,1\}^{\mathbb Z}$, with
$(\sigma\omega)_i=\omega_{i+1}$.
The two outer branches give the affine horseshoe
$\Lambda=\chi^{-1}(C\times C)$, with coding
\begin{align}
 \pi(\omega)=\chi^{-1}\left(
   \sum_{i\geq0}\frac{\omega_i+1}{3^{i+1}},
   \sum_{i\geq0}\frac{\omega_{-(i+1)} +1}{3^{i+1}}
 \right).    \label{eq:pi}
\end{align}
Thus $F_0(\Lambda)=\Lambda$ and $F_0\circ\pi=\pi\circ\sigma$. Moreover, $\Lambda$ is a uniformly hyperbolic set for $F_0$. The map $\pi$ is a homeomorphism, and the derivative on the horseshoe is $\operatorname{diag}(3,1/3)$ in these coordinates. For $0<\theta<1$, use $
 d_\theta(\omega,\omega')=
 \theta^{\min\{|i|:\omega_i\ne\omega'_i\}}$ and $\|u\|_\theta=\|u\|_\infty+\operatorname{Lip}_{d_\theta}(u).$

Let $\nu$ be the prescribed ergodic invariant measure. The conjugacy $\pi$ transfers the statistical properties of $(\Sigma,\sigma,\nu)$ to $(\Lambda,F_0,\mu)$, where $\mu=\pi_*\nu$. We choose a $\nu$-generic point $\omega$ and set $x=\pi(\omega)$, so that the time averages along the $F_0$-orbit of $x$ converge to the integrals against $\mu$.

We then perturb $F_0$ away from $\Lambda$ to obtain a
smooth diffeomorphism $T$ and a nonempty open set $W$.
Since $T=F_0$ near $\Lambda$, we have $T^i(x)=F_0^i(x)$ for every $i\geq0$. The orbits of points in $W$ under $T$ follow the orbit $(T^i(x))_{i\in\mathbb N}$ most of the time. Hence, they have the same statistics as the orbit of $x$ in the sense that, for every continuous observable, their time averages converge uniformly on $W$ to its integral against $\mu$. Thus $\mu$ is physical.

Geometrically, this resembles the figure-eight example
described by Young \cite[Section 2.2, Figure~1]{Young2002}.
Here the prescribed statistical behaviour comes from
a generic orbit in a horseshoe.

\begin{proof}[Proof of Proposition \ref{prop:cv-all-measures}]
Since $\nu$ is ergodic, Birkhoff's ergodic theorem allows us to
choose $\omega=(\omega_j)_{j\in\mathbb Z}\in\Sigma$ in the basin of $\nu$ under $\sigma$.
We use a fixed Riemannian distance $d$ on $M$.
Until the transfer to $M$, all constructions take place in
the coordinates of \eqref{eq:short-affine}. The sequence $\omega\in\Sigma$ is fixed throughout the proof. 

\begin{step}[1]
For every integer $N_0\geq3$, there exist a strictly increasing
sequence of positive integers $(n_k)_{k\geq0}$, sequences
of points $((x_k,y_k))_{k\geq0}$ in $[0,1]^2$ and
$((v_k,z_k))_{k\geq0}$ in $S_2$, and a sequence of real
numbers $(\varepsilon_k)_{k\geq0}$, such that, for every $k\geq0$,
\begin{itemize}
\item $T_0^j(x_k,y_k)\in S_0\cup S_1$ for $0\leq j<n_k$.
\item $T_0^{n_k}(x_k,y_k)=(v_k,z_k)$.
\item $T_0(v_k,z_k+\varepsilon_k)=(x_{k+1},y_{k+1})$.
\item $|\varepsilon_k|\leq3^{-N_0-\lfloor\sqrt{k+2}\rfloor}$.
\end{itemize}
\end{step}

\begin{stepproof}[Proof of Step 1]
Fix an integer $N_0\geq3$ and set
\begin{equation}
\label{eq:cv-block-lengths}
 h_0=0,\quad h_{k+1}=N_0+\lfloor\sqrt{k+2}\rfloor,\quad
 n_k=k+1+h_{k+1},\quad t_0=0,\quad t_{k+1}=t_k+n_k+1.
\end{equation}
It follows from \eqref{eq:cv-block-lengths} that
$h_{k+1}\leq n_k$, $1\leq n_{k+1}-n_k\leq2$
and $h_k=O(\sqrt{k})$.

We call a finite word
$\omega^{(k)}=(\omega^{(k)}_0,\ldots,\omega^{(k)}_{n_k-1})
\in\{-1,1\}^{n_k}$ a block of length $n_k$.
We will choose $(x_k,y_k)$ so that, for $0\leq j<n_k$,
$$
T_0^j(x_k,y_k)\in
\begin{cases}
S_0,&\omega^{(k)}_j=-1\\
S_1,&\omega^{(k)}_j=1
\end{cases},\quad\text{and}\quad T_0^{n_k}(x_k,y_k)\in S_2.$$
Define the words by $\omega^{(0)} \in \{-1,1\}^{n_0}$ with $\omega^{(0)}_j:=\omega_j$ for $0\leq j<n_0$. Moreover, we recursively define
\begin{equation}
\label{eq:cv-block-words}
 \omega^{(k+1)}_j:=
 \begin{cases}
  -\omega^{(k)}_{n_k-1-j},&0\leq j<h_{k+1},\\
  \omega_{t_{k+1}+j},&h_{k+1}\leq j<n_{k+1}.
 \end{cases}
\end{equation}
Observe that $\omega^{(k+1)}$ is obtained by reversing the order and changing the signs of the last $h_{k+1}$ symbols of $\omega^{(k)}$, then completing the block with symbols from $\omega$.

For $s\in\{-1,1\}$, write $\phi_s(u):=(u+1+s)/3$.
The maps $\phi_{-1}$ and $\phi_1$ contract $[0,1]$ onto
its left and right thirds, respectively. By \eqref{eq:short-affine},
\begin{equation}
\label{eq:cv-branch-form}
T_0(x,y)
=
\begin{cases}
\bigl(\phi_{-1}^{-1}(x),\phi_{-1}(y)\bigr),
    &(x,y)\in S_0,\\
\bigl(\phi_1^{-1}(x),\phi_1(y)\bigr),
    &(x,y)\in S_1.
\end{cases}
\end{equation}

Set
\begin{equation}
\label{eq:cv-coordinate-parameters}
 v_k=\frac{1}{2}+\frac{1}{48(k+1)},\quad
 y_0=\frac{1}{2}\quad\text{and}\quad
 y_{k+1}=\frac{1}{2}+\gamma\left(v_k-\frac{1}{2}\right),
\end{equation}
and define
\begin{align}
 x_k=\phi_{\omega^{(k)}_0}\circ\cdots
       \circ\phi_{\omega^{(k)}_{n_k-1}}(v_k),\  z_k=\phi_{\omega^{(k)}_{n_k-1}}\circ\cdots
       \circ\phi_{\omega^{(k)}_0}(y_k) \text{ and }\varepsilon_k&=1-z_k-x_{k+1}.
\label{eq:cv-block-coordinates}
\end{align}
Since $\phi_s([0,1])\subset[0,1]$,
\eqref{eq:cv-coordinate-parameters} and
\eqref{eq:cv-block-coordinates} give
$$
 (x_k,y_k)\in[0,1]^2\quad\text{and}\quad
 (v_k,z_k)\in\left[\frac{1}{2},\frac{25}{48}\right]\times[0,1]\subset S_2.
$$
The maps $\phi_{-1}$ and $\phi_1$ send $[0,1]$ into
$[0,1/3]$ and $[2/3,1]$, respectively.

We claim that, for
$0\leq j\leq n_k$,
\begin{equation}
\label{eq:cv-block-iterates}
 T_0^j(x_k,y_k)=
 \left(
 \phi_{\omega^{(k)}_j}\circ\cdots
 \circ\phi_{\omega^{(k)}_{n_k-1}}(v_k),
 \phi_{\omega^{(k)}_{j-1}}\circ\cdots
 \circ\phi_{\omega^{(k)}_0}(y_k)
 \right),
\end{equation}
where an empty composition should be understood as the identity.
For $j=0$, \eqref{eq:cv-block-iterates} follows from
the definition of $x_k$ in \eqref{eq:cv-block-coordinates}.

Suppose \eqref{eq:cv-block-iterates} holds for some $j<n_k$.
Since $v_k,y_k\in[0,1]$, the first coordinate of $T_0^j(x_k,y_k)$ lies in
$$
 \phi_{\omega^{(k)}_j}([0,1])=
 \begin{cases}
 \left[0,\frac{1}{3}\right],&\omega^{(k)}_j=-1,\\
 \left[\frac{2}{3},1\right],&\omega^{(k)}_j=1,
 \end{cases}
$$
and its second coordinate lies in $[0,1]$.  Hence, $T_0^j(x_k,y_k)\in S_0$ if
$\omega^{(k)}_j=-1$, and $T_0^j(x_k,y_k)\in S_1$ if
$\omega^{(k)}_j=1$. Using \eqref{eq:cv-branch-form} and
$\phi_{\omega^{(k)}_j}^{-1}\circ\phi_{\omega^{(k)}_j}
=\operatorname{id}$, we obtain
$$
 T_0^{j+1}(x_k,y_k)=
 \left(
 \phi_{\omega^{(k)}_{j+1}}\circ\cdots
 \circ\phi_{\omega^{(k)}_{n_k-1}}(v_k),
 \phi_{\omega^{(k)}_j}\circ\cdots
 \circ\phi_{\omega^{(k)}_0}(y_k)
 \right).
$$
This proves \eqref{eq:cv-block-iterates} by induction, together with
$T_0^j(x_k,y_k)\in S_0\cup S_1$ for $0\leq j<n_k$.
Taking $j=n_k$ in \eqref{eq:cv-block-iterates} and using
the definition of $z_k$ in \eqref{eq:cv-block-coordinates} gives
$$
 T_0^{n_k}(x_k,y_k)=
 \left(
 v_k,
 \phi_{\omega^{(k)}_{n_k-1}}\circ\cdots
 \circ\phi_{\omega^{(k)}_0}(y_k)
 \right)
 =(v_k,z_k).
$$

From \eqref{eq:cv-block-coordinates},
$z_k+\varepsilon_k=1-x_{k+1}\in[0,1]$.
Since $v_k\in[1/2,25/48]$ from \eqref{eq:cv-coordinate-parameters},
the point $(v_k,z_k+\varepsilon_k)$ lies in $S_2$.
Using \eqref{eq:short-affine},
\eqref{eq:cv-coordinate-parameters} and
\eqref{eq:cv-block-coordinates}, we obtain
$$
 T_0(v_k,z_k+\varepsilon_k)
 =
 \left(
 1-z_k-\varepsilon_k,
 \frac{1}{2}+\gamma\left(v_k-\frac{1}{2}\right)
 \right)
 =(x_{k+1},y_{k+1}),
$$
which proves the second identity.

Finally, the identity $1-\phi_s(u)=\phi_{-s}(1-u)$,
together with \eqref{eq:cv-block-words} and
\eqref{eq:cv-block-coordinates}, implies
$$
 1-z_k,\ x_{k+1}\in
 \left(
 \phi_{\omega^{(k+1)}_0}\circ\cdots
 \circ\phi_{\omega^{(k+1)}_{h_{k+1}-1}}
 \right)([0,1]).
$$
This interval has length $3^{-h_{k+1}}$.
Since $\varepsilon_k=1-z_k-x_{k+1}$ by
\eqref{eq:cv-block-coordinates}, we obtain, using
\eqref{eq:cv-block-lengths},
\begin{equation}
\label{eq:short-correction}
 |\varepsilon_k|\leq3^{-h_{k+1}}
 =3^{-N_0-\lfloor\sqrt{k+2}\rfloor}.
\end{equation}
\end{stepproof}

\begin{step}[2]
For $N_0$ sufficiently large, there exists a smooth
diffeomorphism $P$ of $\mathbb R^2$, equal to the identity
outside $\operatorname{int}S_2$, such that, for every
$k\geq0$, $P(x,y)=(x,y+\varepsilon_k)$
on a neighbourhood of $(v_k,z_k)$.
\end{step}

\begin{stepproof}[Proof of Step 2]
Define the constants
\begin{equation}
\label{eq:cv-cutoff-scales}
 \rho_k=\frac{1}{100(k+2)^2}.
\end{equation}
Define the closed squares
$$
 Q_k=(v_k,z_k)+[-\rho_k,\rho_k]^2\quad\text{and}
 \quad Q_k'=(v_k,z_k)+[-\rho_k/2,\rho_k/2]^2.
$$
They are centred at $(v_k,z_k)$ and have sides
$2\rho_k$ and $\rho_k$, respectively.
Choose a smooth function $\beta:\mathbb R^2\to[0,1]$,
equal to one on $[-1/2,1/2]^2$ and supported in $(-1,1)^2$.
Define
\begin{equation}
\label{eq:cv-perturbation}
 f(z)=\sum_{k\geq0}\varepsilon_k
       \beta\left(\frac{z-(v_k,z_k)}{\rho_k}\right)\quad \text{and}
 \quad P(x,y)=(x,y+f(x,y)).
\end{equation}
For each $k$, the term
$\varepsilon_k\beta((z-(v_k,z_k))/\rho_k)$ in the sum above is zero
outside $Q_k$.
Since $\beta(x)=1$, for each $x\in[-1/2,1/2]^2$, it follows that
$$
 \varepsilon_k\beta\left(\frac{z-(v_k,z_k)}{\rho_k}\right)
 =\varepsilon_k
$$
for every $z\in Q_k'$.
From \eqref{eq:cv-coordinate-parameters} and
\eqref{eq:cv-cutoff-scales}, the squares $Q_k$ are pairwise disjoint. Indeed, if $j<k$, then $v_k\leq v_{j+1}$, so
$$
 v_j-v_k\geq v_j-v_{j+1}
 =\frac{1}{48(j+1)(j+2)}.
$$
Also, $\rho_k\leq\rho_j$, giving
$$
 \rho_j+\rho_k\leq2\rho_j
 =\frac{1}{50(j+2)^2}
 <\frac{1}{48(j+1)(j+2)}.
$$
Comparing these estimates gives $v_k+\rho_k<v_j-\rho_j$, and therefore $Q_k$ and $Q_j$ are disjoint.
The definition of $v_k$ gives $1/2<v_k\leq25/48$.
Also, $z_k\in[0,1]$, because $y_k\in[0,1]$ and
each of the maps $\phi_{-1}$ and $\phi_1$ maps $[0,1]$ into itself.
Since $\rho_k\leq1/400$, it follows that  
$$Q_k\subset 
 \left[\frac{1}{2}-\frac{1}{400},
       \frac{25}{48}+\frac{1}{400}\right]
 \times
 \left[-\frac{1}{400},1+\frac{1}{400}\right]\ \text{for every }k\geq 0.
$$
Moreover, $\overline{Q_k}\subset \operatorname{int}S_2$. Thus $\bigcup_{k\geq0}Q_k$ is compactly contained in
$\operatorname{int}S_2$.

For each integer $m\geq0$, we differentiate the function
$z\mapsto\varepsilon_k\beta((z-(v_k,z_k))/\rho_k)$,
$m$ times with respect to $z=(x,y)$, keeping $k$ fixed.
The chain rule gives
$$
 D^m\left[
 \varepsilon_k\beta\left(\frac{z-(v_k,z_k)}{\rho_k}\right)
 \right]
 =
 \frac{\varepsilon_k}{\rho_k^{m}}
  D^m\beta\left(\frac{z-(v_k,z_k)}{\rho_k}\right).
$$
Thus, $$\left\|D^m\left[
 \varepsilon_k\beta\left(\frac{\cdot-(v_k,z_k)}{\rho_k}\right)
 \right]\right\|_\infty\leq  C_m|\varepsilon_k|\rho_k^{-m},\ \text{where }C_m=\|D^m\beta\|_\infty$$ 
 is finite and depends only on $\beta$ and $m$.
Using \eqref{eq:short-correction} and \eqref{eq:cv-cutoff-scales},
we obtain
\begin{equation}
\label{eq:cv-derivative-bound}
 C_m\sum_{k\geq0}|\varepsilon_k|\rho_k^{-m}
 \leq C_m100^m3^{-N_0}
 \sum_{k\geq0}(k+2)^{2m}3^{-\lfloor\sqrt{k+2}\rfloor}<\infty.
\end{equation}
Therefore the series defining $f$, and the series obtained
by differentiating its summands any number of times,
converge uniformly. Hence $f$ is smooth.
Taking $m=1$ in \eqref{eq:cv-derivative-bound} allows us
to fix $N_0$ sufficiently large that $\|Df\|_\infty<1/2$.
For each fixed $x$, the map $y\mapsto y+f(x,y)$ is then
strictly increasing and equals the identity for large
$|y|$. Thus $P$ is a smooth diffeomorphism.

For  $z\in Q_k'$, we have $
 \varepsilon_k\beta\left(( z-(v_k,z_k))/\rho_k\right)
 =\varepsilon_k. $ For $j\ne k$, the squares $Q_j$ and $Q_k$ are disjoint.
Since $Q_k'\subset Q_k$, every $z\in Q_k'$ lies outside
$Q_j$, implying $
 \varepsilon_j\beta\left((z-(v_j,z_j))/{\rho_j}\right)
 =0.$
Thus, for every $z\in Q_k'$,
$$
 \begin{aligned}
 f(z)
 &=\sum_{j\geq0}\varepsilon_j
   \beta\left(\frac{z-(v_j,z_j)}{\rho_j}\right)=\varepsilon_k
   \beta\left(\frac{z-(v_k,z_k)}{\rho_k}\right)
   +\sum_{\substack{j\geq0\\j\ne k}}\varepsilon_j
   \beta\left(\frac{z-(v_j,z_j)}{\rho_j}\right)=\varepsilon_k.
 \end{aligned}
$$
This implies
\begin{equation}
\label{eq:cv-local-translation}
 P(z)=z+(0,\varepsilon_k)
 \quad\text{for }z\in Q_k'.
\end{equation}
This proves the stated local translation property.
\end{stepproof}

\begin{step}[3]
Choose $N_0$ and $P$ as in Step~2 and let $G=T_0\circ P$.
There exists a sequence of closed rectangles $(R_k)_{k\geq0}$,
centred at $(x_k,y_k)$ and with nonempty interiors, such that
\begin{equation}
\label{eq:cv-rectangle-chain}
 G^{n_k+1}(R_k)=R_{k+1}
 \quad\text{for every }k\geq0.
\end{equation}
Moreover, for $0\leq j<n_k$, one has
$G^j(R_k)\subset S_0$ if $\omega^{(k)}_j=-1$, and
$G^j(R_k)\subset S_1$ if $\omega^{(k)}_j=1$.
\end{step}

\begin{stepproof}[Proof of Step 3]
Set $\delta=10^{-3}$ and define positive numbers
\begin{equation}
\label{eq:cv-rectangle-sizes}
 p_0=\delta3^{-n_0},\quad q_0=\frac{\delta}{45},\quad
 p_{k+1}=3^{-n_k}q_k\quad\text{and}\quad q_{k+1}=\gamma3^{n_k}p_k.
\end{equation}
Let $R_k := (x_k,y_k)+ [-p_k,p_k]\times[-q_k,q_k]$. 

Since $n_{k+1}-n_k\leq2$ by
\eqref{eq:cv-block-lengths} and $\gamma=1/225$,
\eqref{eq:cv-rectangle-sizes} gives
$$
 p_{k+1}=3^{-n_k}q_k
 \leq9\cdot3^{-n_{k+1}}q_k
 \quad\text{and}\quad
 q_{k+1}=\frac{3^{n_k}p_k}{225}.
$$
Starting from $p_0=\delta3^{-n_0}$ and $q_0=\delta/45$,
induction gives
\begin{equation}
\label{eq:cv-rectangle-decay}
 p_k\leq\delta3^{-n_k}5^{-k}
 \quad\text{and}\quad
 q_k\leq\frac{\delta}{45}5^{-k}.
\end{equation}

We prove the claimed inclusions by induction on $j$.
Assume that they hold for $0\leq i<j$.
Then $G^j(R_k)=T_0^j(R_k)$ is a rectangle centred at
$T_0^j(x_k,y_k)$, with horizontal side $2\cdot3^jp_k$.
By \eqref{eq:cv-block-iterates},
$$
 \bigl(T_0^j(x_k,y_k)\bigr)_1
 =\phi_{\omega^{(k)}_j}\circ\cdots\circ
 \phi_{\omega^{(k)}_{n_k-1}}(v_k),
$$
where $(x,y)_1:=x.$ Since each $\phi_s$ is increasing with derivative $1/3$,
and $\min\{v_k,1-v_k\}\geq23/48$, if
$\omega^{(k)}_j=-1$, then
$$
 \frac{23}{48}3^{-(n_k-j)}
 \leq\bigl(T_0^j(x_k,y_k)\bigr)_1
 \leq\frac{1}{3}-\frac{23}{48}3^{-(n_k-j)}.
$$
If $\omega^{(k)}_j=1$, then
$$
 \frac{2}{3}+\frac{23}{48}3^{-(n_k-j)}
 \leq\bigl(T_0^j(x_k,y_k)\bigr)_1
 \leq1-\frac{23}{48}3^{-(n_k-j)}.
$$
From \eqref{eq:cv-rectangle-decay},
$$
 2\cdot3^jp_k
 \leq2\delta5^{-k}3^{-(n_k-j)}
 <\frac{23}{48}3^{-(n_k-j)}.
$$
Hence the horizontal projection of $G^j(R_k)$ lies
in $[0,1/3]$ if $\omega^{(k)}_j=-1$, and in $[2/3,1]$
if $\omega^{(k)}_j=1$.

For the vertical coordinates,
\eqref{eq:cv-coordinate-parameters} and
\eqref{eq:cv-rectangle-decay} give
$$
 [y_k-q_k,y_k+q_k]
 \subset
 \left[\frac{1}{2}-\frac{\delta}{45},
       \frac{1}{2}+\frac{\gamma}{48}+\frac{\delta}{45}\right]
 \subset[0,1].
$$
Moreover,
$$
 \phi_{-1}([0,1])=\left[0,\frac{1}{3}\right]
 \quad\text{and}\quad
 \phi_1([0,1])=\left[\frac{2}{3},1\right].
$$
Thus the vertical projection of $G^j(R_k)$ remains in $[0,1]$.
Therefore, $G^j(R_k)\subset S_0$ if
$\omega^{(k)}_j=-1$, and $G^j(R_k)\subset S_1$ if
$\omega^{(k)}_j=1$, completing the induction.

The inclusions proved above show that all points of $R_k$
follow the same sequence of branches in $S_0$ and $S_1$
for the first $n_k$ iterates. Hence
$G^{n_k}(R_k)=T_0^{n_k}(R_k)$.
Taking $j=n_k$ in \eqref{eq:cv-block-iterates} gives
$$
 G^{n_k}(x_k,y_k)=T_0^{n_k}(x_k,y_k)=(v_k,z_k).
$$
On both $S_0$ and $S_1$, the affine formula for $G$ has
linear part $\operatorname{diag}(3,1/3)$.
Thus, for $|u|\leq p_k$ and $|w|\leq q_k$,
$$
 \begin{aligned}
 G^{n_k}(x_k+u,y_k+w)
 &=G^{n_k}(x_k,y_k)+(3^{n_k}u,3^{-n_k}w)=(v_k+3^{n_k}u,z_k+3^{-n_k}w).
 \end{aligned}
$$
Since $R_k=(x_k,y_k)+[-p_k,p_k]\times[-q_k,q_k]$,
letting $u$ and $w$ vary over these intervals gives
$$
 G^{n_k}(R_k)
 =(v_k,z_k)
 +[-3^{n_k}p_k,3^{n_k}p_k]
 \times[-3^{-n_k}q_k,3^{-n_k}q_k].
$$
By \eqref{eq:cv-rectangle-decay}, its sides satisfy
$$
 2\cdot3^{n_k}p_k\leq2\delta5^{-k}
 \quad\text{and}\quad
 2\cdot3^{-n_k}q_k
 \leq\frac{2\delta}{45}3^{-n_k}5^{-k}
 \leq2\delta5^{-k}.
$$
Since $\delta=10^{-3}$ and $(k+2)^2\leq5^{k+1}$,
$$
 2\cdot \delta\cdot 5^{-k}
 =\frac{1}{500\cdot5^k}
 \leq\frac{1}{100(k+2)^2}
 =\rho_k.
$$
Thus $ G^{n_k}(R_k)
 \subset(v_k,z_k)+[-\rho_k/2,\rho_k/2]^2
 =Q_k'.$ From \eqref{eq:cv-local-translation}, $P$ translates this
rectangle by $(0,\varepsilon_k)$. Its centre is therefore $
 (v_k,z_k+\varepsilon_k)=(v_k,1-x_{k+1}),$ where we used $\varepsilon_k=1-z_k-x_{k+1}$.
Thus
$$
 P\bigl(G^{n_k}(R_k)\bigr)
 =(v_k,1-x_{k+1})
 +[-3^{n_k}p_k,3^{n_k}p_k]
 \times[-3^{-n_k}q_k,3^{-n_k}q_k].
$$
Since $(v_k,1-x_{k+1})\in[1/2,25/48]\times[0,1],$
and both sides of this rectangle are at most $2\delta$,
we obtain
$$
 P\bigl(G^{n_k}(R_k)\bigr)
 \subset
 \left[\frac{1}{2}-\delta,\frac{25}{48}+\delta\right]
 \times[-\delta,1+\delta]
 \subset S_2.
$$
We can therefore apply the formula for $T_0$ on $S_2$,
which sends $(v_k,1-x_{k+1})$, the centre of $P(G^{n_k}(R_k))$, to
$$
 T_0(v_k,1-x_{k+1})
 =\left(x_{k+1},
        \frac{1}{2}+\gamma\left(v_k-\frac{1}{2}\right)\right)
 =(x_{k+1},y_{k+1}).
$$
The first coordinate of $T_0(x,y)$ is $1-y$, and the
second coordinate is $1/2+\gamma(x-1/2)$. Thus
$$
 \begin{aligned}
 G^{n_k+1}(R_k)
 &=T_0\bigl(P(G^{n_k}(R_k))\bigr)=(x_{k+1},y_{k+1})
   +[-3^{-n_k}q_k,3^{-n_k}q_k]
   \times[-\gamma3^{n_k}p_k,\gamma3^{n_k}p_k]\\
 &=(x_{k+1},y_{k+1})
   +[-p_{k+1},p_{k+1}]\times[-q_{k+1},q_{k+1}]=R_{k+1},
 \end{aligned}
$$
where we used \eqref{eq:cv-rectangle-sizes}.
Hence \eqref{eq:cv-rectangle-chain} holds.
\end{stepproof}

\begin{step}[4]
There exists a $\mathcal C^\infty $ diffeomorphism $T:M\to M$, equal to
$F_0$ near $\Lambda$, and a nonempty open set $W\subset M$
such that
\begin{equation}
\label{eq:cv-orbit-following}
 \lim_{N\to\infty}\frac{1}{N}\sum_{j=0}^{N-1}
 \sup_{y\in W}d(T^jy,\pi(\sigma^j\omega))=0.
\end{equation}
\end{step}

\begin{stepproof}[Proof of Step 4]
Define $T$ on $M$ by
$$
 T(x)=
 \begin{cases}
  \chi^{-1}\circ G\circ\chi(x),&x\in V,\\
  x,&x\in M\setminus V.
 \end{cases}
$$ Since $G$ equals the identity outside
a compact subset of $\Omega$, this defines a $\mathcal C^\infty$ diffeomorphism $T:M\to M$.
The equality $G=T_0$ near $C\times C$ gives $T=F_0$
near $\Lambda$. Define
$$
 W:=\chi^{-1}(\operatorname{int}R_0),
$$
which is nonempty and open.

Fix an integer $q\geq1$ and let $j=t_k+\ell$, where $
 h_k+q\leq\ell<n_k-q.$
For every integer $s$ with $-q\leq s\leq q$, we have $ h_k\leq\ell+s<n_k.$ Thus \eqref{eq:cv-block-words} gives
$$
 \omega^{(k)}_{\ell+s}
 =\omega_{t_k+\ell+s}
 =\omega_{j+s}.
$$
Now let $x\in W$. Since $\chi(x)\in R_0$,
\eqref{eq:cv-rectangle-chain} implies $
 \chi(T^{t_k}x)=G^{t_k}(\chi(x))\in R_k$ and therefore $\chi(T^{j+s}x)
 =G^{\ell+s}\bigl(\chi(T^{t_k}x)\bigr)
 \in G^{\ell+s}(R_k).$ From Step~3,
$$
 \chi(T^{j+s}x)\in S_0\ \text{if }\omega^{(k)}_{\ell+s}=\omega_{j+s}=-1\quad\text{and}\quad
 \chi(T^{j+s}x)\in S_1\ \text{if }\omega^{(k)}_{\ell+s}=\omega_{j+s}=1.
$$

For $-q\leq s\leq q$, write
$$
 \chi(T^{j+s}x)=(u_s,v_s)
 \quad\text{and}\quad
 \chi(\pi(\sigma^{j+s}\omega))=(\bar u_s,\bar v_s).
$$
Since $(u_s,v_s)\in G^{\ell+s}(R_k)\subset[0,1]^2$ by
Step~3, we have $u_s,v_s\in[0,1]$.
Also, $\bar u_s,\bar v_s\in[0,1]$, because
$\chi(\Lambda)=C\times C\subset[0,1]^2$.
By \eqref{eq:pi}, $(\bar u_s,\bar v_s)$ belongs to
$S_0$ if $\omega_{j+s}=-1$, and to $S_1$ if
$\omega_{j+s}=1$.
Since $G=T_0$ on $S_0\cup S_1$, subtracting the
coordinate formulas in \eqref{eq:short-affine} for
$G(u_s,v_s)$ and $G(\bar u_s,\bar v_s)$ gives,
for $-q\leq s<q$,
$$
 u_{s+1}-\bar u_{s+1}=3(u_s-\bar u_s)
 \quad\text{and}\quad
 v_{s+1}-\bar v_{s+1}=\frac{v_s-\bar v_s}{3}.
$$
Iterating these identities gives
$$
 \begin{aligned}
 |u_0-\bar u_0|
 &=3^{-q}|u_q-\bar u_q|\leq3^{-q}\quad\text{and}\quad |v_0-\bar v_0|=3^{-q}|v_{-q}-\bar v_{-q}|\leq3^{-q}.
 \end{aligned}
$$
Hence $\|\chi(T^jx)-\chi(\pi(\sigma^j\omega))\|_{\mathbb R^2}
 \leq\sqrt{2}\,3^{-q}.$
Since $\chi^{-1}$ is Lipschitz on $[0,1]^2$ and $x\in W$
was arbitrary, we obtain
\begin{equation}
\label{eq:cv-cylinder-distance}
 \sup_{y\in W}d(T^jy,\pi(\sigma^j\omega))\leq C3^{-q},
\end{equation}
where $C$ is independent of $k$, $q$ and $j$.
Fix $q\geq1$ and, for each $N\geq1$, let $E_N$ be the
set of times $j=t_k+\ell\in\{0,\ldots,N-1\}$, with
$0\leq\ell\leq n_k$, for which $\ell<h_k+q$ or
$\ell\geq n_k-q$.
Thus \eqref{eq:cv-cylinder-distance} holds for every
$j\in\{0,\ldots,N-1\}\setminus E_N$.
We now show that $\#E_N/N\to0$ and deduce
\eqref{eq:cv-orbit-following}.

Each interval of times $t_k,\ldots,t_k+n_k$ contains
at most $h_k+q$ excluded times at its beginning,
$q$ at its end, and the return time $t_k+n_k$.
Hence it contains at most $h_k+2q+1$ excluded times.
Choose $K$ such that $t_K\leq N<t_{K+1}$.
Counting the completed intervals $0,\ldots,K-1$
and bounding the contribution from the unfinished interval
by $n_K+1$ gives
$$
 \#E_N
 \leq\sum_{k=0}^{K-1}(h_k+2q+1)+(n_K+1)
 \leq K(N_0+\sqrt{K}+2q+1)+(n_K+1).
$$
By \eqref{eq:cv-block-lengths},
$$
 t_K=\sum_{k=0}^{K-1}(n_k+1)\geq\frac{K(K+1)}{2}
 \quad\text{and}\quad
 n_K+1=K+2+N_0+\lfloor\sqrt{K+2}\rfloor.
$$
Since $N\geq t_K$, for $K\geq1$ we obtain
$$
 \frac{\#E_N}{N}
 \leq
 \frac{2(N_0+\sqrt{K}+2q+1)}{K+1}
 +\frac{2(n_K+1)}{K(K+1)}
 \longrightarrow0
 \quad\text{as }N\to\infty,
$$
with $q$ fixed. For $j\notin E_N$, \eqref{eq:cv-cylinder-distance} bounds
the distance by $C3^{-q}$. For $j\in E_N$, the distance
is at most $\operatorname{diam}(M)$. Therefore
$$
 \begin{aligned}
 \frac{1}{N}\sum_{j=0}^{N-1}
   \sup_{y\in W}d(T^jy,\pi(\sigma^j\omega))&=
 \frac{1}{N}\sum_{\substack{0\leq j<N\\j\notin E_N}}
   \sup_{y\in W}d(T^jy,\pi(\sigma^j\omega))
 +\frac{1}{N}\sum_{j\in E_N}
   \sup_{y\in W}d(T^jy,\pi(\sigma^j\omega))\\
 &\quad\leq
 C3^{-q}\frac{N-\#E_N}{N}
 +\operatorname{diam}(M)\frac{\#E_N}{N}.
 \end{aligned}
$$
Taking the limit superior as $N\to\infty$ gives
$$
 \limsup_{N\to\infty}\frac{1}{N}\sum_{j=0}^{N-1}
 \sup_{y\in W}d(T^jy,\pi(\sigma^j\omega))
 \leq C3^{-q}.
$$
Letting $q\to\infty$ proves \eqref{eq:cv-orbit-following}.
The final averaging argument follows \cite[proof of Theorem~A]{KirikiNakanoSoma2024}.
 \end{stepproof}

\begin{step}[5]
The measure $\mu=\pi_*\nu$ is $T$-invariant and, for every
$v\in\mathcal C^0(M)$,
$$
 \lim_{N\to\infty}\sup_{y\in W}
 \left|\frac{1}{N}\sum_{j=0}^{N-1}v(T^jy)-\int v\,d\mu\right|=0.
$$
In particular, $\mu$ is physical and $W$ is contained
in its basin. Moreover, $T(\Lambda)=\Lambda$,
$T\circ\pi=\pi\circ\sigma$, and there exists $C_1<\infty$
such that $\|v\circ\pi\|_{1/3}\leq C_1\|v\|_{\mathcal C^1(M)}$ for every $v\in\mathcal C^1(M).$
\end{step}

\begin{stepproof}[Proof of Step 5]
By Step~4, $T=F_0$ near $\Lambda$. Hence
$T(\Lambda)=\Lambda$ and $T\circ\pi=\pi\circ\sigma$.
Since $\nu$ is $\sigma$-invariant, $\mu=\pi_*\nu$ is
$T$-invariant.

Since $\omega$ lies in the basin of $\nu$, for every
$v\in\mathcal C^0(M)$,
\begin{equation}
\label{eq:cv-generic-averages}
 \frac{1}{N}\sum_{j=0}^{N-1}v(\pi(\sigma^j\omega))
 \xrightarrow[]{N\to\infty} \int v\,d\mu.
\end{equation}
For Lipschitz $v$, \eqref{eq:cv-orbit-following} transfers
the convergence in \eqref{eq:cv-generic-averages}
uniformly to $W$. Uniform approximation by Lipschitz
functions gives the same conclusion for every continuous
$v$, proving \eqref{eq:cv-uniform-physical}.
The set $W$ is nonempty and open, hence has positive
smooth volume. Thus $\mu$ is physical.

Finally, the coding of $\Lambda$ and the affine formulas
in \eqref{eq:short-affine} give
\(
 d\bigl(\pi(\omega),\pi(\omega')\bigr)
 \leq C d_{1/3}(\omega,\omega')\)
for all \(\omega,\omega'\in\Sigma\).
Composing with $v\in\mathcal C^1(M)$ gives
the stated norm bound for a constant $C_1$ independent of $v$.
\end{stepproof}
\end{proof}

\end{document}